\documentclass[11pt,a4paper]{amsart}
\usepackage[english]{babel}
\usepackage[utf8]{inputenc}
\usepackage[T1]{fontenc}
\usepackage{csquotes} 
\usepackage{lmodern}
\usepackage{amssymb}
\usepackage{float}
\usepackage{amscd}
\usepackage{amsthm}
 \usepackage[top=1.7cm, bottom=1.7cm, left=1.7cm, right=1.7cm,twoside=false]{geometry}
\usepackage{array}
\usepackage{longtable}
\usepackage{graphicx}
\usepackage{url} 
\usepackage{wrapfig}
\usepackage{color}
\usepackage[usenames,x11names]{xcolor}
\usepackage[all]{xy}
\usepackage{enumitem}
\setlist[1]{leftmargin=*}
\setlist[enumerate,1]{label=(\alph*)}
\setlist[enumerate,2]{label=\normalfont{(\roman*)}, ref=\normalfont{(\alph{enumi}.\roman*)}}

\usepackage{multirow}
\usepackage{units} 
\usepackage{yfonts}
\usepackage{mathrsfs}
\usepackage{caption}
\usepackage{subcaption}
\usepackage[export]{adjustbox}
\usepackage[hypertexnames=false]{hyperref} 
\hypersetup{linkcolor  =DodgerBlue3, citecolor  = teal, urlcolor   = teal, colorlinks = true, hyperfootnotes =false}

\usepackage[normalem]{ulem}

\usepackage[scr=boondoxo]{mathalfa}

\renewcommand{\ll}{\mathscr{l}}

\newcommand{\pp}{\mathscr{p}}
\newcommand{\qq}{\mathscr{q}}
\newcommand{\cc}{\mathscr{c}}
\newcommand{\kk}{\mathscr{k}}

\newcommand{\etop}{e_{\mathrm{top}}}

\renewcommand{\to}{\longrightarrow}
\newcommand{\map}{\dashrightarrow}

\def\Spec{\operatorname{Spec}}

\def\Aut{\operatorname{Aut}}
\def\Sing{\operatorname{Sing}}

\def\Pic{\operatorname{Pic}}

\newcommand{\R}{\mathbb{R}}
\newcommand{\F}{\mathbb{F}}

\newcommand{\C}{\mathbb{C}}
\newcommand{\Z}{\mathbb{Z}}
\newcommand{\Q}{\mathbb{Q}}
\renewcommand{\P}{\mathbb{P}}

\renewcommand{\O}{\mathcal{O}}

\renewcommand{\tilde}{\widetilde}
\renewcommand{\hat}{\widehat}
\renewcommand{\epsilon}{\varepsilon}
\renewcommand{\phi}{\varphi}
\renewcommand{\theta}{\vartheta}
\renewcommand{\d}{\partial}
\newcommand{\id}{\mathrm{id}}

\newcommand{\de}{:=}

\newcommand{\toin}[1]{\overset{#1}{\to}}

\newcommand{\invlim}{\varprojlim}

\newcommand{\NS}{\operatorname{NS}}
\newcommand{\Exc}{\operatorname{Exc}}
\newcommand{\Supp}{\operatorname{Supp}}
\newcommand{\Bk}{\operatorname{Bk}}
\newcommand{\redd}{_{\mathrm{red}}}

\renewcommand{\leq}{\leqslant}
\renewcommand{\geq}{\geqslant}

\newcommand{\hor}{_{\mathrm{hor}}}
\renewcommand{\vert}{_{\mathrm{vert}}}

\newcommand{\cS}{\mathcal{S}}
\newcommand{\bS}{\mathbb{S}}
\newcommand{\bD}{\mathbb{D}}
\newcommand{\bT}{\mathbb{T}}

\newcommand{\htp}{\simeq_{\mathrm{htp}}}
\newcommand{\diff}{\simeq_{\mathrm{diffeo}}}
\newcommand{\lk}{\ell k}

\newcommand{\AAut}{\mathrm{AAut}_{\mathrm{hol}}}

\newcommand{\lc}{(\!(}
\newcommand{\rc}{)\!)}

\theoremstyle{plain}
\newtheorem{thm}{Theorem}[section]
\newtheorem*{thm*}{Theorem}

\theoremstyle{definition}
\newtheorem{dfn}[thm]{Definition}
\newtheorem{lem}[thm]{Lemma}
\newtheorem{prop}[thm]{Proposition}
\newtheorem{cor}[thm]{Corollary}

\newtheorem{notation}[thm]{Notation}
\newtheorem{rem}[thm]{Remark}
\newtheorem{constr}[thm]{Construction} 
\newtheorem*{question}{Question} 

\theoremstyle{remark}

\newtheorem*{claim*}{Claim}
\newtheorem*{rem*}{Remark} 

\newcommand{\Bs}{\operatorname{Bs}}

\makeatletter \def\subsection{\@startsection{subsection}{3}
	\z@{.5\linespacing\@plus.7\linespacing}{.5\linespacing}
	{\bfseries\itshape}} \makeatother

\makeatletter \renewenvironment{proof}[1][\proofname]{
	\par\pushQED{\qed}\normalfont
	\topsep6\p@\@plus6\p@\relax
	\trivlist\item[\hskip\labelsep\bfseries#1\@addpunct{.}]
	\ignorespaces}{
	\popQED\endtrivlist\@endpefalse} \makeatother

\def\:{\colon}
\numberwithin{equation}{section}

\renewcommand{\bar}{\overline}

\begin{document}

\title[Factorial affine surfaces with $\kappa=0$ and trivial units]{Classification of smooth factorial affine surfaces \\ of Kodaira dimension zero with trivial units}
\author{Tomasz Pełka}
\address{Basque Center for Applied Mathematics, Alameda de Mazarredo 14, 48009 Bilbao, Spain}
\email{tpelka@bcamath.org}
\author{Paweł Raźny}
\address{Instytut Matematyki, Uniwersytet Jagielloński, ul. Łojasiewicza 6, 30-348 Kraków, Poland}
\email{pawel.razny@uj.edu.pl}
\begin{abstract}
	We give a corrected statement of \cite[Theorem 2]{GM_k<2}, which classifies smooth affine surfaces of Kodaira dimension zero, whose coordinate ring is factorial and has trivial units. Denote the class of such surfaces by $\cS_{0}$. An infinite series of surfaces in $\cS_{0}$, not listed in \cite{GM_k<2}, was recently obtained by Freudenburg, Kojima and Nagamine  \cite{FKN_k0} as affine modifications of the plane. We complete their list to a series containing arbitrarily high-dimensional families of pairwise non-isomorphic surfaces in $\cS_{0}$. Moreover, we classify them up to a diffeomorphism, showing that each occurs as an interior of a $4$-manifold whose boundary is an exceptional surgery on a $2$-bridge knot \cite{BW_surgery}. In particular, we show that $\cS_{0}$ contains countably many pairwise non-homeomorphic surfaces.
\end{abstract}

\maketitle

We work with complex algebraic varieties.
\section{Introduction}

The structure of smooth affine surfaces of non-general type is essentially understood, see \cite[Chapter III]{Miyan-OpenSurf}. It can be summarized in a following way, which parallels the classical Enriques-Kodaira classification of smooth projective surfaces. Let $S$ be a smooth affine surface and let $\kappa\in \{-\infty,0,1,2\}$ be its Kodaira--Iitaka dimension. If $\kappa=-\infty$ then $S$ admits a $\C^{1}$-fibration \cite[III.1.3.2]{Miyan-OpenSurf}. If $\kappa=1$ then $S$ has a (canonical) $\C^{*}$-fibration \cite[III.1.7.1]{Miyan-OpenSurf}. The surfaces with $\kappa=0$ are peculiar, but understandable: they are obtained in a controlled way from some specific minimal models, see \cite[II.6.4]{Miyan-OpenSurf}.  

Let $\cS_{\kappa}$ be the class of smooth affine surfaces $S$ of Kodaira--Iitaka dimension $\kappa$, such that $\C[S]^{*}=\C^{*}$ and $\C[S]$ is factorial, i.e.\ all line bundles on $S$ are trivial. It is known \cite[III.2.2.1]{Miyan-OpenSurf} that $S\in \cS_{-\infty}$ if and only if $S\cong \C^{2}$. The above description of affine surfaces with $\kappa<2$ suggests that it is the class $\cS_{0}$ where one could look for unusual examples.

In \cite[Theorem 2]{GM_k<2}, Gurjar and Miyanishi claimed that, up to an isomorphism, there are at most two surfaces in $\cS_{0}$. This result, however, was recently proved wrong by Freudenburg, Kojima and Nagamine \cite{FKN_k0}, who constructed an infinite series of pairwise non-isomorphic surfaces in $\cS_{0}$. 

The aim of this article is to complete the description of $\cS_{0}$ by proving the following theorem.

\begin{thm}\label{THM}
Let $\cS_{0}$ be the class of smooth affine surfaces $S$ such that the ring $\C[S]$ is factorial, $\C[S]^{*}=\C^{*}$ and the logarithmic Kodaira--Iitaka dimension of $S$ is zero. Then $S\in \cS_{0}$ if and only if
\begin{equation*}
S\cong S_{p_1,p_2}\de \Spec \C[x_1,x_2][(x_2x_1^{-\deg p_1}-p_1(x_1^{-1}))x_1^{-1},(x_1x_2^{-\deg p_2}-p_2(x_2^{-1}))x_2^{-1}]
\end{equation*}
for some monic polynomials $p_1,p_2\in \C[t]$. Moreover, $S_{p_1,p_2}\not \cong S_{p_1',p_2'}$ for $\{p_1,p_2\}\neq \{p_1',p_2'\}$.
\end{thm}

\begin{rem}\label{rem:literature}
	\begin{enumerate}
		\item\label{item:GM} The two surfaces constructed in \cite[Theorem 2]{GM_k<2} are both isomorphic to $S_{1,1}$, see \cite[Theorem 4.4]{FKN_k0}.
		This surface was extensively studied in \cite[\S 8]{Kal-Ku1_AL-theory}.
		\item\label{item:FKN} For $n\geq 1$, the surface $V_{n}=\Spec \C[x,y][(x-1)/(x^{n}y-1)]$ constructed in \cite{FKN_k0} is isomorphic to $S_{t^{n-1},1}$ via $\C[x_1,x_2][(x_2-1)x_1^{-n},(x_1-1)x_2^{-1}]\ni (x_1,x_2)\mapsto (x,1-x^{n}y)\in \C[x,y][(x-1)/(x^{n}y-1)]$.
	\end{enumerate}
\end{rem}

\begin{figure}[htbp]
	\begin{tabular}{ccc}
		\begin{subfigure}[t]{0.35\textwidth}
			\includegraphics[scale=0.25]{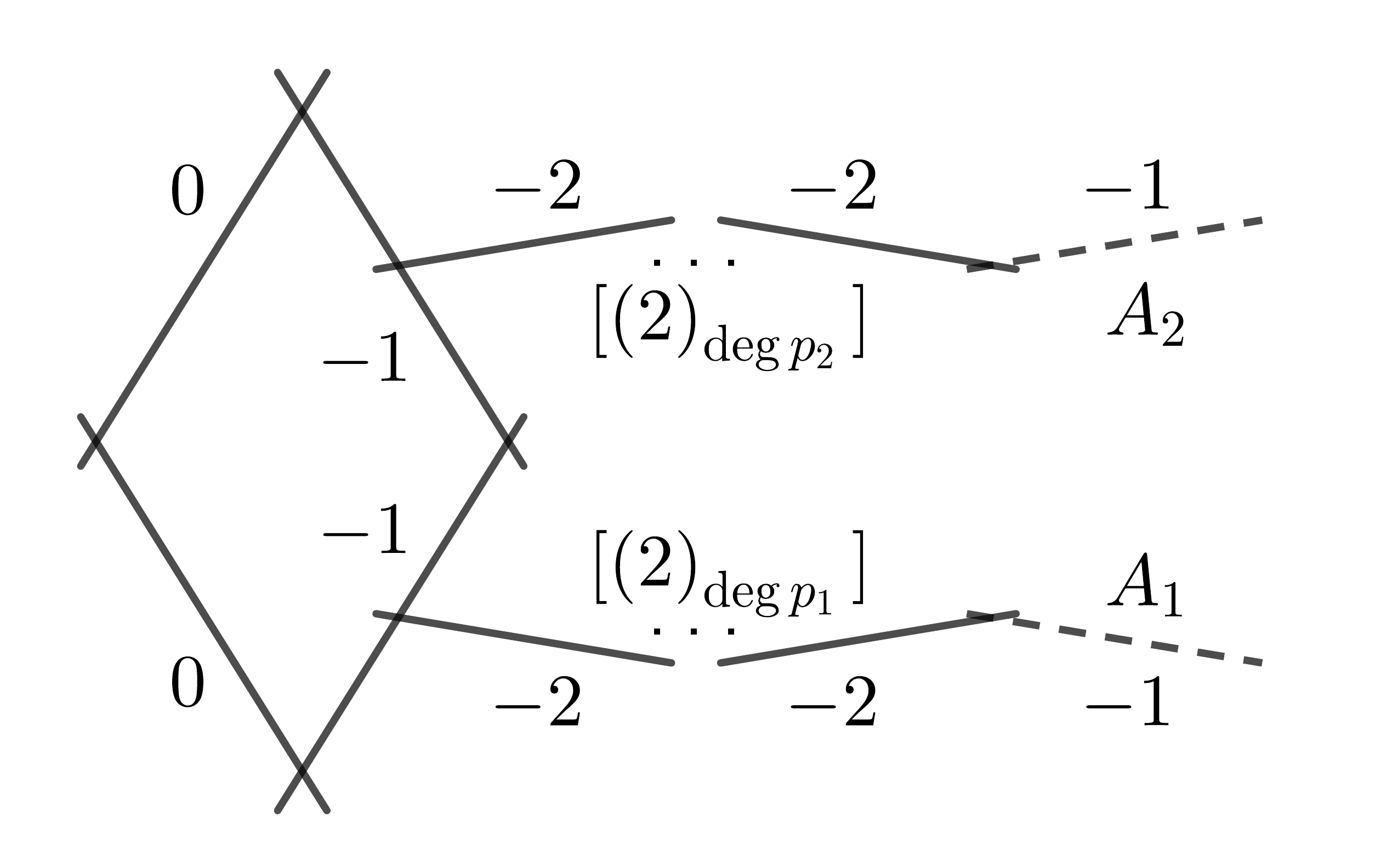}
			\caption{$p_1,p_2\neq 1$}
			\label{fig:Skl}
		\end{subfigure}
		&
		\begin{subfigure}[t]{0.35\textwidth}
			\includegraphics[scale=0.25]{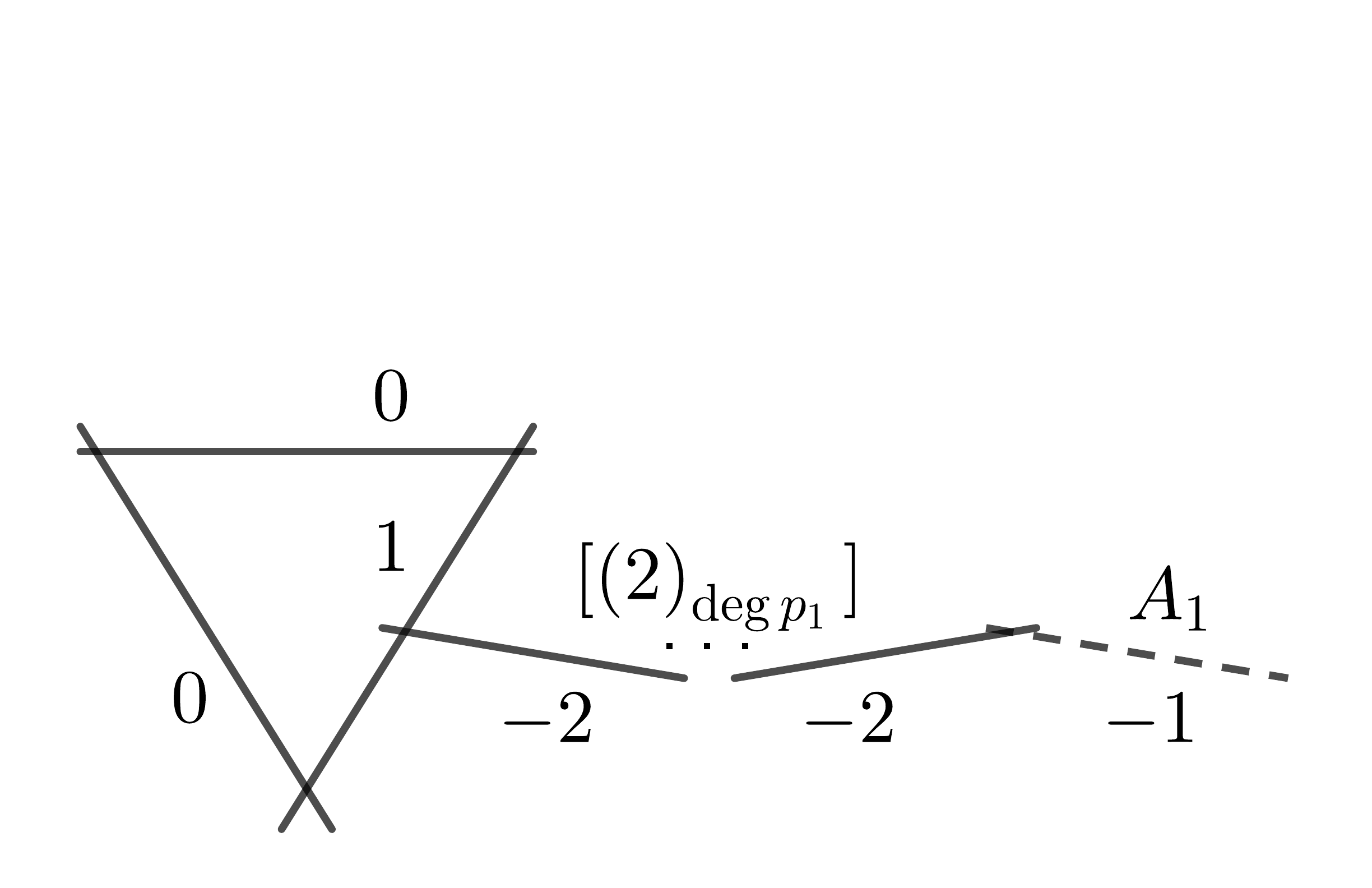}
			\caption{$p_1\neq 1$, $p_2=1$}
			\label{fig:Sk0}
		\end{subfigure}
		&
		\begin{subfigure}[t]{0.15\textwidth}
			\includegraphics[scale=0.25]{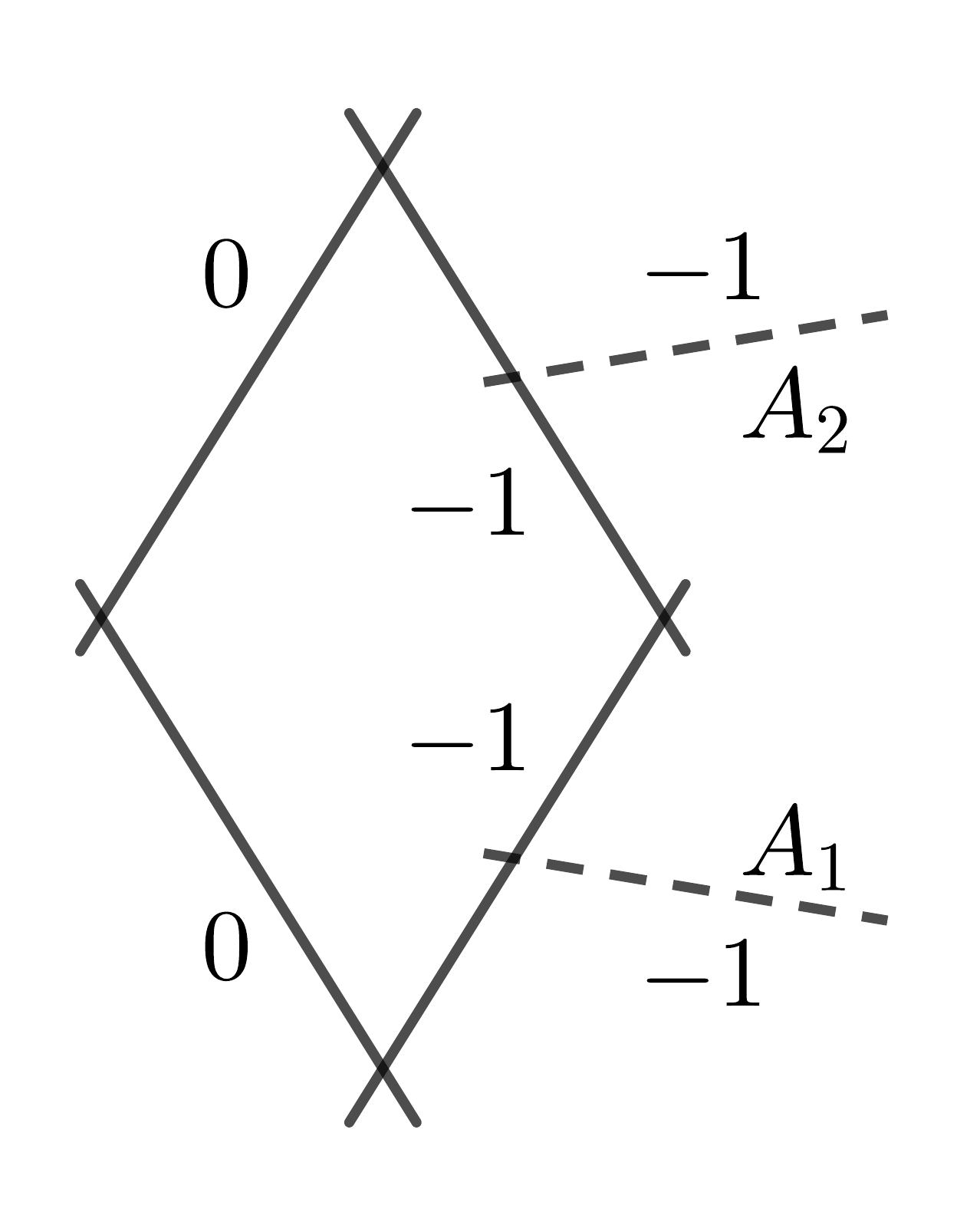}
			\caption{$p_1=p_2=1$}
			\label{fig:S00}
		\end{subfigure}
	\end{tabular}
	\caption{Standard boundaries of the surfaces $S_{p_1,p_2}$ from Theorem \ref{THM}.}
	\label{fig:boundaries}
\end{figure}

\begin{constr}\label{constr}
	In the remaining part of the article we will use the following geometric construction of the surface $S_{p_{1},p_{2}}$, or rather of its log smooth completion. Let $(x_1,x_2)$ be coordinates on $\P^{1}\times \P^{1}$. For $t\in \P^{1}$, $j\in \{1,2\}$ put $\ll_{j,t}=\{x_{j}=t\}$. Write $p_1(t)=t^{d-1}+a_{1}t^{d-2}+\dots+a_{d-1}$. We perform a sequence of $d$ blowups  over $(0,1)$, as follows.  The exceptional curve of the first blowup is parametrized by $x_{2}'=(x_2-1)x_1^{-1}\in \P^{1}$, where $x_{2}'=\infty$ lies on the proper transform of $\ll_{1,0}$. We blow up the point $x_{2}'=a_{1}$, and so on: after $m$-th blowup, the exceptional curve is parameterized by $x_{2}^{(m)}=x_2 x_1^{-m}-x_1^{-m}-a_{1}x_1^{-m+1}-\dots-a_{m-1}$, and we blow up the point with $x_{2}^{(m)}=a_{m}$. Next, we use $p_2$ to analogously define a sequence of $\deg p_2+1$ blowups over the preimage of $(1,0)$. We denote the resulting morphism by $\phi\colon X\to \P^{1}\times \P^{1}$, put $L_{j,t}\de \phi^{-1}_{*}\ll_{j,t}$ and denote by  $A_{1},A_{2}\subseteq X$ the exceptional curves of the last blowup over $(0,1)$ and $(1,0)$, respectively. Eventually, we put $D=\phi^{*}(\ll_{1,\infty}+\ll_{2,\infty}+\ll_{1,0}+\ll_{2,0})\redd-A_{1}-A_{2}$ and $S_{p_1,p_2}=X\setminus D$.
	
	Figure \ref{fig:Skl} shows the graph of $D+A_{1}+A_{2}$, where we use the following notation: each line denotes a curve isomorphic to $\P^{1}$ of specified self-intersection number; the solid ones are the components of $D$, the dashed ones are $A_{1}$ and $A_{2}$; and $[(2)_{m}]$ is short for a chain of $m$ $(-2)$-curves (see Section \ref{sec:log_surfaces}). 
	
	If $p_1,p_2\neq 1$ or $p_1=p_2=1$, then $(X,D)$ is a \emph{standard} completion of $S_{p_1,p_2}$ in the sense of \cite{FKZ-weighted-graphs}, see Section \ref{sec:standard}. To obtain a standard completion in the remaining case, blow up at the preimage of the point $(\infty,\infty)$ and contract $L_{1,\infty}+L_{2,0}$: this gives a graph as in Figure \ref{fig:Sk0}. In Corollary \ref{cor:st}\ref{item:distinct} we use \cite[Corollary 3.36]{FKZ-weighted-graphs} to prove that $S_{p_1,p_2}\not\cong S_{p_1',p_2'}$ for $\{p_1,p_2\}\neq \{p_1',p_2'\}$ just by comparing these graphs.
\end{constr}

In Section \ref{sec:diff} we describe the diffeomorphism types of surfaces in $\cS_{0}$. The result is summarized in Theorem \ref{thm:diff}. We refer to \cite[\S 4,5]{GS_Kirby} for a language of Kirby diagrams (see also Section \ref{sec:Kirby_overview}) and use the notation of \cite{BW_surgery} for $2$-bridge knots.

\begin{thm}\label{thm:diff}
	The surface $S_{p_{1},p_{2}}$ from Theorem \ref{THM} is diffeomorphic to the interior of a $2$-handlebody on a $0$-framed $2$-bridge knot $K_{[2d_1,2d_2]}$, where $d_{j}=\deg p_{j}+1$, $j\in \{1,2\}$. 
	Moreover, $S_{p_{1},p_{2}}$ is not homeomorphic to $S_{p_{1}',p_{2}'}$ for $\{\deg p_{1},\deg p_{2}\}\neq \{\deg p_{1}', \deg p_{2}'\}$.  
\begin{figure}[ht]
	\includegraphics[scale=0.35]{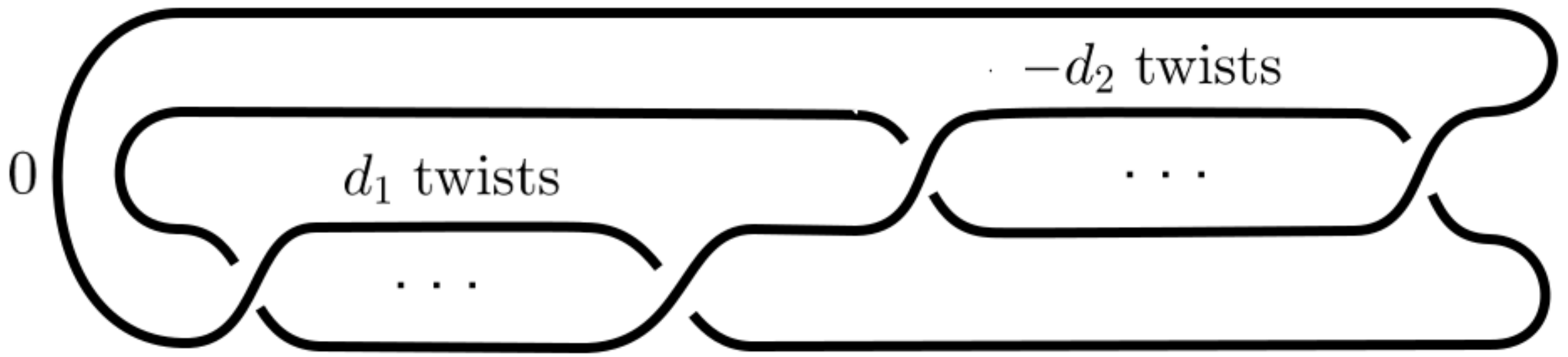}
	\caption{Kirby diagram of a $4$-manifold whose interior is $S_{p_{1},p_{2}}$ with $\deg p_{j}=d_{j}-1$.}
	\label{fig:knot}
\end{figure}	
\end{thm}
Theorem \ref{thm:diff} follows by some elementary handle slides from Proposition \ref{prop:handles}, where we translate Construction \ref{constr} to the Kirby diagram in Figure \ref{fig:handles}. 
	 	 \begin{figure}[ht]
	 	 	\includegraphics[scale=0.4]{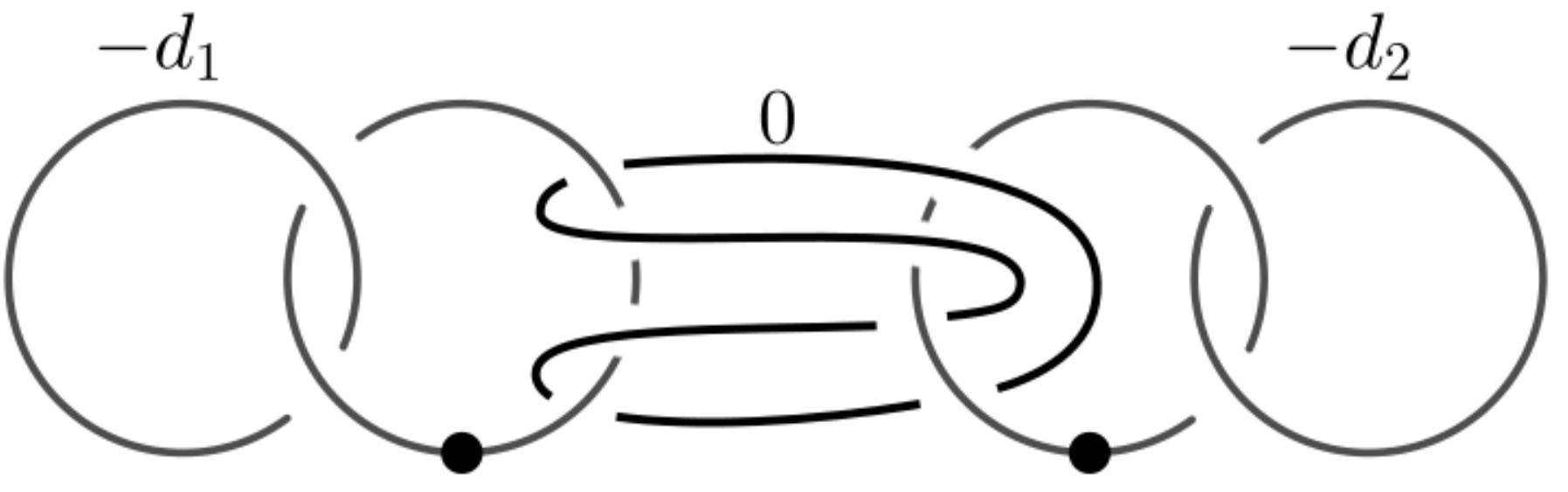}
	 	 	\caption{Kirby diagram equivalent to the one in Figure \ref{fig:knot}, cf.\ \cite[Fig.\ 5.14]{GS_Kirby}.}
	 	 	\label{fig:handles}
	 	 \end{figure}	
	 	 
In Corollary \ref{cor:not_homeo} we prove that $S_{p_{1},p_{2}}$ for different degrees of $p_{j}$ are not homeomorphic, because they have different homotopy types at infinity, see \cite[1.21]{Fujita-noncomplete_surfaces}. Like in the proof of Theorem \ref{THM}, we deduce it from the graph of $D$, this time applying the calculus of graph $3$-manifolds \cite{Neumann-plumbing_graphs} to $M=\d \mathrm{Tub}(D)$, see \cite{Mumford-surface_singularities}. That argument becomes more explicit once it is related to the fact that, by Theorem \ref{thm:diff}, $M$ is a $0$-surgery on $K_{[2d_1,2d_2]}$.  We exploit it in Section \ref{sec:pi_1}, giving a thorough geometric description of $M$. The fundamental group $\pi_{1}(M)=\pi_{1}^{\infty}(S_{p_{1},p_{2}})$ is computed in Proposition \ref{prop:pi_1}.	

\begin{rem}
	We do not know which of the $S_{p_{1},p_{2}}$, for fixed degrees of $p_{1},p_{2}$, are biholomorphic. 
\end{rem}

An interesting homogeneity property of $S_{1,1}$ was shown in \cite[\S 8]{Kal-Ku1_AL-theory}: although $S_{1,1}$ has no algebraic automorphisms except $(x_1,x_2)\mapsto(x_2,x_1)$, it admits a lot of nice holomorphic ones, coming from the flows of algebraic vector fields. Theorem \ref{thm:AAut} shows that the same holds for all $S_{p_{1},p_{2}}$. Denote by $\AAut(S_{p_1,p_2})$ the subgroup of $\Aut_{\mathrm{hol}}(S_{p_1,p_2})$ generated by elements of flows of complete algebraic vector fields on $S_{p_{1},p_{2}}$. 

\begin{thm}\label{thm:AAut}
	We keep notation from Construction \ref{constr}.
	\begin{enumerate}
		\item\label{item:Aut_thm} If $p_1\neq p_2$ then $\Aut(S_{p_1,p_2})$ is trivial. Otherwise, $\Aut(S_{p,p})=\Z_{2}$ is generated by $(x_1,x_2)\mapsto (x_2,x_1)$.
		\item\label{item:AAut_thm} The group $\AAut(S_{p_{1},p_{2}})$ acts $m$-transitively, for any $m$, on its open orbit $S_{p_1,p_2}\setminus Z$, where:
		\begin{enumerate}
			\item $Z=\emptyset$ if $p_{1}=p_{2}=1$,
			\item $Z=((L_{3-j,1}\cap A_{j}) \cup A_{3-j})\cap S_{p_1,p_2} \cong \{\mathrm{pt}\}\sqcup \C^{1}$ if $p_{j}=1, p_{3-j}\neq 1$ for some $j\in \{1,2\}$,
			\item $Z=(A_{1}\cup A_{2})\cap S_{p_1,p_2}\cong \C^{1}\sqcup \C^{1}$ if $p_{1},p_{2}\neq 1$.
		\end{enumerate}
	\end{enumerate}
\end{thm}

Part \ref{item:Aut_thm} is obtained as  Corollary \ref{cor:st}\ref{item:Aut} from the uniqueness of standard boundary. To prove \ref{item:AAut_thm} it suffices, by \cite[Theorem 1.6]{KKL_avf}, to describe all $\C^{*}$-fibrations of $S_{p_{1},p_{2}}$: this is done in Proposition \ref{prop:fibrations}. 
\bigskip

We now explain what is missing in the proof of \cite[Theorem 2]{GM_k<2}, and how we are going to correct it. The aforementioned result is  deduced from the (partial, but sufficient) classification of almost minimal pairs of Kodaira dimension zero given in \cite[\S 8]{Fujita-noncomplete_surfaces}. For the general statement see \cite{Kojima_k0}. 

More precisely, let $S\in \cS_{0}$ and let $(X_0,D_0)$ be a smooth completion of $S$. Then by the Miyanishi's theory of peeling (see \cite[p.\ 107]{Miyan-OpenSurf} or Section \ref{sec:peeling} below), there is a sequence of birational morphisms
\begin{equation*}
(X_0,D_0)\toin{\psi_{1}}(X_{1},D_{1})\toin{\psi_{2}}\dots \toin{\psi_{n}} (X_{n},D_{n}),
\end{equation*}
such that $(X_{n},D_{n})$ is almost minimal, hence known, and each $\psi_{i}$ is an snc-minimalization of $D_{i-1}+A$, where $A\subseteq X_{i-1}$ is a $(-1)$-curve not contained in $D_{i-1}$ such that $A\cdot D_{i-1}=1$. The latter process is called a \emph{half-point attachment} in \cite[8.15]{Fujita-noncomplete_surfaces}. Therefore, to find $S$ it suffices to identify an $(X_{n},D_{n})$ on the list in \cite[8.70]{Fujita-noncomplete_surfaces} (cf.\  \cite[Proposition 1.5]{Kojima_k0}), and perform suitable half-point attachments.

The proof in \cite[Theorem 2]{GM_k<2} proceeds exactly along these lines. However, what is meant in loc.\ cit.\ by a half-point attachment is just one blowup, the contraction of $A$ (cf.\ e.g.\ \cite[Definition 4.15]{FZ-deformations}). Nonetheless, if $A$ meets a $(-2)$-curve $C\subseteq D_{i-1}$ which is not branching in $D_{i-1}$, then after the contraction of $A$, the image of $D_{i-1}$ is not snc-minimal, and $\psi_{i}$ should contract the image of $C$, too. 

This is why in \cite{GM_k<2} only $S_{1,1}$ is reconstructed. Indeed, it is shown (essentially) that $n=2$, $X_{2}=\P^{1}\times \P^{1}$ and $D_{2}$ is a union of two vertical and two horizontal lines. But $(X,D)\to (X_{2},D_{2})$ is obtained by two blowups at smooth points of $D_2$, see Figure \ref{fig:S00}. The smooth completions of $S_{p_1,p_2}$, which are obtained from $(X_{2},D_{2})$ by two \emph{sequences of blowups} over $D_{2}$, see Figure \ref{fig:Skl}, are missing.

In Section \ref{sec:proof} we repeat the above proof, filling in the missing part. Our approach contains no essentially new ingredient. However, since the assumption $S\in \cS_{0}$ is substantially stronger than just $\kappa(S)=0$, we do not need to rely on the structure theorem \cite[8.70]{Fujita-noncomplete_surfaces}, as it is done in \cite{GM_k<2}. So for the convenience of the reader, we make our proof more  self-contained. The only general result we refer to is the theorem of Kawamata \cite[2.2]{Kawamata-classification_surfaces} which asserts that if $\kappa(S)=0$ then the positive part of the Zariski decomposition of $K_{X_{n}}+D_{n}$ is numerically trivial, see \cite[II.6.2.1]{Miyan-OpenSurf}. 
\bigskip

\textbf{Acknowledgments.} We would like to thank the referee for valuable comments, including pointing out a flaw in the earlier version of the proof. The first author would like to thank prof.\ Frank Kutzschebauch for drawing his attention to the article \cite{FKN_k0} and for many insightful discussions on the subject.

\section{Preliminaries}\label{sec:prelims}

\subsection{Log surfaces}\label{sec:log_surfaces}

We now briefly recall the language of log surfaces. For a complete introduction we refer to \cite{Fujita-noncomplete_surfaces}.

By a \emph{curve} we mean an irreducible, reduced variety of dimension $1$.

Let $X$ be a smooth projective surface. A curve $C\subseteq X$ is called an \emph{$n$-curve} if $C\cong \P^{1}$ and $C^{2}=n$. In particular, a $(-1)$-curve is an exceptional curve of a blowup, and a $0$-curve is a fiber of a $\P^{1}$-fibration.

Let $D$ be an effective $\Q$-divisor on $X$. By a \emph{component} of $D$ we always mean an irreducible component. The number of components of $D$ is denoted by $\#D$. The \emph{branching number} of a component $C\subseteq D$ is
\begin{equation*}
\beta_{D}(C)=C\cdot (D-C).
\end{equation*}
We say that $C$ is a \emph{tip} of $D$ if $\beta_{D\redd}(C)=1$, and \emph{branching} in $D$  if $\beta_{D\redd}(C)\geq 3$. 

Assume that $D$ is \emph{simple normal crossing} (\emph{snc}), that is, $D$ is reduced, all components of $D$ are smooth and meet transversally, at most two at each point. A $(-1)$-curve $C\subseteq D$ is called \emph{superfluous} if $1\leq \beta_{D}(C)\leq 2$ and $C$ meets each component of $D$ at most once: in other words, after contraction of $C$, the image of $D$ remains snc. We say that $D$ is \emph{snc-minimal} if it contains no superfluous $(-1)$-curves. An \emph{snc-minimalization} of $D$ is the contraction of all superfluous $(-1)$-curves in $D$ and its images.

Assume further that $D$ has connected support and $\beta_{D}(C)\leq 2$ for all components $C$ of $D$. We say that $D$ is a \emph{chain} if at least one inequality is strict, otherwise we say that $D$ is \emph{circular}. We order the components $T_{1},\dots, T_{m}$ of $D$ in such a way that $T_{i}\cdot T_{i+1}=1$ for $i\in \{1,\dots, m-1\}$, and $T_{1}$ is a tip of $D$ in case $D$ is a chain. The sequence of integers $(-T_{1}^{2},\dots, -T_{m}^{2})$ is then called a \emph{type} of $D$. We often abuse notation and write 
\begin{equation*}
D=[-T_{1}^{2},\dots,-T_{m}^{2}]\mbox{ if $D$ is a chain and } D=\lc -T_{1}^{2},\dots, -T_{m}^{2}\rc \mbox{ if $D$ is circular}.
\end{equation*}
A sequence consisting of an integer $a$ repeated $k$ times will be abbreviated as $(a)_{k}$.

An effective divisor $D$ is called a \emph{tree} if $D\redd$ is snc and has no circular subdivisor. A \emph{fork} is a tree with exactly one branching component and three maximal twigs. A \emph{rational} tree (chain, fork...) is a tree (chain,  fork...) whose components are all rational.

We say that a reduced divisor $D=\sum D_{i}$ is \emph{negative definite} if its intersection matrix $[D_{i}\cdot D_{j}]_{1\leq i,j\leq r}$ is.

Let again $D$ be an snc divisor. A chain $T\subseteq D$ is a \emph{twig} of $D$ if it contains a tip of $D$ and its components are non-branching in $D$. We order a twig in such a way that its first component is a tip of $D$. A twig is called \emph{admissible} if all its components are rational and have self-intersection number at most $-2$. An (admissible) twig of $D$ is \emph{maximal} if it is maximal in the set of (admissible) twigs ordered by inclusion of supports. A \emph{$(-2)$-twig} is a twig whose components are $(-2)$-curves.

Let $T$ be an admissible twig of $D$. Then $T$ is negative definite, so there is a unique $\Q$-divisor $\Bk_{D}(T)\leq T$, called the \emph{bark} of $T$,  such that for every component $T_{0}$ of $T$ 
\begin{equation*}
T_{0}\cdot\Bk_{D}(T)=T_{0}\cdot (K_{X}+D),
\end{equation*}
that is, $T_{0}\cdot \Bk_{D}(T)=-1$ if $T_{0}$ is a tip of $D$ and $T_{0}\cdot \Bk_{D}(T)=0$ otherwise. An important result for the theory of peeling asserts that the coefficients of $\Bk_{D}(T)$ are strictly between $0$ and $1$  \cite[II.3.3]{Miyan-OpenSurf}.

Assume that $D$ is a connected snc-minimal divisor which is not negative definite. Then $\Bk D$ is defined as the sum of barks of all maximal admissible twigs of $D$. We put $D^{\#}=D-\Bk D$.
\smallskip

Let $D$ be an snc divisor on $X$. We say that a blowup $X'\to X$ at a point $p\in D$ is \emph{outer} with respect to $D$ if $p$ is a smooth point of $D$, otherwise it is called \emph{inner}. We say that a birational map $X\map Y$ \emph{touches} a divisor if it is not an isomorphism in any of its neighborhoods.
\smallskip

A \emph{log smooth} pair $(X,D)$ consists of a smooth projective surface $X$ and an snc divisor $D$ on $X$. Any smooth surface $S$ admits a \emph{log smooth completion}, that is, a log smooth pair $(X,D)$ such that $S\cong X\setminus D$. We say that such a completion is \emph{minimal} if it does not dominate birationally any other log smooth completion, or, equivalently, if $D$ is snc-minimal.

\subsection{Affine surfaces whose coordinate rings are factorial and have trivial units}

For a divisor $D$ on a smooth surface $X$, we denote by $\Z[D]$ the free abelian group whose basis is the set of components of $D$. There is a natural group homomorphism $\Z[D]\to\Pic(X)$. The following lemma is well known, cf.\ \cite[1.17]{Fujita-noncomplete_surfaces} or \cite[Lemma 2.2]{GM_k<2}. 

\begin{lem}[Criterion for $S\in \cS_{\kappa}$]\label{lem:S0}
	Let $(X,D)$ be a smooth completion of a smooth affine surface $S$.
	\begin{enumerate}
		\item\label{item:UFD} $\C[S]$ is a UFD if and only if $\Z[D]\to\Pic(X)$ is surjective.
		\item\label{item:units} $\C[S]^{*}=\C^{*}$ if and only if $\Z[D]\to\Pic(X)$ is injective.
	\end{enumerate}
	Assume that $S\in \cS_{\kappa}$. Then $\Z[D]=\Pic(X)=\NS(X)$ and $q(X)=0$. If $\kappa<2$ then $X$ is rational.
\end{lem}
\begin{proof}
	\ref{item:UFD} By \cite[II.6.2]{Hartshorne_AG}, $\C[S]$ is factorial if and only if the divisor class group $\mathrm{Cl}(S)$ of $S$ is trivial. It follows from \cite[II.6.5(c)]{Hartshorne_AG} that $\mathrm{Cl}(S)$ is the cokernel of the natural map $\Z[D]\to\mathrm{Cl}(X)$. Eventually, $\mathrm{Cl}(X)=\Pic(X)$ because $X$ is smooth \cite[II.6.16]{Hartshorne_AG}.
	
	\ref{item:units} A regular function on $S$ is invertible if and only if it extends to $f\in \C(X)$ with all zeros and poles in $D$, that is, with $\mathrm{div}(f)\in \Z[D]$. Two rational functions on $X$ with the same divisor differ by an element of $\O_{X}(X)^{*}=\C^{*}$, so $f\mapsto \mathrm{div}(f)$ gives an isomorphism of $\C[S]^{*}/\C^{*}$ with the kernel of $\Z[D]\to \Pic(X)$.
	
	For the remaining part, we follow \cite[Lemma 2.2]{GM_k<2}. Assume that $S\in \cS_{\kappa}$. Then \ref{item:UFD}, \ref{item:units} give $\Z[D]=\Pic(X)$. The exponential sequence induces an exact sequence
	\begin{equation*}
	0\to H^{1}(X,\Z)\to H^{1}(X,\O_{X})\to \Pic(X)\to H^{2}(X,\Z).
	\end{equation*}
	Since $H^{1}(X,\O_{X})$ is a vector space over $\C$, and by \ref{item:UFD} $\Pic(X)$ is a finitely generated abelian group, we obtain $q(X)=h^{1}(X,\O_{X})=0$ and that $\Pic(X)\to H^{2}(X,\Z)$ is injective, i.e.\ $\Pic(X)=\NS(X)$. Now to prove that $X$ is rational, it remains to show that $p_{g}(X)=0$ \cite[IV.6.2]{Hartshorne_AG}. Assume $\kappa<2$. Because $S$ is affine, $D$ supports an  effective ample divisor, so if $|K_{X}|\neq \emptyset$ then for $m\gg 0$ we have $h^{0}(m(K_{X}+D))\sim m^{2}$, hence $\kappa=2$; a contradiction.
\end{proof}

\begin{rem}\label{rem:open_UFD}
	Lemma \ref{lem:S0}\ref{item:UFD} implies that an open affine subset of a smooth factorial surface is factorial, and that the converse holds if and only if its complement is a principal divisor. This is a special case of the Nagata lemma, see e.g.\ \cite[Lemma 19.20]{Eisenbud_comm-alg}
\end{rem}

\begin{cor}[$S_{p_1,p_2}\in \cS_{\kappa}$]\label{cor:S_k}
	Let $(X,D)$ be the log smooth completion of $S_{p_1,p_2}$ as in Construction \ref{constr}. Then $\Pic(X)=\Z[D]$. In particular,  $\C[S_{p_1,p_2}]$ is a UFD and $\C[S_{p_1,p_2}]=\C^{*}$.
\end{cor}
\begin{proof}
	We use the notation introduced in Construction \ref{constr}. The group $\Pic(X)$ is generated by the components of $\phi^{*}(\ll_{1,\infty}+\ll_{1,0}+\ll_{2,\infty}+\ll_{2,0})\redd=D+A_{1}+A_{2}$, with relations $0=\phi^{*}(\ll_{j,0}-\ll_{j,\infty})=A_{j}+T_{j}+L_{j,0}-L_{j,\infty}$, $j\in \{1,2\}$, where $T_{j}$ is the twig of $D$ meeting $L_{j,0}$. Therefore, $\Pic(X)=\Z[D+A_{1}+A_{2}]/\Z[A_{1}+A_{2}]=\Z[D]$. The remaining assertion follows from Lemma \ref{lem:S0}.
\end{proof}

\subsection{$\P^{1}$-fibrations}

A \emph{fibration} of a smooth surface $S$ is a surjective morphism onto a curve whose general fiber is irreducible and reduced.  A curve on $S$ is called \emph{vertical} (resp.\ \emph{horizontal}) if $C\cdot F=0$ (resp.\ $C\cdot F>0$) for a fiber $F$; it is called an \emph{$n$-section} if $C\cdot F=n$. Any divisor $D$ decomposes uniquely as $D=D\vert+D\hor$, where all components of $D\vert$ are vertical and all components of $D\hor$ are horizontal.

A $\P^{1}$- (resp.\ $\C^{1}$-, $\C^{*}$-) fibration is a fibration whose general fiber is isomorphic to $\P^{1}$ (resp.\ $\C^{1}$, $\C^{*}\de \C^{1}\setminus \{0\}$). A fiber not isomorphic to a general one will be called \emph{degenerate}. 

The surfaces in $\cS_{0}$ do not admit $\C^{1}$-fibrations by the Iitaka Easy Addition Theorem \cite[Theorem 11.9]{Iitaka_AG}, but they admit plenty of $\C^{*}$-fibrations (eg.\ produced by Lemma \ref{lem:producing_Cst}\ref{item:HP_Cst}), which we will use in Proposition \ref{prop:fibrations}\ref{item:vf} as first integrals for complete algebraic vector fields, see \cite[Theorem 1.6]{KKL_avf}. The structure of $\C^{*}$-fibrations is well described in \cite[\S 7]{Fujita-noncomplete_surfaces}, however, for our purposes it will be more convenient to study them directly, by completing them to $\P^{1}$-fibrations. 

Let $F$ be a degenerate fiber of a $\P^{1}$-fibration of a smooth projective surface. Then $F$ is obtained from a $0$-curve by a sequence of blowups, hence its geometry is easy to understand, see eg.\ \cite[\S 4]{Fujita-noncomplete_surfaces}. In particular, $F$ is a rational tree with no branching $(-1)$-curves; moreover, if a $(-1)$-curve has multiplicity $1$ in $F$ then it is a tip of $F$, and $F$ contains another $(-1)$-curve \cite[7.3]{Fujita-noncomplete_surfaces}.

The following observation, see \cite[7.14(1)]{Fujita-noncomplete_surfaces}, will be used several times.

\begin{lem}[our $\C^{*}$-fibrations are untwisted]\label{lem:untwisted}
	Let $(X,D)$ be a log smooth completion of a smooth affine surface $S$ such that $\C[S]$ is a UFD. Let $H_{1},\dots, H_{r}$ be all components of $D\hor$ for some $\P^{1}$-fibration of $X$. Put $h_j=H_{j}\cdot F$ for a fiber $F$. Then $\gcd(h_{1},\dots, h_{r})=1$.
\end{lem}
\begin{proof}
	By Tsen's theorem, $X$ contains a $1$-section $H$. By Lemma \ref{lem:S0}\ref{item:UFD}, $H\equiv V+\sum_{j=1}^{r} a_{j}H_{j}$ for some vertical $V$ and $a_{1},\dots, a_{r}\in \Z$. Intersecting with $F$ gives $1=\sum_{j=1}^{r} a_{j}h_{j}$, so $\gcd(h_1,\dots, h_r)=1$.
\end{proof}

\subsection{Standard completions and elementary transformations}\label{sec:standard}

A minimal log smooth completion of an open surface may not be unique. However, \cite[Definition 2.13]{FKZ-weighted-graphs} distinguishes a class of \emph{standard completions}, which gives a convenient tool to tell affine surfaces apart. We now recall this definition. A rational chain is \emph{standard} if it is of type
\begin{equation*}
[(0)_{2k+1}]\quad\mbox{or}\quad[(0)_{2k},a_{1},\dots, a_{l}]\quad\mbox{for some integers }k,l\geq 0\mbox{ and }a_{1},\dots,a_{l}\geq 2.
\end{equation*}
A rational circular divisor is \emph{standard} if it is of type
\begin{equation*}
\lc(0)_{2k},a_{1},\dots, a_{l}\rc\mbox{ or }
\lc(0)_{k},a\rc\mbox{ or }
\lc(0)_{2k},1,1\rc\quad\mbox{for some }k,l\geq 0,\ a\geq 0\mbox{ and }a,a_{1},\dots,a_{l}\geq 2.
\end{equation*}
Let now $D$ be any snc divisor, and let $B$ be the sum of all components of $D$ which are either branching or non-rational; so every connected component of  $D-B$ is rational, chain or circular. We say that $D$ is standard if every connected component of $D-B$ is standard in the above sense. A log smooth completion $(X,D)$ is \emph{standard} if $D$ is. Note that the log smooth completions of $S_{p_1,p_2}$ in Figure \ref{fig:boundaries} are standard.
\smallskip

Let $D$ be an snc divisor on $X$, and let $C\subseteq D$ be a non-branching $0$-curve. If $\beta_{D}(C)=2$ choose $p\in C\cap (D-C)$, otherwise let $p$ be any point of $C$. A \emph{flow} (or \emph{elementary transformation}) \emph{on} $C$ is the birational map $(X,D)\map (X',D')$, where $D'$ is the reduced total transform of $D$, defined as a blowup at $p$, followed by the contraction of the proper transform of $C$. We remark that in Section \ref{sec:AAut} we will use the word \enquote{flow} in its standard meaning, i.e.\ flow of a vector field: this will lead to no confusion.

Let $\phi\colon (X,D)\map (X',D')$ be a flow on $C\subseteq D$. Then  $\phi|_{X\setminus D}\colon X\setminus D\to X'\setminus D'$ is an isomorphism. In particular, if the components of $D$ generate $\Pic(X)$ (respectively, are $\Z$-linearly independent in $\Pic(X)$), then by Lemma \ref{lem:S0} the same is true for the components of $D'$ in $\Pic(X')$. Moreover, $\phi$ does not touch components of $D$ other than $C$ and the components of $D$ meeting $C$. If $C$ is not a tip of $D$, then $\phi$ replaces a subchain $[a,0,b]\subseteq D$, whose middle component is $C$, by $[a+1,0,b-1]\subseteq D'$.
\smallskip

We are ready to formulate the result used to distinguish the surfaces $S_{p_1,p_2}$ for different $\{p_1,p_2\}$.

\begin{lem}[{\cite[Corollary 3.36]{FKZ-weighted-graphs}}]\label{lem:boundaries}
	Any smooth affine surface admits a standard log smooth completion. Any two such completions differ by a sequence of flows on some $0$-curves in the boundary.
\end{lem}

\begin{cor}[$S_{p_1,p_2}$ are non-isomorphic]\label{cor:st}
	Let $S_{p_1,p_2}$ be as in Theorem \ref{THM}. Then
	\begin{enumerate}
		\item\label{item:st_uniq} The standard completion $(X,D)$ of $S_{p_1,p_2}$ is unique up to an isomorphism, with $D$ as in Figure \ref{fig:boundaries}.
		\item\label{item:distinct} The surfaces $S_{p_1,p_2}$ and $S_{p_1',p_2'}$ are not isomorphic unless $\{p_1,p_2\}=\{p_1',p_2'\}$.
		\item \label{item:Aut} If $p_1\neq p_2$ then $\Aut(S_{p_1,p_2})$ is trivial. Otherwise, $\Aut(S_{p,p})=\Z_{2}$ is generated by $(x_1,x_2)\mapsto (x_2,x_1)$.
	\end{enumerate}
\end{cor}
\begin{proof}
	\ref{item:st_uniq} By Lemma \ref{lem:boundaries}, $D$ differs from the one in Figure \ref{fig:boundaries} by a sequence of flows. Because all twigs of the one in Figure \ref{fig:boundaries} are admissible, they are not touched by any flow. Since a flow preserves the non-weighted graph and the sum of self-intersection numbers; we only need to check that the latter agree.
	
	Consider the case when $p_1,p_2\neq 1$ or $p_1=p_2=1$. Then any such $D$ contains a rational circular divisor $L$ with components, say, $L_{1},L_{2},L_{3},L_{4}$, such that $L_{i}\cdot L_{i+1}=1$ for $i\in \{1,2,3\}$; if $p_{j}\neq 1$ for $j\in \{1,2\}$ then $L_{2+j}$ meets a twig of $D$ of type $[(2)_{\deg p_j}]$, and $L_{1}^{2}=0$, i.e.\ $L_{1}$ is the $0$-curve on which the last flow is performed. If $p_1,p_2\neq 1$ then the chain $L_{1}+L_{2}$ is standard, and if $p_1=p_2=1$ then the whole $L$ is standard. In any case, we infer that $L_{2}^{2}=0$. It follows that $L_{3}^{2},L_{4}^{2}<0$: indeed, for $i\in \{1,2\}$, $L_{i+2}$ is vertical for the $\P^{1}$-fibration of $X$ induced by $|L_{i}|$, and it is not a fiber because otherwise $L_{i+2}=L_{i}$ in $\Pic(X)$, contrary to Lemma \ref{lem:S0}\ref{item:units}. Because the sum $L_{1}^{2}+L_{2}^{2}+L_{3}^{2}+L_{4}^{2}$ is not changed by a flow, we get $L_{3}^{2}+L_{4}^{2}=-2$, so $L_{3}^{2}=L_{4}^{2}=-1$, as claimed.
	
	In case $p_j\neq 1$, $p_{3-j}= 1$ we similarly obtain that $D$ has a rational circular divisor $L_{1}+L_{2}+L_{3}$, where $L_{3}$ meets a twig $[(2)_{\deg p_j}]$ and $L_{1}$ is a $0$-curve. The chain $L_{1}+L_{2}$ is standard, so $L_{2}^{2}=0$, and $L_{3}^{2}=L_{1}^{2}+L_{2}^{2}+L_{3}^{2}=1$.
	
	\ref{item:distinct}, \ref{item:Aut} Let $S_{p_1,p_2}\to S_{p_1',p_2'}$ be an isomorphism. By Lemma \ref{lem:boundaries}, it extends to an isomorphism, say $\tau\colon (X_1,D_1)\to(X_2,D_2)$, between their standard completions, which by \ref{item:st_uniq} are both as in Construction \ref{constr}. If they are as in Figure \ref{fig:Sk0}, we reverse the last step of Construction \ref{constr} and lift $\tau$ accordingly, so that $(X_{i},D_{i})$ are as in Figure \ref{fig:Skl} or \ref{fig:S00} (but possibly no longer standard). 
	
	For $i,j\in \{1,2\}$ let $A_{j,i}\subseteq X_{i}$ be as in Figure \ref{fig:boundaries}. Because by Lemma \ref{lem:S0} $\Pic(X_{i})=\Z[D_{i}]$, the curve $\tau(A_{j,1})$ is linearly equivalent to $A_{j,2}$. In fact, $\tau(A_{j,1})=A_{j,2}$ because $\tau(A_{j,1})\cdot A_{j,2}=A_{j,2}^{2}<0$. Therefore, $\tau$ descends to an automorphism $\tau_{0}$ of $\P^{1}\times \P^{1}$ fixing $\ll_{1,0}+\ll_{1,\infty}+\ll_{2,0}+\ll_{2,\infty}$ and the points $(0,1)$, $(1,0)$, so  $\tau_{0}(x_1,x_2)=(x_1,x_2)$ or $(x_2,x_1)$. Hence $\tau$ or $\tau\circ \epsilon$, where $\epsilon(x_1,x_2)=(x_2,x_1)$, is trivial on each exceptional curve of $\phi\colon X\to \P^{1}\times \P^{1}$. This shows that $\{p_1,p_2\}=\{p_1',p_2'\}$ and $\tau=\id$ or $\epsilon$.
\end{proof}

\begin{rem}[{$S_{p_1,p_2}$ are stably non-isomorphic, cf.\ \cite[Corollary 4.3]{FKN_k0}}]
	Because $\kappa(S_{p_1,p_2})=0$ by Corollary \ref{cor:S_kappa} below, \cite{Iitaka_Fujita-cancellation}  implies that $S_{p_1,p_2}\times \C^{n}\not\cong S_{p'_1,p'_2}\times \C^{n}$ for any $n\geq 0$ and $\{p_1,p_2\}\neq \{p_1',p_2'\}$.
\end{rem}

We conclude with a well-known application of flows, which will be useful in the proof of Lemma \ref{lem:producing_Cst}.

\begin{lem}[{\cite[6.13]{Fujita-noncomplete_surfaces}, cf.\ \cite[Lemma 4.14]{FZ-deformations}}]\label{lem:rational-twigs-are-admissible}
	Let $(X,D)$ be a log smooth pair with $\kappa(X\setminus D)\geq 0$. Let $T$ be a rational twig of $D$. Then $C^2\leq -1$ for every component $C$ of $T$. In particular, if $D$ is snc-minimal then all rational twigs of $D$ are admissible.
\end{lem}
\begin{proof}
	Let $T_{k}$ be the $k$-th component of $T$, put $T_0=0$. Suppose $T_{k}^{2}\geq 0$ for some $k\in \{1,\dots, \#T\}$. Write $\{p\}=T_{k}\cap (D-T_{k}-T_{k-1})$, or choose $p\in T_{k}\setminus T_{k-1}$ if that intersection is empty. Blowing up $T_{k}^{2}$ times at $p$ and its infinitely near points on the proper transforms of $T_{k}$, we can assume that $T_{k}=[0]$. If $k>1$, then after $\pm T_{k-1}^{2}$ flows on $T_{k}$, we can assume $T_{k-1}=[0]$, too. Thus by induction, we can assume $k=1$, i.e.\ $T_{1}=[0]$ is a tip of $D$. Then $|T_{1}|$ induces a $\P^{1}$-fibration of $X$ which restricts to a $\C^{1}$-fibration of $X\setminus D$, so $\kappa(X\setminus D)=-\infty$ by the Iitaka Easy Addition theorem; a contradiction.
\end{proof}

\subsection{Graph $3$-manifolds and their normal forms}\label{sec:graph_manifolds}

An important topological invariant of an affine surface $S$ is its \emph{fundamental group at infinity},  $\pi_{1}^{\infty}(S)\de\invlim \pi_{1}(S\setminus K)$, where $\invlim$ runs over all compact $K\subseteq S$. If $(X,D)$ is a log smooth completion of $S$, then $\pi_{1}^{\infty}(S)=\pi_{1}(M)$, where $M$ is the boundary of a nice tubular neighborhood of $D$, constructed in \cite{Mumford-surface_singularities} as a plumbed manifold whose graph is the dual graph of $D$, see \cite[4.6.2]{GS_Kirby}. More precisely, $\mathrm{Tub}(D)$ is a union, taken over components $C$ of $D$, of disk bundles $\xi_C$ with Euler numbers $C^2$, where for any component $C'$ of $D$ meeting $C$, a fiber of $\xi_C$ over $C\cap C'$ is  glued to a neighborhood of $C\cap C'$ in $C$.  

For a general introduction to $3$-manifolds we refer to \cite{Hatcher_3M,AFW_3M}. 
We denote by $\bS^{k}$, $\bD^{k}$ the (real) $k$-dimensional sphere and disk, respectively, and by $\bT^{k}=(\bS^{1})^{k}$ the $k$-dimensional torus. We say that a $3$-manifold is \emph{prime} if it cannot be written as a connected sum of two $3$-manifolds other than $\bS^{3}$, and \emph{irreducible} if any embedded $\bS^{2}$ bounds a ball. A prime, oriented $3$-manifold is either irreducible or $\bS^{1}\times \bS^{2}$.

If $M$ is prime, and not a lens space (which will be the case for $S=S_{p_{1},p_{2}}$), then $\pi_{1}(M)$ determines $M$ uniquely, up to a diffeomorphism \cite[Theorem 2.1.2]{AFW_3M}. Therefore, in this case to distinguish the homeomorphism type of $S$ at infinity, it suffices to describe $\pi_1(M)$, which, in turn, is given by the graph of $D$, once put in a normal form \cite[Theorem 4.2]{Neumann-plumbing_graphs}. The definition of normal form is similar to the one of standard boundary, but requires some more flexibility to allow more operations between $3$-manifolds than just blowing up and down. We now recall the definition of graph $3$-manifolds following \cite{Neumann-plumbing_graphs}. We refer to \cite{EN_torus-links} for details and relation to the JSJ-decomposition (graph manifolds are exactly those for which the latter consists only of Seifert fibered spaces: we compute it for our $M$ in Proposition \ref{prop:JSJ}).

A \emph{graph manifold} is a $3$-manifold which is a union of $\bS^{1}$-bundles over compact surfaces, glued according to some  graph, as follows. An $\bS^{1}$-bundle over $C$ with Euler number $e$ is represented by a vertex of weight $e$ and two additional numbers $g$, $r$, where $g$ is the genus of $C$ ($g<0$ if $C$ is non-orientable)  and $r=b_{0}(\d C)$. They are skipped whenever both are zero, i.e.\ $C\cong \P^{1}$. An edge between two vertices $v_{1}$, $v_{2}$ represents gluing of a fiber of $v_{1}$ to a small loop around a point in a base of $v_2$, and vice versa. Multiple edges and loops are allowed. An edge is labeled by \enquote{$+$} if the gluing respects the chosen orientations, and \enquote{$-$} if it reverses both of them.  Clearly, reversing the sign of all edges adjacent to a single vertex, or of any edge adjacent to a vertex with $g<0$, does not change the manifold. Hence we will label only the edges contained in circular subgraphs with no vertices of $g<0$, see \cite[p.\ 304]{Neumann-plumbing_graphs}.

For example, the boundary $M$ from \cite{Mumford-surface_singularities}, described above, is a graph manifold associated to the dual graph of $D$, with all edges labeled with \enquote{$+$}. For such graph, we use the notions of branching numbers, chains, twigs as in Section \ref{sec:log_surfaces}, a vertex being \enquote{rational} if $g=r=0$.

A connected, rational graph $\Gamma$ is \emph{normal} if the following three conditions hold. First, all non-branching vertices of $\Gamma$ have weights at most $-2$. Second, if a vertex with $\beta=3$ meets two twigs of type $[2]$ then $\Gamma$ is a fork. Third, if $\Gamma=\lc (2)_{k} \rc$ for some $k> 0$ then at least two edges are labeled with \enquote{$-$}. To extend this definition to non-rational graphs, one needs to exclude some additional special cases, see \cite[\S 4]{Neumann-plumbing_graphs}. We will not need this extension. In general, a graph is \emph{normal} if all its connected components are.

For a graph $3$-manifold $N$, \cite[Theorem 4.1]{Neumann-plumbing_graphs} gives an algorithm to reduce its graph to a normal form, which we denote by $\Gamma(N)$. This algorithm uses certain operations R0--R8 defined in Proposition 2.1 loc.\ cit, which do not change the graph manifold. Among these, R1 is blowing down a vertex with weight $\pm 1$, and R3 is an \enquote{absorption} of a cylinder $(\bD^{1}\times \bS^{1})\times \bS^{1}$ corresponding to a rational vertex of weight $0$ and $\beta=2$.  The graph $\Gamma(-N)$ can be directly computed from $\Gamma(N)$  using \cite[Theorem 7.1]{Neumann-plumbing_graphs}.

Now we can state the main result of \cite{Neumann-plumbing_graphs} used to distinguish $\pi_{1}^{\infty}(S_{p_1,p_2})$ for different degrees of $p_{j}$.

\begin{lem}[{\cite[Theorems 4.2, 4.3]{Neumann-plumbing_graphs}}]\label{lem:graphs}
	Let $N,N'$ be graph $3$-manifolds.
	\begin{enumerate}
		\item\label{item:Gamma}  $\Gamma(N)=\Gamma(N')$ if and only if $N$ and $N'$ are orientation-preserving diffeomorphic.
		\item\label{item:irr} If $\Gamma$ is connected then $N$ is prime.
	\end{enumerate}
\end{lem}

	 	 \begin{figure}[ht]	
	\begin{tabular}{cc}
		\begin{subfigure}[t]{0.45\textwidth}
			\centering	 	 			
			\includegraphics[scale=0.3]{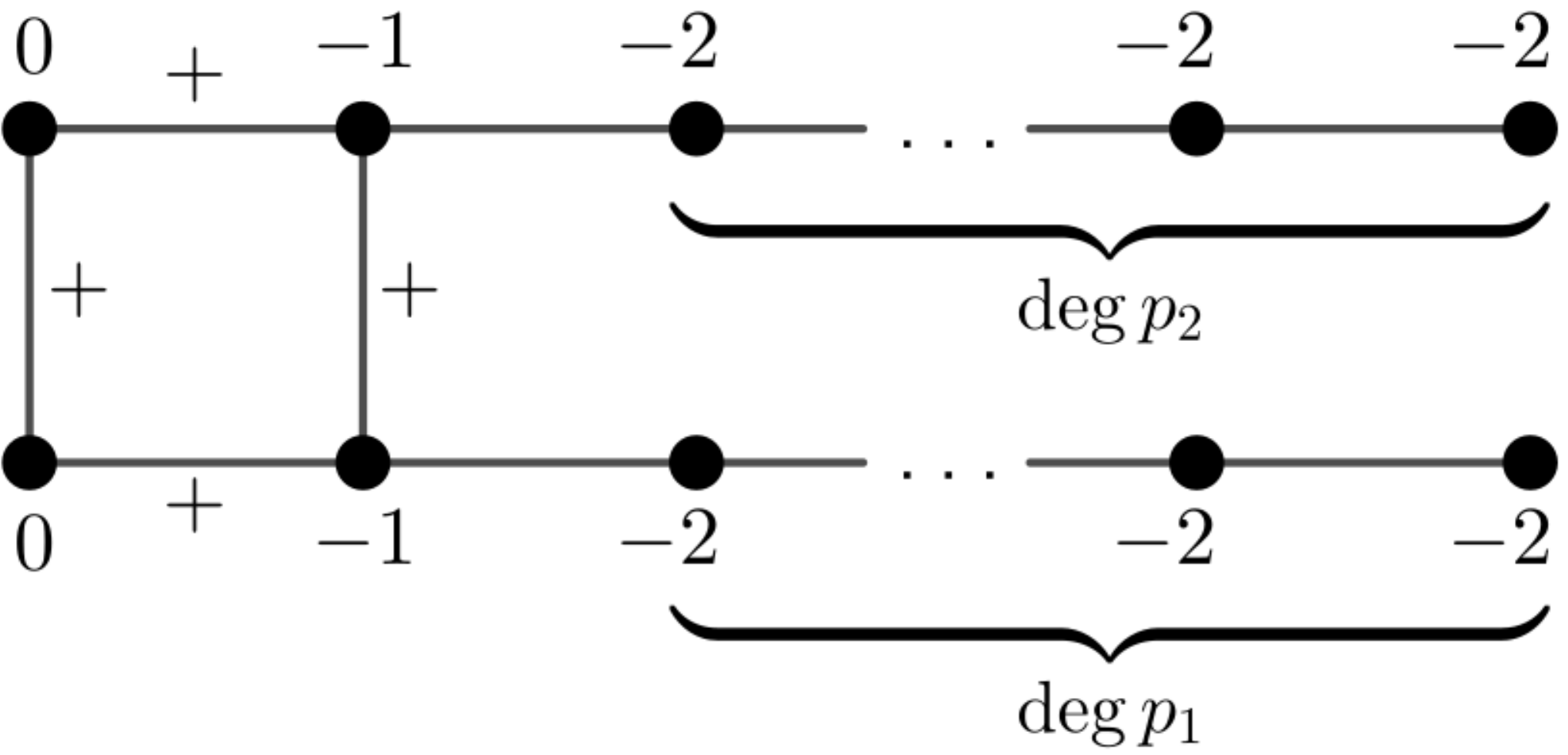}
			\caption{Graph of $M$ as in \cite{Mumford-surface_singularities}, see Figure \ref{fig:Skl}}
			\label{fig:g_1}
		\end{subfigure}	
		&
		\begin{subfigure}[t]{0.45\textwidth}
			\centering	 	 			
			\includegraphics[scale=0.3]{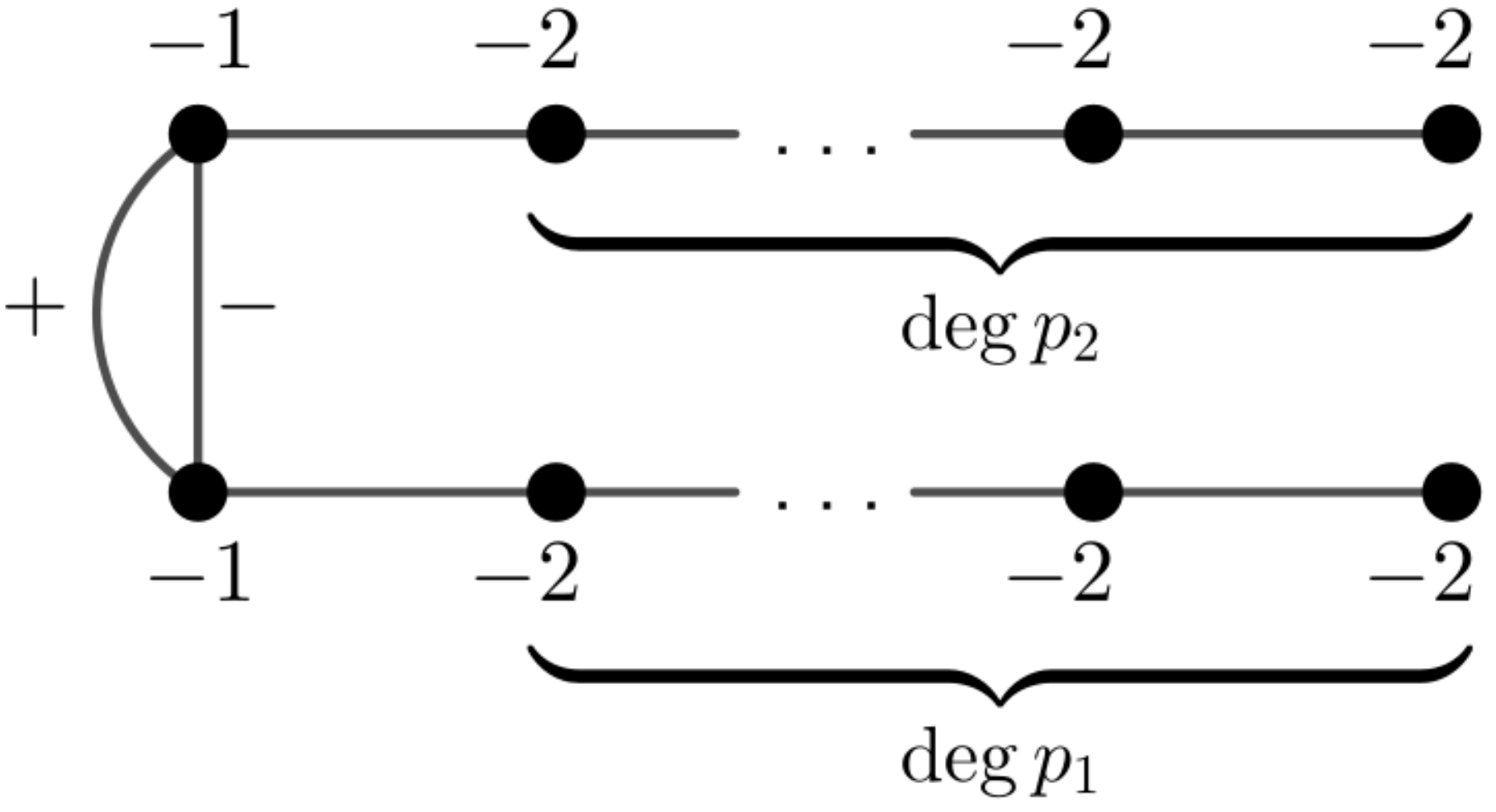}
			\caption{$\Gamma(M)$ for $p_{1},p_{2}\neq 1$}
			\label{fig:g_2}
		\end{subfigure}
		\\
		\begin{subfigure}[t]{0.45\textwidth}
			\centering	 	 			
			\includegraphics[scale=0.35]{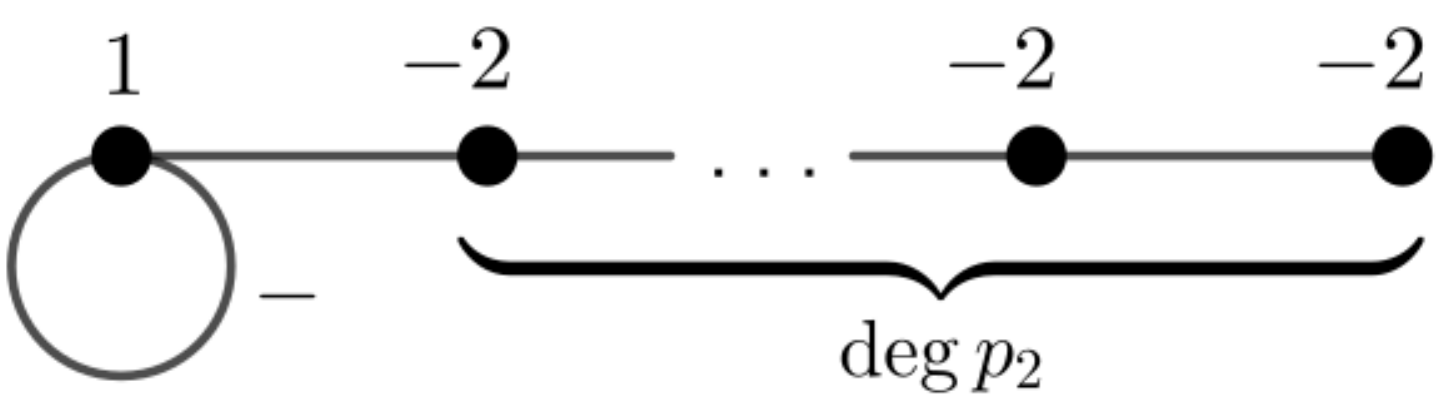}
			\caption{$\Gamma(M)$ for $p_{1}=1$, $p_{2}\neq 1$}
			\label{fig:g_3}
		\end{subfigure}	
		&
		\begin{subfigure}[t]{0.45\textwidth}
			\centering
			\includegraphics[scale=0.3]{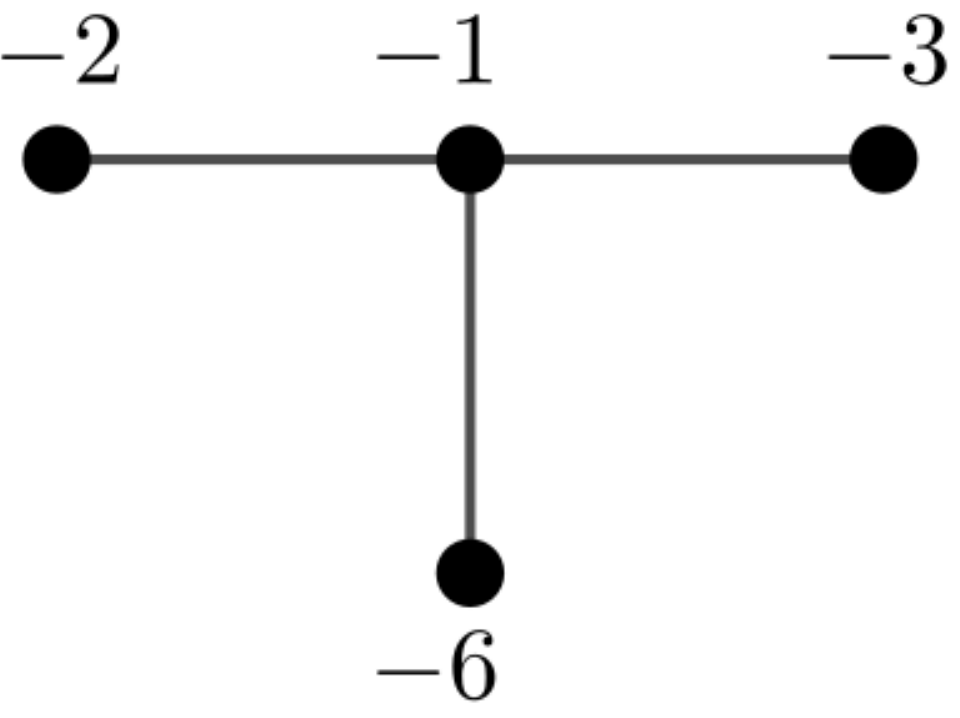}
			\caption{$\Gamma(-M)$ for $p_{1},p_{2}=1$, see Proposition \ref{prop:JSJ}\ref{item:Moser}}
			\label{fig:g_4}
		\end{subfigure}
	\end{tabular}
	\caption{Graphs of $M=\d\mathrm{Tub}(D)$ in the proof of Corollary \ref{cor:not_homeo}}
	\label{fig:graphs}
\end{figure}

\begin{cor}\label{cor:not_homeo}
	The surfaces $S_{p_{1},p_{2}}$ and $S_{p_{1}',p_{2}'}$  are not homeomorphic for $\{\deg p_{1},\deg p_{2} \}\neq \{\deg p_{1}',\deg p_{2}' \}$.
\end{cor}
\begin{proof}
	For $S_{p_{1},p_{2}}$ and $S_{p_{1}',p_{2}'}$, $\{\deg p_{1},\deg p_{2}\}\neq\{\deg p_{1}',\deg p_{2}'\}$ let $M$, $M'$ be a boundary of tubular neighborhoods of respective $D\subseteq X$ as in Construction \ref{constr}.  Such $M$ is a graph manifold represented by Figure \ref{fig:g_1}. Absorption (move R3 in \cite{Neumann-plumbing_graphs}) of a vertex corresponding to $L_{1,\infty}$ reduces it to a graph in Figure \ref{fig:g_2}. It is in normal form if $p_{1},p_{2}\neq 1$. If, say, $p_{1}=1$ then the normal form is obtained by blowing down (move R1) the vertex corresponding to $L_{1,0}$. If $p_{2}\neq 1$, then the resulting graph in Figure \ref{fig:g_3} is normal, otherwise it needs to be replaced by a standard graph of a Seifert fiber space (move R7). In our case this graph is a graph $E_{9}$, which after orientation reversing gives the one in Figure \ref{fig:g_4}, cf.\ Proposition \ref{prop:JSJ}\ref{item:Moser} and \cite[p.\ 309]{Neumann-plumbing_graphs}. 
	 	 	
	Hence $\Gamma(M)\neq \Gamma(M')$, so by Lemma \ref{lem:graphs}\ref{item:Gamma} $M$ is not orientation-preserving diffeomorphic to $M'$. Reversing orientation in Figures \ref{fig:graphs}\subref{fig:g_2}-\subref{fig:g_3} inverts the edge signs and replaces a twig $[(2)_{d}]$ with $[d+2]$, so the resulting graph is different from the previous ones. Therefore, $M\not\cong_{\mathrm{diffeo}}M'$, so $M\not\cong_{\mathrm{homeo}}M'$, see \cite[1.1]{AFW_3M}. Because $\Gamma(M)$ is connected and not a rational chain, $M$ is prime, and not a lens space (see \cite[6.1]{Neumann-plumbing_graphs} for lens space graphs), so by \cite[2.1.2]{AFW_3M} $\pi_{1}(M)\not\cong \pi_{1}(M')$. The result follows because $\pi_{1}^{\infty}(S_{p_{1},p_{2}})=\pi_{1}(M)$ \cite{Mumford-surface_singularities}.
\end{proof}

\subsection{Almost minimal models}\label{sec:peeling}

We now recall the construction of an almost minimal model for a log smooth pair $(X,D)$. We consider only the case when $X\setminus D$ is affine, for a complete treatment we refer to \cite[\S 3]{Miyan-OpenSurf}, see p.\ 107 loc.\ cit.\ for a summary. The exposition here is based on \cite[\S 3]{Palka-minimal_models}.

The motivation for this construction to get an explicit description of a log MMP run for the pair $(X,D)$, such that the log terminal singularities on the image of $(X,D)$ are introduced as late as possible: in fact, \emph{an almost minimal model} is a log resolution of the minimal model in the sense of Mori. Having understood the latter, one can use this theory to reconstruct the initial pair $(X,D)$. 

The almost minimal model will be created by iterating birational morphisms of the following type.

\begin{dfn}[{half-point attachment, cf.\ \cite{Fujita-noncomplete_surfaces}}]\label{def:hp}
Let $(X,D)$ be a log smooth pair. Assume that
\begin{equation}\label{eq:hp}
A\subseteq X\mbox{ is a }(-1)\mbox{-curve such that } A\not\subseteq D\mbox{ and } A\cdot D=1.
\end{equation}
The composition of the contraction of $A$ and new superfluous $(-1)$-curves in the subsequent images of $D$:
\begin{equation*}
\phi_{A}\colon (X,D)\to (X',D'),
\end{equation*} 
where $D'=(\phi_{A})_{*}D$, is called a \emph{half-point attachment} (of $A$).
\end{dfn}
\begin{rem*}
	In general, $\phi_{A}$ is \emph{not} uniquely determined by $A$. Indeed, it may happen that after some contraction within $\phi_{A}$, there appear two superfluous $(-1)$-curves. For example, if $A$ meets the middle component $T'$ of a subchain $T=[2,2,2]$ of $D$, whose components are non-branching in $D$, then after the contraction of $A+T'$, both components of the image of $T$ are superfluous $(-1)$-curves in the image of $D$, but only one of them can be contracted.
	
	However, such ambiguity never happens if $\C[X\setminus D]$ is a UFD. Indeed, Lemma \ref{lem:producing_Cst}\ref{item:HP_C} shows that in this case $\phi_{A}$ contracts exactly $A$ and a maximal $(-2)$-twig of $D$ meeting $A$ (if such occurs).
\end{rem*}

\begin{lem}\label{lem:HPD_k}
	Let $(X,D)\to (X',D')$ be a half-point attachment. Then for all $m\geq 0$ we have $h^{0}(m(K_{X'}+D'))=h^{0}(m(K_{X}+D))$. In particular, $\kappa(X'\setminus D')=\kappa(X\setminus D)$.
\end{lem}
\begin{proof}
	Let $A$ be as in \eqref{eq:hp}. We have $A\cdot(K_{X}+D+A)=-1<0$, so for any $m>0$, $mA$ is in the fixed locus of $|m(K_{X}+D+A)|$. Therefore, $|m(K_{X}+D)| \cong |m(K_{X}+D+A)| \cong |m(K_{X'}+D')|$ because $X\to X'$ contracts only components of $D+A$.
\end{proof}

\begin{cor}\label{cor:S_kappa}
	The surfaces $S_{p_1,p_2}$ from Theorem \ref{THM} are in $\cS_{0}$.
\end{cor}
\begin{proof}
	By Corollary \ref{cor:S_k}, $S_{p_1,p_2}\in \cS_{\kappa}$ for some $\kappa$. Construction \ref{constr} provides a log smooth completion of $S_{p_1,p_2}$ obtained from a log smooth completion of $\C^{*}\times \C^{*}$ by two half-point attachments. Hence by Lemma \ref{lem:HPD_k}, $\kappa=\kappa(S_{p_1,p_2})=\kappa(\C^{*}\times \C^{*})=0$.
\end{proof}

We now state the main result of the theory of peeling. Recall that $D^{\#}=D-\Bk D$, and $\Bk D$ is a $\Q$-effective divisor with coefficients in $(0,1)$, supported on the twigs of $D$.

\begin{prop}[Construction of an almost minimal model]\label{prop:MMP}
	Let $(X,D)$ be a minimal log smooth completion of a smooth affine surface. Then there is a birational morphism
		\begin{equation}\label{eq:MMP}
		\psi\colon (X,D)\toin{\psi_{1}}(X_{1},D_{1})\toin{\psi_{2}}\dots \toin{\psi_{n}} (X_{n},D_{n})
		\end{equation}
	such that, for all $i\in \{1,\dots, n\}$:
	\begin{enumerate}
		\item\label{item:psi_hpd} $\psi_{i}$ is a half-point attachment.
		\item\label{item:X_i-smooth} $D_{i}=(\psi_{i})_{*}D_{i-1}$ is snc-minimal, $(X_{i},D_{i})$ is log smooth and $X_{i}\setminus D_{i}$ is an affine open subset of $X\setminus D$.
		\item\label{item:psi_kappa} $h^{0}(m(K_{X_{i}}+D_{i}))=h^{0}(m(K_{X}+D))$ for all $m\geq 0$. In particular, $\kappa(X_{i}\setminus D_{i})=\kappa(X\setminus D)$.
		\item\label{item:psi_min} If $\kappa(X\setminus D)\geq 0$ then $K_{X_{n}}+D_{n}^{\#}$ is nef. Otherwise, either $X_{n}$ admits a $\P^{1}$-fibration with $(K_{X_{n}}+D_n^{\#})\cdot F<0$ for a fiber $F$; or $-(K_{X_{n}}+D_{n})$ becomes ample after the contraction of $\Bk D_n$.
	\end{enumerate}
\end{prop}
\begin{proof}[Sketch of a proof]
	Let $\alpha\colon (X,D)\to (Y,D_{Y})$ be the contraction of $\Supp\Bk D$, i.e.\ of all admissible twigs of $T$. Then $\alpha^{*}(K_{Y}+D_{Y})=K_{X}+D^{\#}$. Indeed, for every component $T_{0}$ of $\Bk D$, we have $T_{0}\cdot (K_{X}+D-\alpha^{*}(K_{Y}+D_{Y}))=T_{0}\cdot(K_{X}+T_{0})+\beta_{D}(T_{0})=-2+\beta_{D}(T_{0})=T_{0}\cdot \Bk D$ by definition of bark, so $K_{X}+D-\alpha^{*}(K_{Y}+D_{Y})=\Bk D$ because $\Bk D$ is negative definite.
	
	In particular, $(Y,D_{Y})$ is log terminal. General theorems of Mori theory imply that either $(Y,D_{Y})$ is minimal, so \ref{item:psi_min} holds for $n=0$, or there is a log extremal curve $\ll\subseteq Y$ such that $\ll^{2}<0$ and $\ll\cdot (K_{Y}+D_{Y})<0$. Put $A=\alpha^{-1}_{*}\ll$. We claim that $A$ satisfies \eqref{eq:hp}.
	
	We have $0>\ll\cdot (K_{Y}+D_{Y})=A \cdot(K_{X}+D^{\#})$. Suppose $A\subseteq D$. Then $0>-2+\beta_{D}(A)-A\cdot \Bk D$, hence $A$ is a tip of $D-\Supp\Bk D$ and meets at most one component of $\Bk D$. This means that $A$ is a component of an admissible twig of $D$, a contradiction with the definition of $\Bk D$.
	
	Therefore, $A\not \subseteq D$, so $A\cdot D^{\#}>0$. Thus $A\cdot K_{X}<0$, which together with $A^{2}<0$ implies that $A$ is a $(-1)$-curve. Moreover, $A\cdot D^{\#}<-A\cdot K_{X}=1$, so $A$ meets $D$ only on $\Bk D$. Now a computation \cite[II.3.7.1(2)]{Miyan-OpenSurf}, which is rather complicated, but uses nothing more than negative definiteness of $A+\Bk D$, shows that $A$ meets each connected component of $D$ at most once. Because $X\setminus D$ is affine, it follows that $A\cdot D=1$, so $A$ satisfies \eqref{eq:hp}, as claimed.
	
	We now put $\psi_{1}=\phi_{A}\colon (X,D)\to(X_{1},D_{1})$, for some choice of $\phi_{A}$ such that $D_{1}$ is snc-minimal. Then for $i=1$, \ref{item:psi_hpd},\ref{item:X_i-smooth} hold by definition and \ref{item:psi_kappa} holds by Lemma \ref{lem:HPD_k}. Next, we replace $(X,D)$ by $(X_{1},D_{1})$ and repeat the procedure, which eventually ends since $\rho(X)$ drops after each step.
\end{proof}

\section{Proof of Theorem \ref{THM}}\label{sec:proof}
\setcounter{claim}{0}

Recall that $S_{p_1,p_2}\in \cS_{0}$ by Corollary \ref{cor:S_kappa}, and $S_{p_1,p_2}\not\cong S_{p_1',p_2'}$ for $\{p_1,p_2\}\neq \{p_1',p_2'\}$ by Corollary \ref{cor:st}\ref{item:distinct}. It remains to show that any $S\in \cS_{0}$ is isomorphic to some $S_{p_1,p_2}$. Let $(X,D)$ be a minimal log smooth completion of $S$. The idea of the proof is to view $\phi$ from Construction \ref{constr} as some almost minimalization \eqref{eq:MMP} of $(X,D)$. By Proposition \ref{prop:MMP} each intermediate pair $(X_{i},D_{i})$ in \eqref{eq:MMP} is a log smooth completion of an affine open subset of $X\setminus D$, of the same Kodaira dimension. Hence under the assumptions of Theorem \ref{THM}, $\kappa(X_{i}\setminus D_{i})=0$ and $\C[X_{i}\setminus D_{i}]$ is a UFD by Remark \ref{rem:open_UFD}. 

To reconstruct $(X,D)$ we need to understand each half-point attachment $\psi_{i+1}\colon (X_{i},D_{i})\to(X_{i+1},D_{i+1})$. This is done in Lemma \ref{lem:producing_Cst}, which is a version of \cite[Lemma III.4.4.3]{Miyan-OpenSurf}, with almost the same proof.

\begin{lem}\label{lem:producing_Cst}
	Let $S$ be a smooth affine surface. Assume that $\kappa(S)=0$ and $\C[S]$ is UFD. Let $(X,D)$ be a minimal log smooth completion of $S$ and let $\phi_{A}$ be a half-point attachment of some $A\subseteq X$. Then
	\begin{enumerate}
		\item\label{item:HP_Cst} $A\cap S\cong \C^{1}$ is a smooth fiber of some $\C^{*}$-fibration of $S$
		\item\label{item:HP_C} $\Exc \phi_{A}=[1,(2)_{k}]$, for some $k\geq 0$, is a twig of $D+A$. Its image does not lie on a rational twig of $(\phi_A)_{*}D$.
	\end{enumerate}
	\begin{figure}[htbp]
		\begin{tabular}{ccc}
			\begin{subfigure}[t]{0.4\textwidth}
				\includegraphics[scale=0.34]{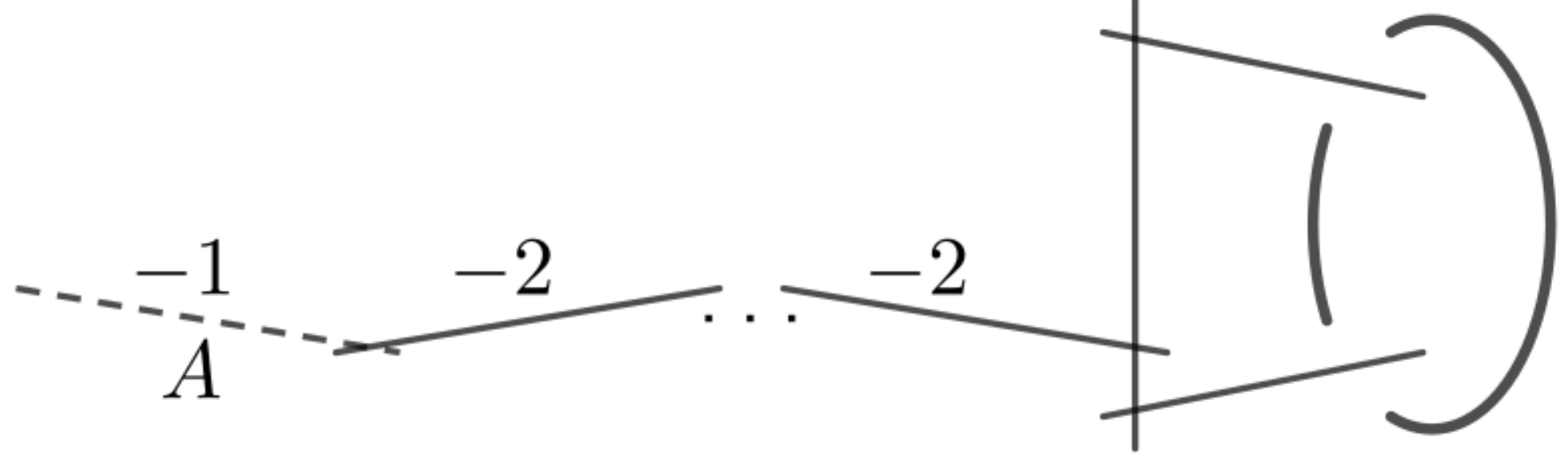}
			\end{subfigure}
			&
			\raisebox{1cm}{$\xrightarrow{\quad \phi_{A}\quad}$}
			&
			\begin{subfigure}[t]{0.2\textwidth}
				\includegraphics[scale=0.34]{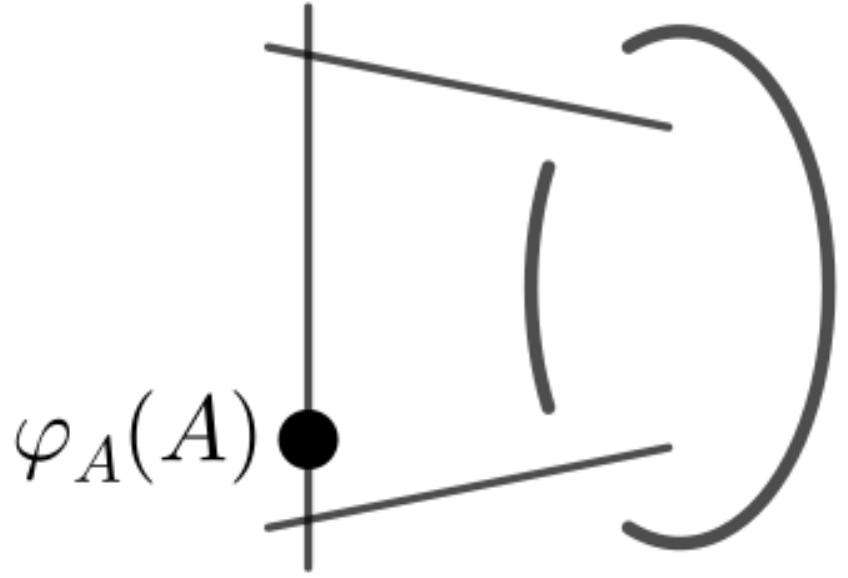}
			\end{subfigure}
		\end{tabular}
		\caption{A half-point attachment to $S$ if $\C[S]$ is a UFD and $\kappa(S)=0$, see Lemma \ref{lem:producing_Cst}.}
		\label{fig:HP}
	\end{figure}
\end{lem}
\begin{proof}
\ref{item:HP_Cst} Because $\C[S]$ is a UFD, $A|_{S}$ is a divisor of zeros of some $f\in \C[S]$. Let $F_{\mathrm{aff}}$ be some general fiber of $f$. It is reduced and irreducible because $A|_{S}$ is. Put $S'=S\setminus A$, $U'=f(S')\subseteq \C^{*}$. By Lemma \ref{lem:HPD_k} $\kappa(S')=\kappa(S)=0$, so the Kawamata addition theorem \cite[Theorem 11.15]{Iitaka_AG} gives $0=\kappa(S')\geq \kappa(F_{\mathrm{aff}})+\kappa(U')\geq \kappa(F_{\mathrm{aff}})+\kappa(\C^{*})=\kappa(F_{\mathrm{aff}})\geq 0$, where the last inequality follows from the Iitaka Easy Addition theorem. Hence $\kappa(F_{\mathrm{aff}})=0$, which implies that $\kappa(F_{\mathrm{aff}})\cong \C^{*}$ because $F_{\mathrm{aff}}$ is smooth and affine. Thus $f$ is a $\C^{*}$-fibration of $S$, with $A|_{S}$ as one of its fibers.

\ref{item:HP_C} Let $T$ be a maximal admissible twig of $D$ meeting $A$ in its first tip; put $T=0$ if there is no such. Since $S$ is affine, $D\neq T$. Let $C$ be the component of $D-T$ meeting $A+T$. 
Let $\tau\colon \hat{X}\to X$ be a minimal sequence of blowups over $D$ such that the $\C^{*}$-fibration from \ref{item:HP_Cst}  extends to a $\P^{1}$-fibration of $\hat{X}$; let $F$ be a fiber containing $\hat{A}\de \tau^{-1}_{*}A$. Put $\hat{D}=(\tau^{*}D)\redd$. Note that $\tau$ does not touch $\hat{A}$: indeed, otherwise the only base point of $\tau^{-1}$ is the unique common point of $A$ and $D$, so  $F\subseteq \tau^{*}A$ is negative definite, which is false.

If a $(-1)$-curve $L\subseteq \hat{D}\vert$ is superfluous in $\hat{D}$, then $L$ is not a component of a rational twig of $\hat{D}$. Indeed, $L\not\subseteq \Exc\tau$ since $\tau$ is minimal, and $\tau(L)^{2}\neq -1$ since $D$ is snc-minimal; hence $\tau(L)^{2}\geq 0$, and the claim follows from Lemma \ref{lem:rational-twigs-are-admissible}. Let $\sigma$ be the contraction of all vertical superfluous $(-1)$-curves in $\hat{D}$ and its images. By our claim, $\sigma$ does not touch rational twigs of $\hat{D}$. Put $D'=\sigma_{*}\hat{D}$, $F'=\sigma_{*}F$, $A'\de \sigma_{*}\hat{A}$.

If $\sigma$ touches $\hat{A}$ then $T=0$ and $\tau^{-1}_{*}C\subseteq \Exc\sigma$ does not lie in a rational twig of $\hat D$, so the result follows. Assume $\sigma$ does not touch $\hat A$, so $A'$ is a $(-1)$-curve of multiplicity one in a fiber $F'$. Now,  $F'-A'$ contains another $(-1)$-curve, say $C'$. By definition of $\sigma$, $\beta_{D'}(C')\geq 3$. Since $\beta_{F'\redd}(C')\leq 2$, $C'$ meets $D'\hor$. By Lemma \ref{lem:untwisted}, $D'\hor$ consists of two $1$-sections, so $C'$ has multiplicity one in $F'$. Hence $\beta_{F'\redd}(C')=1$, so $C'\cdot D'\hor=2$, which implies that $(F'-C')\cdot D'\hor=0$, and that $C'$ is a unique $(-1)$-curve in $F'-A'$. Therefore, $F'=[1,(2)_{k},1]$ for some $k\geq 0$. In other words, $T'\de F'-A'-C'=[(2)_{k}]$ is zero or a maximal twig of $D'$.  Applying Lemma \ref{lem:rational-twigs-are-admissible} to image of $D'$ after contraction of $A'+T'$, we infer that $C'$ does not lie in a rational twig of $D'-T'$. 

Since $\sigma$ does not touch rational twigs of $D'$, $\hat{T}\de \sigma^{-1}_{*}T'$ is zero or a maximal twig of $\hat{D}$; and the component $\hat{C}\de \sigma^{-1}_{*}C'$ meeting it does not lie in a rational twig of $\hat{D}-\hat{T}$. Thus $\tau_{*}\hat{T}=T$, $\tau_{*}\hat{C}=C$: indeed, otherwise $\tau$ touches $\hat{C}$, so $C^{2}\geq 0$, and $C$ lies in an admissible twig of $D$, which is impossible by Lemma \ref{lem:rational-twigs-are-admissible}. Therefore, $T=[(2)_{k}]$, and $C$ does not lie in any rational twig of $D$, as claimed.
\end{proof}

Our next lemma is a standard result concerning curves of canonical type, see eg.\ \cite[I.3.3.1]{Miyan-OpenSurf}. It will give a description of the almost minimal model of $(X,D)$, or, more precisely, the pair obtained from $(X,D)$ after all possible half point attachments, called a \emph{strongly minimal model} in \cite[II.4.9]{Miyan-OpenSurf}.

Recall that $\F_{m}$ for an integer $m\geq 0$ denotes the $m$-th Hirzebruch surface $\P(\O_{\P^{1}}\oplus \O_{\P^{1}}(m))$.

\begin{lem}[strongly minimal models]\label{lem:O}
	Let $S$ be a smooth rational affine surface such that $\C[S]$ is a UFD. Let $(X,D)$ be a minimal log smooth completion of $S$. Assume that $K_{X}+D\equiv 0$ and that all $(-1)$-curves on $X$ are contained in $D$. Then  one of the following holds:
		\begin{enumerate}
			\item\label{item:Fm_4} $X=\F_{m}$ for some $m\neq 1$ and  $D=\lc0,m,0,-m\rc$,
			\item\label{item:Fm_3} $X=\F_{m}$ for some $m\neq 1$ and  $D=\lc0,m,-m-2 \rc$,
			\item\label{item:triangle} $X=\P^{2}$ and  $D=\lc -1,-1,-1 \rc$, i.e.\ $D$ is a triangle,
			\item\label{item:conic} $X=\P^{2}$ and  $D=\lc -1,-4\rc $, i.e.\ $D$ is a sum of a line and a conic which meet at two points.
		\end{enumerate}
\end{lem}
\begin{proof}
	Let $\tau\colon (X,D)\to (X',D')$ be a contraction of some $(-1)$-curves in $D$ and its images, such that all singularities of $D'$ are ordinary double points. Since $D$ is snc-minimal, $\Bs\tau^{-1}\subseteq \Sing D$. Let $A'$ be a $(-1)$-curve in $X'$. Then $A'\cdot D'=-A'\cdot K_{X'}=1$. If $A'\not\subseteq D'$ then $A'\cap \Sing D'=\emptyset$, so $\tau$ does not touch $A\de\tau^{-1}_{*}A$. Therefore, $A\subseteq X$ is a $(-1)$-curve not contained in $D$, contrary to our assumption. Hence $A'\subseteq D'$, so $\beta_{D'}(A')=A'\cdot D'-A^2=2$. Since no two components of $D'$ are tangent, $\#A'\cap (D'-A')=2$. After the contraction of $A'$, the singularities of the image of $D'$ are still at most ordinary double points. 
	
	Thus we can assume that $X'$ has no $(-1)$-curves, i.e.\ $X'=\P^2$ or $\F_{m}$ for some $m\neq 1$. By Lemma \ref{lem:S0}\ref{item:UFD}, the components of $D$ generate $\Pic(X)$, so the components of $D'$ generate $\Pic(X')$, too. 
	
	Assume first that $X'=\P^{2}$. Then $D'\equiv -K_{\P^{2}} \equiv 3\cdot [\mbox{line}]$. Because the components of $D'$ generate $\Pic(\P^{2})$, $D'$ is reducible. In particular, $D'$ has no singular component, so $\tau=\id$ and $D$ is as in \ref{item:triangle} or \ref{item:conic}.
	
	Assume now that $X'=\F_{m}$. Then $2=\rho(X')\leq \#D'$. Because $X'\setminus D'$ is affine, $D'$ is connected, so $\beta_{D'}(C)\geq 1$ for every component $C$ of $D'$. Hence $0=C\cdot (K_{X'}+D')=2p_{a}(C)-2+\beta_{D'}(C)\geq 2p_{a}(C)-1$, It follows that $C$ is smooth, so $\tau=\id$. Moreover, $C$ is rational and $\beta_{D}(C)=2$, so $D$ is circular. 
	
	Denote by $F,T\in\Pic(\F_{m})$ the classes of the fiber of $\F_{m}$ and of the section with $T^{2}=-m$, which is unique if $m\neq 0$. We have $F\cdot D=-F\cdot K_{\F_{m}}=2$, so by Lemma \ref{lem:untwisted} $D\hor$ consists of two $1$-sections, say $H_{i}\equiv a_{i}F+T$ for $i\in\{1,2\}$. We can assume $a_1\geq a_2$. Note that either $a_{i}=0$ or $0\leq H_{i}\cdot T=a_{i}-m$, so $a_{i}\geq m$. Hence $a_1\geq m$ and either $a_2\geq m$ or $H_2\equiv T$. Let $v$ be the number of fibers in $D$. Because all fibers are disjoint and meet $H_{1}$ and $H_{2}$, from the fact that $\beta_{D}(H_j)=2$, $j\in \{1,2\}$ we infer that 
	\begin{equation}\label{eq:Hirz}
	2-v=H_{1}\cdot H_{2}=a_{1}+a_{2}-m.
	\end{equation}
	Suppose $v=0$. Then because $H_{1},H_{2}$ generate $\Pic(\F_{m})=\Z[F+T]$, we have $\det \left[ \begin{smallmatrix} a_{1} & 1 \\ a_{2} & 1 \end{smallmatrix}\right]=\pm 1$, so $a_{1}-a_{2}=1$. Now \eqref{eq:Hirz} gives $a_{2}=\tfrac{1}{2}(m+1)$, so $2\nmid m$, hence $m\geq 3$ and $0\neq a_2<m$, a contradiction
	
	Thus $v\in \{1,2\}$. It follows that $H_{2}\equiv T$, because otherwise $a_{1},a_{2}\geq m$ and \eqref{eq:Hirz} gives $1\geq  a_{1}+a_{2}-m\geq m$, so $m=0$ and $1\geq a_{1}+a_{2}\geq 2$, which is false. Hence $a_{2}=0$, $H_{2}^{2}=-m$ and by \eqref{eq:Hirz} $a_{1}=m+2-v$, so $H_{1}^{2}=m+4-2v$. We conclude that \ref{item:Fm_4} holds if $v=2$ and \ref{item:Fm_3} holds if $v=1$.
\end{proof}

\begin{rem*}
  Lemma \ref{lem:O} will imply that almost minimal models for surfaces in $\cS_{0}$ are of type $(O)$ in \cite[8.70]{Fujita-noncomplete_surfaces}, cf.\ \cite[Proposition 1.5(1)]{Kojima_k0}. For definition and basic properties of this type see \cite[8.8, 8.12]{Fujita-noncomplete_surfaces}. Parts \ref{item:Fm_4}-\ref{item:conic} correspond to cases 4, 6, 8 of 9 in \cite[Theorem 3.1]{Kojima_k0}. We will not use these consequences.	
\end{rem*}

\begin{proof}[Proof of Theorem \ref{THM}]
	Let $(X,D)$ be a minimal log smooth completion of $S\in \cS_{0}$ and let 
	\begin{equation*}
	\psi\colon (X,D)\to (X_{n},D_{n})
	\end{equation*}
	be an almost minimalization as in \eqref{eq:MMP}. By Proposition \ref{prop:MMP}\ref{item:psi_min} $\kappa(X_n\setminus D_n)=\kappa(X\setminus D)=0$, so \cite[II.6.2.1]{Miyan-OpenSurf} implies that $K_{X_{n}}+D_{n}^{\#}\equiv 0$, that is, $K_{X_{n}}+D_{n}\equiv \Bk D_{n}$. By Proposition \ref{prop:MMP}\ref{item:psi_hpd} $\psi$ is a sequence of half-point attachments, so by Lemma \ref{lem:producing_Cst}\ref{item:HP_C} no base point of $\psi^{-1}$ lies on an admissible twig of $D_n$, i.e.\ $\Bs\psi^{-1}\cap \Supp\Bk D_{n}=\emptyset$. Hence $\psi^{-1}_{*}\Bk D_{n}=\psi^{*}\Bk D_{n}\equiv \psi^{*}(K_{X_{n}}+D_{n})$, which by Lemma \ref{lem:S0} writes uniquely as a sum of components of $D$, with integral coefficients. But $\psi^{-1}_{*}\Bk D_{n}$ is a sum of components of $D$ with coefficients in $[0,1)$. Hence $\psi^{-1}_{*}\Bk D_{n}=0$, which implies that
	\begin{equation}\label{eq:K+D=0}
	K_{X_{n}}+D_{n}\equiv 0.
	\end{equation}
	We claim that there is a sequence of half-point attachments
	\begin{equation}\label{eq:claim}
	\phi\colon (X,D)\to (Z,D_{Z}) \mbox{ such that } (Z,D_Z) \mbox{ is as in Lemma \ref{lem:O}}.
	\end{equation}
	Indeed, let $(X_{n},D_{n})\to (Z',D_{Z'})$ be any sequence of half-point attachments. By Remark \ref{rem:open_UFD},  $\C[Z'\setminus D_{Z}']$ is a UFD, and $K_{Z'}+D_{Z'}\equiv 0$  by \eqref{eq:K+D=0}. Let $A\subseteq Z'$ be a $(-1)$-curve, $A\not\subseteq D_{Z'}$.  Then $A\cdot D_{Z'}=-A\cdot K_{Z'}=1$, so $A$ satisfies \eqref{eq:hp}. After a half-point attachment of $A$, the image of $Z'$ has smaller rank, so after finitely many half-point attachments, we get $(Z,D_{Z})$ with no $(-1)$-curves off $D_Z$, as needed.
	
	We need to show that, for a suitable log smooth completion $(X,D)$ of $S$, the morphism $\phi$ from \eqref{eq:claim} can be chosen as in Construction \ref{constr}. Denote by $A_{1},\dots, A_{r}$ the curves not contained in $D$ which are contracted by $\phi$ (in particular, $r \geq n$). Write $\phi(A_{i})=\{b_{i}\}\subseteq Z$ for $i\in \{1,\dots, r\}$. By Lemma \ref{lem:producing_Cst}\ref{item:HP_C} each $b_{i}$ lies on exactly one component $C_{i}$ of $D_{Z}$. Lemma \ref{lem:S0} gives $\#D=\rho(X)$, so
	\begin{equation*}
	\#D_{Z}-\rho(Z)=\#D+r-\rho(X)=r,
	\end{equation*}
	hence $r=2$ if $(Z,D_Z)$ is as in Lemma \ref{lem:O}\ref{item:Fm_4},\ref{item:triangle} and $r=1$ if $(Z,D_Z)$ is as in Lemma \ref{lem:O}\ref{item:Fm_3},\ref{item:conic}. 
	
	Let $R_{1},\dots, R_{r}$ be a basis of $\ker[\Z[D_{Z}]\to \Pic(Z)]$, and let $a_{ij}$ be the coefficient of $C_{i}$ in $R_{j}$. We have $\Pic(X)=\Pic(Z)\oplus\Z[\Exc\phi]=\Z[\phi^{*}D_{Z}]/(\phi^{*}R_{1},\dots,\phi^{*}R_{i})$. Because by Lemma \ref{lem:producing_Cst}\ref{item:HP_C} each $A_{i}$ has multiplicity $1$ in $\phi^{*}C_{i}$ and $0$ in $\phi^{*}(D_{Z}-C_{i})$, Lemma \ref{lem:S0}\ref{item:units} gives 
	\begin{equation}\label{eq:det}
	\det [a_{ij}]_{1\leq i,j\leq r}=\pm 1.
	\end{equation}

	Now we treat each of the cases \ref{item:Fm_4}--\ref{item:conic} from Lemma \ref{lem:O} separately. We would like to reduce everything to case \ref{item:Fm_4} with $m=0$. So in each of the remaining cases, we  construct $\sigma\colon Z\map Z'$, which is an isomorphism on $Z\setminus D_{Z}$ and on some neighborhood of $b_{1},\dots, b_{r}$; and $\phi''\colon (Z',\sigma_{*}D_{Z})\to (Z'',D_{Z''})$, which is an identity or a half-point attachment, such that $(D_{Z''},D_{Z''})$ is as in one of the cases not considered yet. Then we lift $\sigma$ so that the following diagram is commutative:
	\begin{equation*}
	\xymatrix{
		(X,D) \ar[r]^{\phi} \ar@{-->}[d]^{\sigma'} & (Z,D_{Z}) \ar@{-->}[d]^{\sigma} & \\	
		(X',D') \ar[r]^{\phi'} & (Z',D_{Z'}) \ar[r]^{\phi''} & (Z'',D_{Z''})
	}
	\end{equation*}
	where $D'$, $D_{Z'}$, $D_{Z''}$ are reduced total transforms of $D$, and $(X',D')$ is another log smooth completion of $S$. Eventually, we replace $\phi$ with $\phi''\circ \phi'$, which is another sequence of  half-point attachments. 
	
	Consider case \ref{item:conic}, i.e.\ $r=1$, $D_{Z}$ is a sum of a line $L$ and a conic $C$, and $\#L\cap C=2$. We have $R_{1}=C-2L$, so by \eqref{eq:det} $b_{1}\in C$. Define $\sigma$ as a blow up at a point of $C\cap L$ followed by a flow (see Section \ref{sec:standard}) on the proper transform of $L$. Then $(Z',D_{Z}')$ is as in case \ref{item:Fm_3}, so we put $\phi''=\id$.
	
	Consider case \ref{item:triangle} and let $L\subseteq D_{Z}$ be the line not containing $b_{1},b_{2}$. As before, we define $\sigma$ as a blowup at $L\cap (D_{Z}-L)$, followed by a flow on the proper transform of $L$. Then $(Z',D_{Z}')$ is as in case \ref{item:Fm_4}.
	
	Consider case \ref{item:Fm_3}. Write $D_{Z}=F+T+H$, where $F^{2}=0$, $T^{2}=-m$ and $H^{2}=m+2$. We can assume that $b_{1}\in F$: indeed, otherwise after $\pm C_{1}^{2}$ flows on $F$ we obtain a pair with that property. Now $r=1$ and $R_{1}=(m+1)F+T-H$, so \eqref{eq:det} gives $m=0$. Let $\sigma$ be a flow on $T$ such that $(\sigma_{*}F)^{2}=1$, followed by a blowup at $\sigma(F\cap H)$. Let now $\phi''$ be the contraction of a proper transform of a member of $|T|$ passing through $F\cap H$. The resulting pair is as in \ref{item:Fm_4}.
	
	Eventually, consider case \ref{item:Fm_4}. Write $D_{Z}=F_{1}+T+F_{2}+H$, where $F_{1}^{2}=F_{2}^{2}=F_{1}\cdot F_{2}=0$ and as before $T^{2}=-m$, $b_{1},b_{2}\not\in T$. We have $r=2$ and $R_{1}=F_{1}-F_{2}$, $R_{2}=mF_{2}+T-H$, so \eqref{eq:det} implies (after renaming $b_{i}$, $F_{i}$ if necessary) that $b_{1}\in F_{1}$, $b_{2}\in H+T$. After $m$ flows on $F_{2}$ we can assume that $m=0$, so $Z\cong \P^{1}\times \P^{1}$. We can choose coordinates on $Z$ such that $F_{1}=\ll_{1,0}$, $F_{2}=\ll_{1,\infty}$, $b_2\in H=\ll_{2,0}$, $T=\ll_{2,\infty}$, $b_{1}=(0,1)$, $b_{2}=(1,0)$. Now by Lemma \ref{lem:producing_Cst}\ref{item:HP_C}, $\phi\colon (X,D)\to (Z,D_{Z})$ is as in Construction \ref{constr}.
\end{proof}

\section{Diffeomorphism type of $S_{p_{1},p_{2}}$}\label{sec:diff}

In this section we prove Theorem \ref{thm:diff}. To this end, we need to translate Construction \ref{constr} to the language of Kirby calculus, where blowing up corresponds to attaching handles. We give a quick overview of handlebodies and Kirby diagrams in Section \ref{sec:Kirby_overview}.  In Proposition  \ref{prop:handles} we show that Construction \ref{constr} amounts to attaching two $2$-handles to $\bT^{2}\times \bD^{2}$, whose interior is identified with $\C^{*}\times \C^{*}$, see Figure \ref{fig:handles}. Theorem \ref{thm:diff} follows by sliding its $2$-handle over the attached ones. The boundary $M$ of the obtained $4$-manifold is a $0$-surgery on the knot $K$ from Figure \ref{fig:knot}. It is interesting to note that this is one of the few \emph{exceptional} surgeries on $K$, classified in  \cite{BW_surgery}. More precisely, while all but finitely many surgeries on $K$ are hyperbolic, $M$ contains incompressible tori. Decomposition of $M$ along these tori is described in Proposition \ref{prop:JSJ}. While interesting in its own right, it gives another way to distinguish $S_{p_1,p_2}$, for different $\{\deg p_{1}, \deg p_{2}\}$. 
\smallskip

We begin with the following observation.
\begin{prop}\label{prop:htp}
	For every $p_{1},p_{2}$, $S_{p_1,p_2}$ is homotopically equivalent to $\bS^{2}$. 
\end{prop}
\begin{proof}	
	Viewing $S=S_{p_1,p_2}$ as an iterated affine modification of $\C^{2}$, we infer from \cite[Lemma 3.4]{Kal_exotic-measures} that $\pi_{1}(S)=\pi_{1}(\C^{2})=\{1\}$. This modification replaces $\{x_1x_2=0\}\subseteq \C^{2}$, which is contractible, by the affine part of $A_1+A_2$, which is a disjoint union of two copies of $\C^{1}$. Hence $\etop(S)=\etop(\C^{2})-1+2=2$. Because $S$, being Stein, has a homotopy type of a CW-complex of (real) dimension two \cite[11.2.6]{GS_Kirby}, it follows that $H_{i}(S,\Z)=\Z$ for $i=2$ and $0$ for $i>2$. Thus $S$ has weak homotopy type of $\bS^{2}$. Applying the Whitehead theorem to the continuous map $\bS^{2}\to S$ given by a generator of $\pi_{2}(S)=H_{2}(S;\Z)=\Z$, we infer that $\bS^{2}\htp S$, as claimed.
\end{proof}

\subsection{Overview of Kirby calculus}\label{sec:Kirby_overview}

We now briefly recall the language of Kirby diagrams of $4$-manifolds, for a complete treatment see \cite[\S 4-5]{GS_Kirby}. A ($4$-dimensional) $k$-handle $h$ for $k\in \{0,\dots, 4\}$ is a copy of $\bD^{k}\times \bD^{4-k}$ attached to a boundary of a $4$-manifold $V$ via an embedding $\phi\colon \d\bD^{k}\times \bD^{4-k}\to \d V$. The images of $\d\bD^{k}\times 0$ and $0\times \d \bD^{4-k}$ are called the \emph{attaching} and \emph{belt} spheres of $h$, the image of $\bD^{k}\times 0$ is called the \emph{core} of $h$, see  \cite[Fig.\ 4.1]{GS_Kirby}. Smoothing corners gives a $4$-manifold $V\cup_{\phi}h$ which depends only on the attaching sphere of $h$ and, if $k=2$, on the integer called \emph{framing} of $h$, which equals the linking number of its attaching circle and its push-off along a transverse vector field \cite[4.1, 4.5]{GS_Kirby}. For example, blowing up is the same as attaching a $2$-handle with framing $-1$ \cite[p.\ 150]{GS_Kirby}.

Any $4$-manifold $V$ can be obtained from $\bD^{4}$ by attaching handles of increasing index. There are two ways of modifying  such handle-decomposition without changing the diffeomorphism type of $V$. First, if the attaching sphere of a $k$-handle meets the belt sphere of a $(k-1)$-handle transversally in a single point, these two handles can be \emph{canceled} \cite[4.2.9]{GS_Kirby}. Second, a $2$-handle $h$ can be \emph{slid over} a $2$-handle $h_{0}$ by pushing the attaching circle of $h$ through the belt circle of $h_0$ \cite[4.2.10]{GS_Kirby}. We use the same letter to denote a handle before and after the slide.

A handle-decomposition of $V$ is encoded by a \emph{Kirby diagram}, which is a (decorated) link in $\bS^{3}$, i.e.\ in the boundary of the initial $0$-handle. A $2$-handle is represented by its attaching circle with a framing coefficient. Attaching a $1$-handle to $V$ is the same as drilling a tubular neighborhood of a properly embedded $\bD^{2}\subseteq V$ \cite[p.\ 168]{GS_Kirby}, so it is represented by a circle $\d\bD^{2}\subseteq \d V$, with a dot to distinguish it from a $2$-handle. All dotted circles in a Kirby diagram form an unlink. Once $1$- and $2$-handles are attached, there is usually (e.g.\ if $V$ is closed or $\d V$ is connected and $V$ is simply connected, which is our case) a unique way to attach the remaining ones \cite[p.\ 148]{GS_Kirby}, so they are not drawn. We use the same letters for handles and corresponding knots in the diagram. 

For example, a diagram consisting of one unknot with coefficient $e$ represents a $\bD^{2}$-bundle over $\bS^{2}$ of Euler number $e$ \cite[Fig.\ 4.20]{GS_Kirby}. The diagram in Figure \ref{fig:handles} without the leftmost and rightmost circle represents $\bT^{2}\times \bD^2$, see \cite[Fig.\ 4.36]{GS_Kirby}. Indeed, attaching the $1$-handles gives $(\bT^{2}\setminus \{\mbox{disk}\}) \times \bD^{2}$, whose fundamental group is freely generated by loops based in the initial $0$-handle and going along the cores of the $1$-handles. The $2$-handle, attached along their commutator, caps off the puncture.

Handle cancellation and sliding is represented on the Kirby diagram as follows. First, if a dotted circle meets only one non-dotted circle, once, then the corresponding $1$- and $2$-handle cancel \cite[Fig.\ 5.38]{GS_Kirby}. Second, sliding a $2$-handle $t$ over $t_0$ replaces $t$ by its band-sum with a parallel copy $\tilde{t}_0$ of $t_0$ \cite[Fig.\ 5.4]{GS_Kirby}. More precisely, choose an arc in $t$ and $\tilde{t}_0$ and connect their endpoints such that the obtained \enquote{square} bounds a band which is disjoint from the rest of the diagram. Recall that $\lk(t_0,\tilde{t}_0)$ is the framing of $t_0$. The new framing of $t$ is the sum of framings of $t$ and $t_0$ plus $2\lk(t,t_0)$ \cite[p.\ 142]{GS_Kirby}. It is important to remember that the orientations of $t$ and $\tilde{t}_0$ should match. We denote by $-t$ the knot $t$ taken with opposite orientation. The components of a standard Hopf link drawn as in the left of Figure \ref{fig:added_torus} are always meant to have counterclockwise orientation and linking number $+1$. 

\subsection{Handle decompositions and proof of Theorem \ref{thm:diff}}\label{sec:diff_proof}

\begin{notation}\label{not:d}
	Let $S=S_{p_{1},p_{2}}$ be as in Construction \ref{constr}. For $j\in \{1,2\}$ put $d_{j}=\deg p_{j}+1$, $T_{j,d_j}=L_{j,0}$, and denote by $T_{j,i}$, $i\in \{1,\dots, d_{j}-1\}$ the $i$-th component of the twig $T_{j}$ meeting $A_{j}$, see Figure \ref{fig:Skl}. Let $\mathrm{Tub}(D)$ be the tubular neighborhood of $D$ in $X$ constructed in \cite{Mumford-surface_singularities}. Put $M=\d \mathrm{Tub}(D)$ as in Section \ref{sec:graph_manifolds} and $V=X\setminus \mathrm{int}\, \mathrm{Tub}(D)$. Clearly, $S\diff \mathrm{int}\,  V$. We may and will  assume $d_1\leq d_2$.
\end{notation}

\begin{prop}\label{prop:handles}
	The $4$-manifold $V$ is obtained from $\bT^{2} \times \bD^2$ by attaching two $2$-handles with framings $-d_{1}$, $-d_{2}$ along the standard generators of $\pi_{1}(\bT^{2})$. Its Kirby diagram is given by Figure \ref{fig:handles}.
\end{prop}
\begin{proof} 
	Let $\phi$ be as in Construction \ref{constr}. Let $\bD\subseteq \C^{2}\subseteq \P^1\times \P^1$ be a closed disk of radius $2$ centered at $(0,0)$, so that the base points $(0,1)$, $(1,0)$ of $\phi^{-1}$ belong to $\mathrm{int}\, \bD$. Put $X_0\de \phi^{-1}(\bD)\subseteq X$, $\phi_{0}\de \phi|_{X_0}\colon X_0\to \bD$ and $\ll_{i,j}'=\ll_{i,j}\cap \bD$. Figure \ref{fig:link} shows  $\d \phi_{0}^{-1}(\ll_{1,0}' \cup \ll_{2,0}')$, where the attaching circles of the $2$-handles are drawn together with their framings.
	 	 \begin{figure}[ht]
	 	 	\includegraphics[scale=0.35]{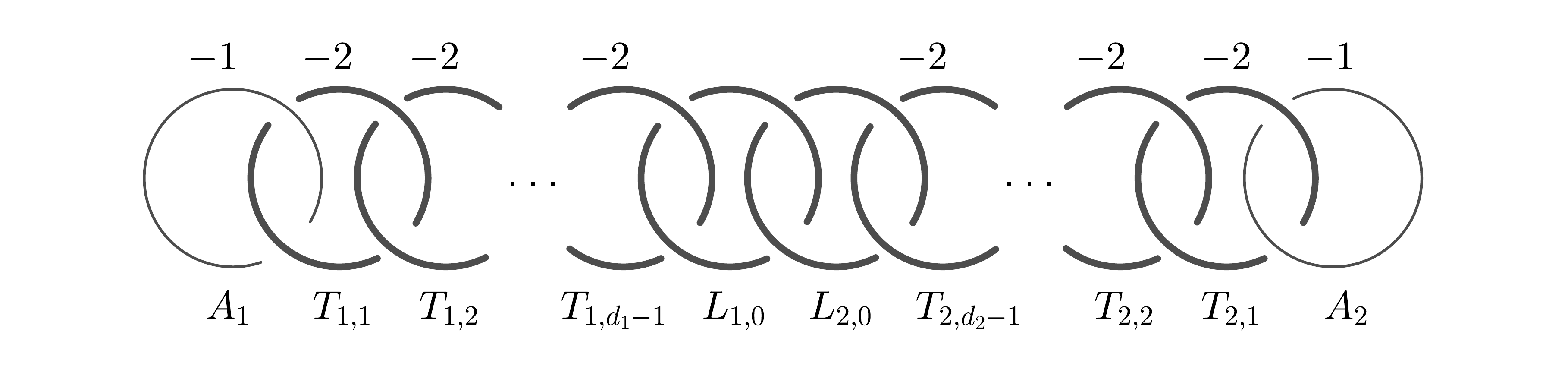}
	 	 	\caption{Underlying disks of $L_{1,0}$, $L_{2,0}$ and handles attached to $\C^{2}$ by $\phi_{0}$.}
	 	 	\label{fig:link}
	 	 \end{figure}
	 	 	 
	 To construct $V$, we need to remove from $X_0$ a tubular neighborhood of $\phi_{0}^{-1}(\ll_{1,0}' \cup \ll_{2,0}')\setminus (A_{1}\cup A_{2})$. To do this, we first remove from the $0$-handle all disks whose boundaries are drawn in Figure \ref{fig:link} in bold, and then we remove the cores of the corresponding $2$-handles (everything suitably thickened).
	 
	 By \cite[p. 214]{GS_Kirby}, the first step gives a Kirby diagram obtained from Figure \ref{fig:link} as follows. Replace each circle corresponding to $T_{j,i}$ for $j\in \{1,2\}$, $i\in \{1,\dots, d_{j}\}$ by a dotted circle $t_{j,i}'$, and each twist corresponding to $T_{j,i+1}\cap T_{j,i}$ (resp.\ $L_{0,1}\cap L_{0,2}$) by a $0$-framed $2$-handle $h_{j,i}$ (resp.\ $h$), linked with the dotted circles as in Figure \ref{fig:added_torus}. We denote by $a_{j}$ the $2$-handle corresponding to $A_{j}$.

	 	 \begin{figure}[ht]
	 	 	\includegraphics[scale=0.35]{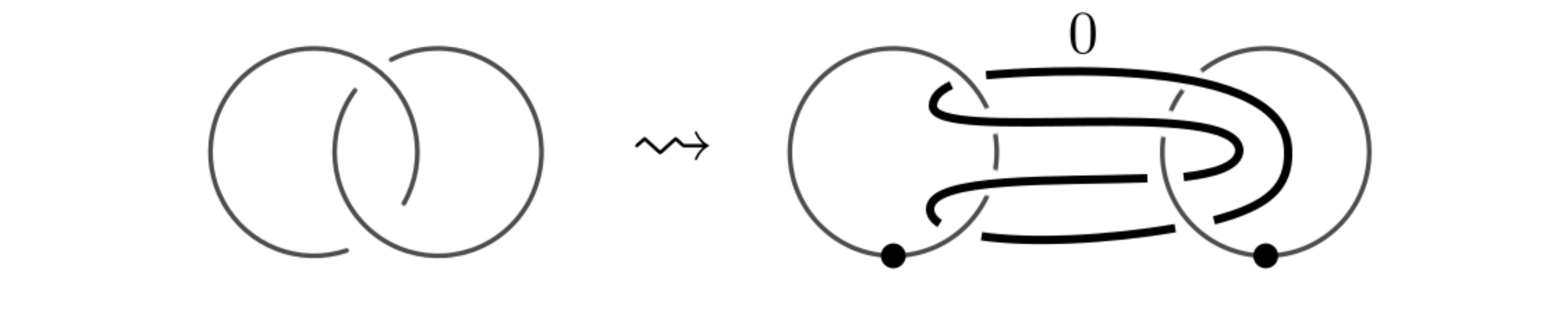}
	 	 	\caption{Intersection of two removed disks, see \cite[Fig.\ 6.27]{GS_Kirby}.}
	 	 	\label{fig:added_torus}
	 	 \end{figure}
	By \cite[p.\ 224]{GS_Kirby}, the next step amounts to attaching a $2$-handle along a parallel copy of each removed circle, together with a certain $3$-handle. In our case, for each $T_{j,i}$, $j\in \{1,2\}$, $i\in \{1,\dots, d_{j}-1\}$ we attach a $(-2)$-framed $2$-handle $t_{j,i}$ such that $\lk(t_{j,i},t_{j,i}')=-2$ and $\lk(t_{j,i},t_{j,i+1}')=\lk(t_{j,i},t_{j,i-1}')=1$, where we put $t_{j,0}'=a_{j}$. The resulting Kirby diagram is shown in  Figure \ref{fig:V} (where we only draw the part corresponding to $L_{1,0}+L_{2,0}+T_{2}+A_{2}$ and skip the subscript \enquote{$2,$}: the other part is analogous). 
\begin{figure}[ht]
	\includegraphics[scale=0.32]{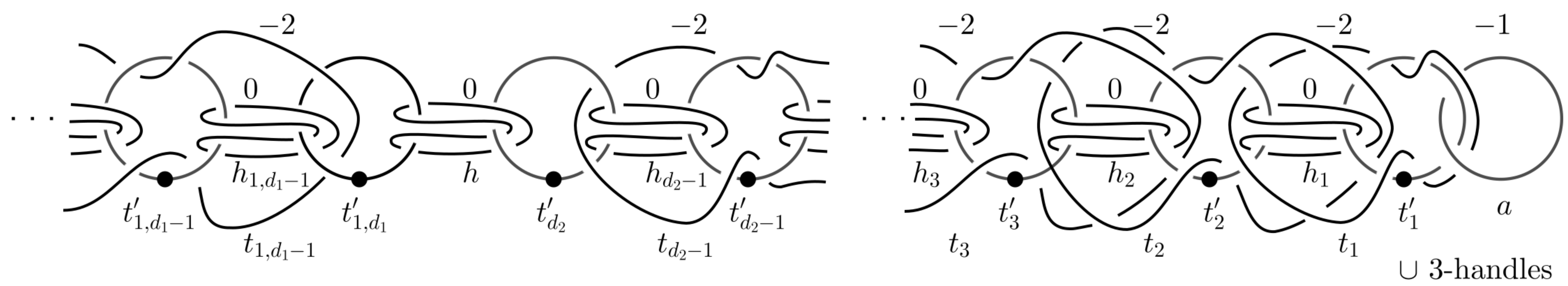}
	\caption{Right half of the Kirby diagram for $V=X\setminus \mathrm{Tub}(D)$.}
	\label{fig:V}
\end{figure}
	
	We will simplify this diagram by handle slides. First, we slide $t_{j,1}$ over $a_{j}$, see Figure \ref{fig:first_slide}.   After this slide, $t_{j,1}$ unlinks from $a$, meets $t_{j,1}'$ once with $\lk(t_{j,1},t_{j,1}')=-1$ and changes framing to $-1$, while the rest of the diagram does not change. 
	\begin{figure}[htbp]
		\begin{tabular}{ccccc}
			\begin{subfigure}[t]{0.22\textwidth}
				\includegraphics[scale=0.3]{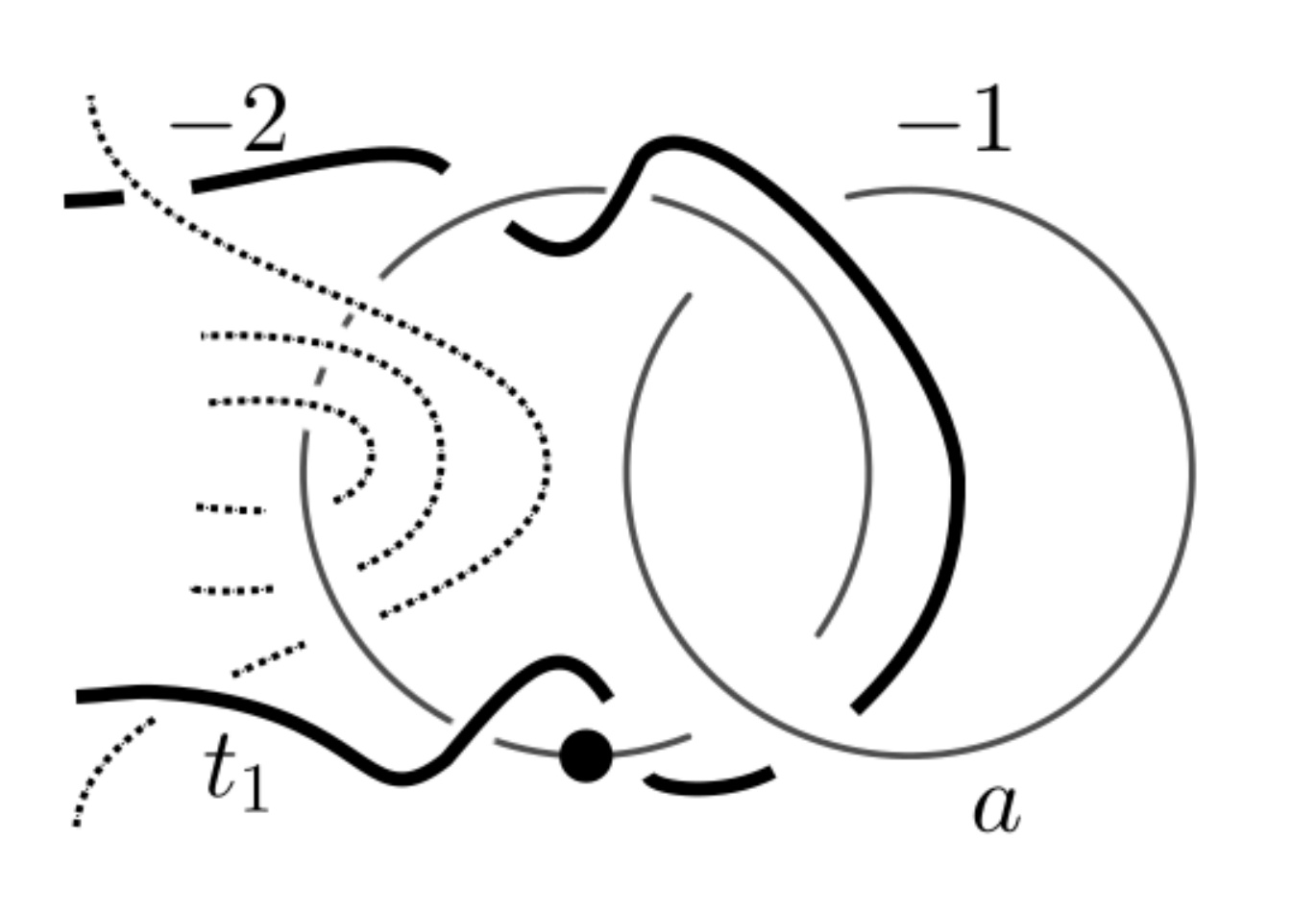}
			\end{subfigure}
			&
			\raisebox{1.2cm}{\scalebox{1.5}{$\rightsquigarrow$}}
			&
			\begin{subfigure}[t]{0.22\textwidth}
				\includegraphics[scale=0.3]{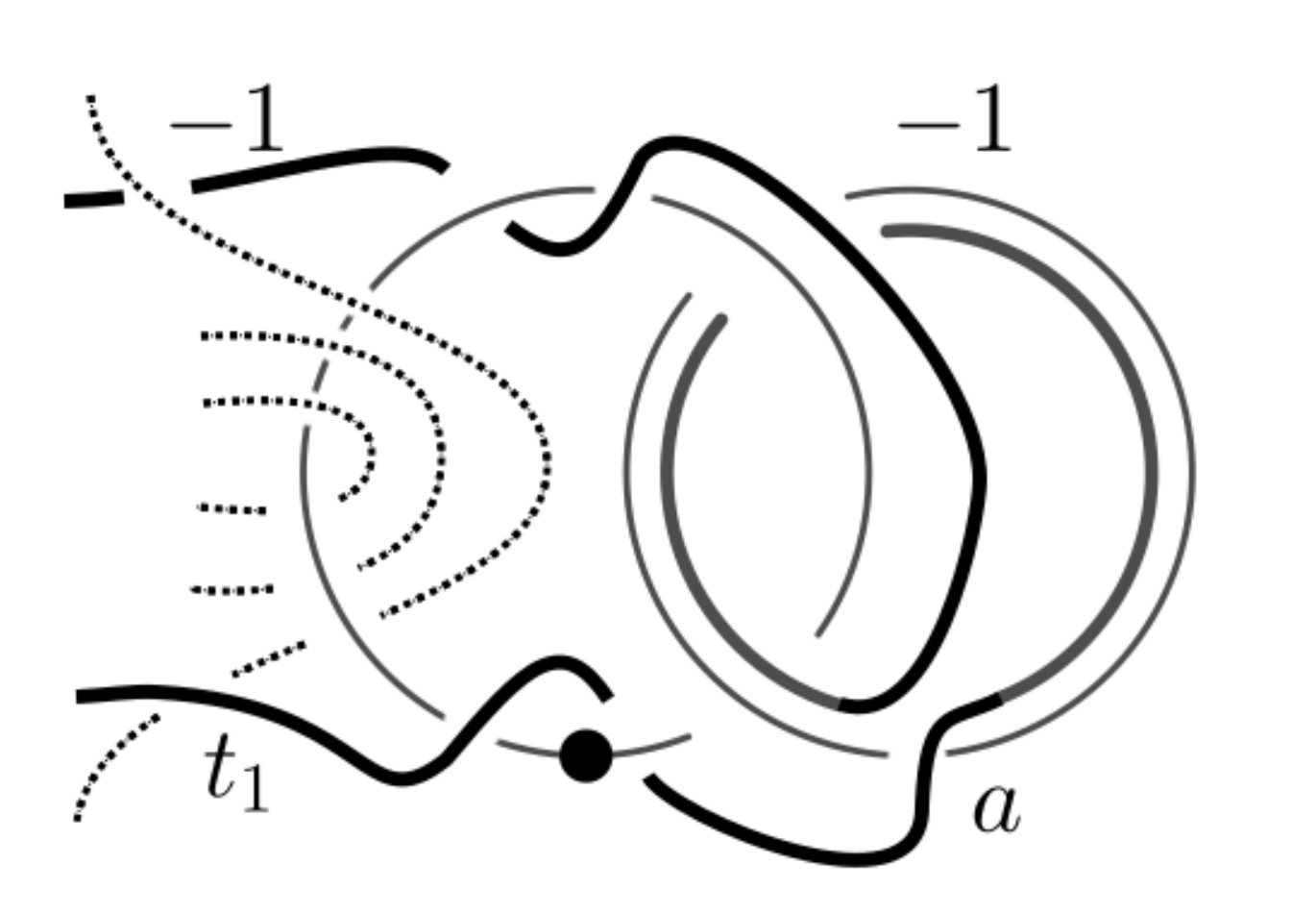}
			\end{subfigure}
			&
			\raisebox{1.2cm}{\scalebox{1.5}{$\sim$}}
			&
			\begin{subfigure}[t]{0.22\textwidth}
				\includegraphics[scale=0.3]{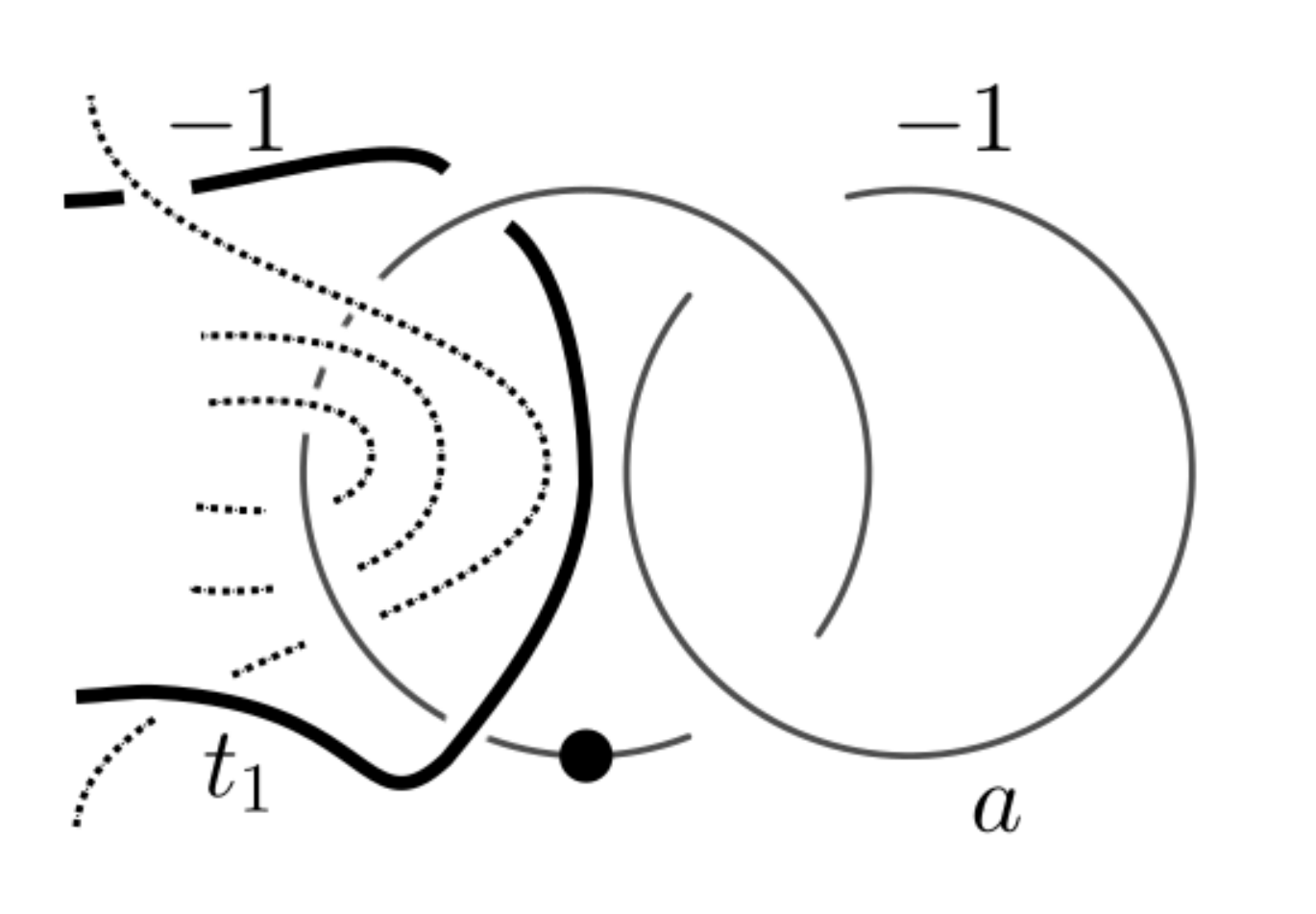}
			\end{subfigure}
		\end{tabular}
		\caption{Sliding $t_{j,1}$ over $a_j$.}
		\label{fig:first_slide}
	\end{figure}		
	Applying Lemma \ref{lem:slides}\subref{fig:ind_0} below inductively to pairs $(t_0,t)=(t_{j,i},t_{j,i+1})$ for $i=1,\dots,d_j-1$, we obtain a diagram in Figure \ref{fig:V_disjoint}. 
	 	 \begin{figure}[ht]
	 	 	\includegraphics[scale=0.32]{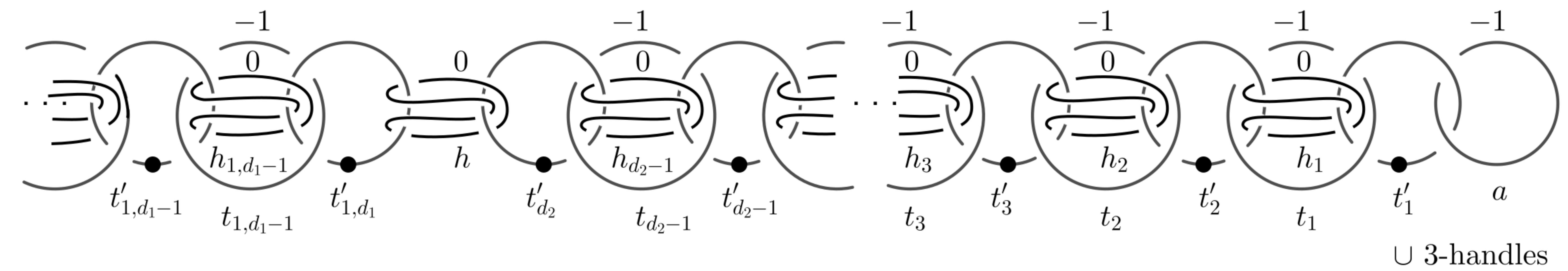}
	 	 	\caption{Right half of the Kirby diagram after sliding $t_{i}$ over $t_{i-1}$ for $i=1,2,\dots, d_{2}-1$}
	 	 	\label{fig:V_disjoint}
	 	 \end{figure}	
 	 
	Now by Lemma \ref{lem:slides}\subref{fig:sliding_torus} each $h_{i,j}$ can be slid to a $0$-framed unknot disjoint from the rest of the diagram. Since $\pi_{1}(V)=\{1\}$ by Proposition \ref{prop:htp}, \cite[p.\ 148]{GS_Kirby} shows that all $h_{j,i}$'s cancel with all $3$-handles (recall that there are $d_1+d_2-2$ of them, attached while removing the cores of $T_{j,i}$'s). Applying Lemma \ref{lem:slides}\subref{fig:a_0} inductively  to pairs $(a,t)=(t_{j,i},t_{j,i+1})$, $i=1,\dots,d_j-1$ gives Figure \ref{fig:handles}.
\end{proof}

In the above proof, we used the following standard exercise in handle slides, cf.\ \cite[p.\ 142]{GS_Kirby}.

\begin{lem}\label{lem:slides}
	Let $\ll_{1}$, $\ll_{2}$ be links in the left and right of one of the Figures \ref{fig:slides}\subref{fig:ind_0}-\subref{fig:a_0}. Define $\ll\subseteq \ll_1$ as $t_0$ in \subref{fig:ind_0}, $t+h$ in \subref{fig:sliding_torus} and $t+c_0+a$ in \subref{fig:a_0}. Let $\kk$ be a Kirby diagram containing $\ll_{1}$ in such a way that $\ll$ is disjoint from $\kk-\ll_{1}$. Then $\kk$ is equivalent to a diagram obtained from $\kk$ by replacing $\ll_{1}$ with $\ll_{2}$.
\end{lem}
	 	 \begin{figure}[ht]
	 	 	\begin{tabular}{cc}
	 	 		\begin{subfigure}[t]{0.55\textwidth}
	 	 			\includegraphics[scale=0.17]{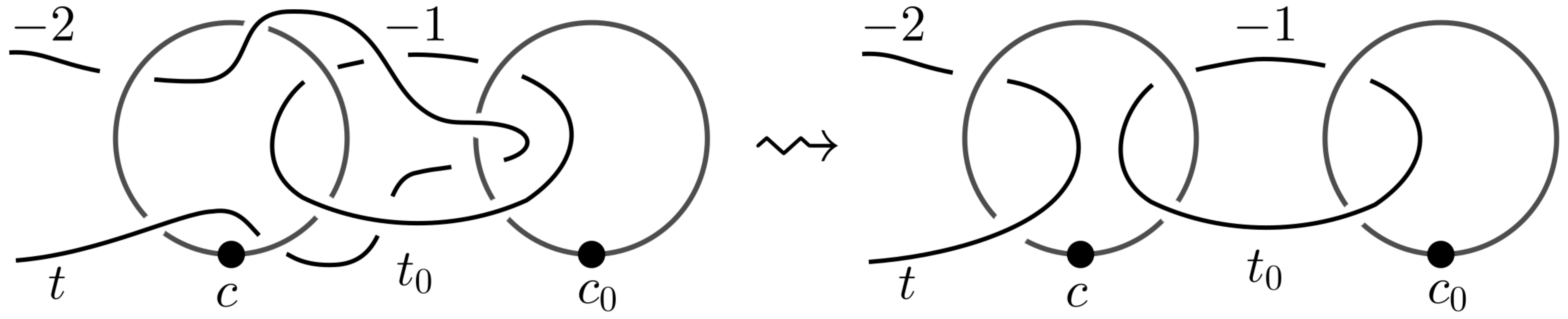}
	 	 			\caption{Sliding $t$ over $t_0$}
	 	 			\label{fig:ind_0}
	 	 		\end{subfigure}	
	 	 		&
	 	 		\setcounter{subfigure}{2}
	 	 		\begin{subfigure}[t]{0.45\textwidth}
	 	 			\includegraphics[scale=0.27]{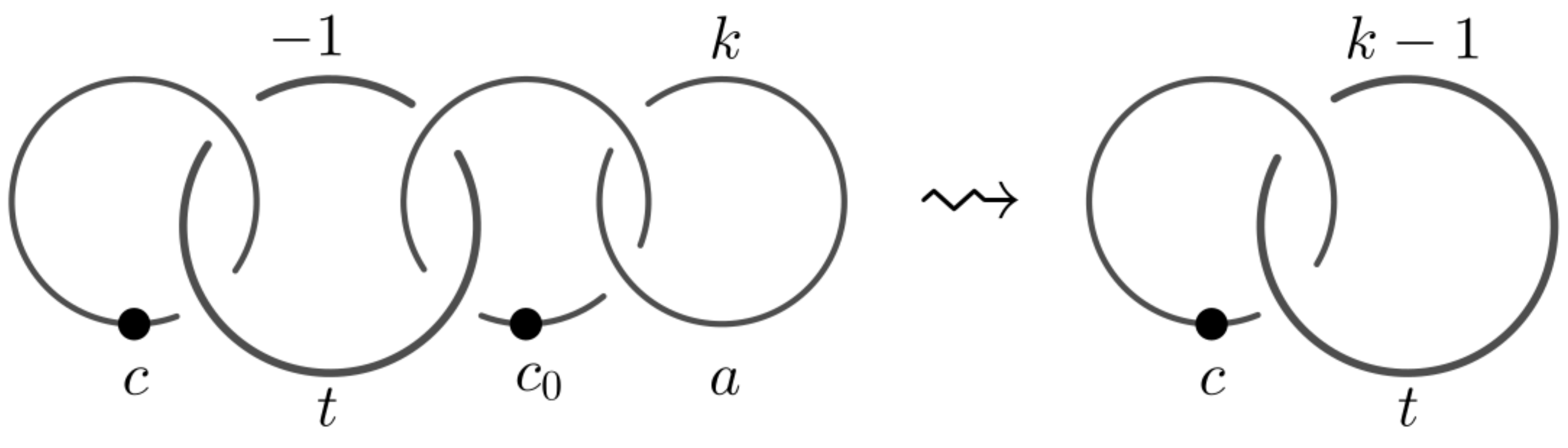}
	 	 			\caption{Sliding $t$ over $a$ and canceling $(a,c_0)$}
	 	 			\label{fig:a_0}
	 	 		\end{subfigure}
	 	 		\\
	 	 		\multicolumn{2}{c}{
	 	 		\setcounter{subfigure}{1}
	 	 			\begin{subfigure}[t]{\textwidth}
	 	 				\centering
	 	 				\includegraphics[scale=0.35]{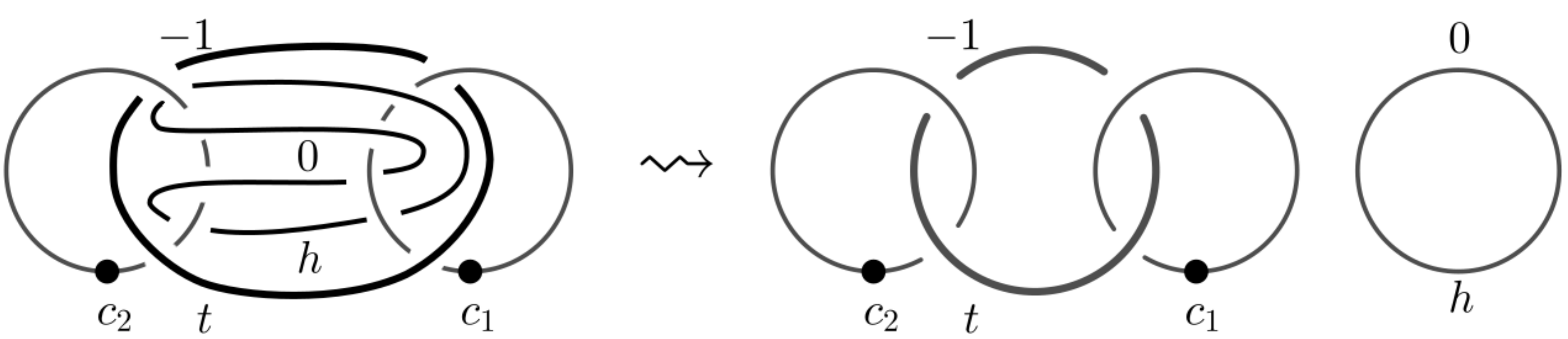}
	 	 				\caption{After adding and subtracting $t$, a $2$-handle $h$ like in Figure \ref{fig:added_torus} is ready to be canceled with a $3$-handle}
	 	 				\label{fig:sliding_torus}
	 	 			\end{subfigure}
	 	 		}
	 	 	\end{tabular}
	 	 	\caption{Lemma \ref{lem:slides}: handle slides applied inductively in the proof of Proposition \ref{prop:handles}}
	 	 	\label{fig:slides}
	 	 \end{figure}
\begin{proof}
	\subref{fig:ind_0}
	In order to have a clear picture, we draw $t_0$ rotated around a horizontal line, so that both $t_0$ and $t$ go under the bottom part of $c_0$, see Figure \ref{fig:ind}. We draw a parallel copy of $t_0$ in such a way that it twists over $t_0$ near that crossing, and cut this twist to make a band-sum with $t$. Now the part of $t$ on the right of the vertical diameter of $c$ can be isotoped to the latter, which proves the claim. 
	\
	\begin{figure}[ht]
		\begin{subfigure}{\textwidth}
			\begin{tabular}{ccccc}
				\begin{subfigure}[t]{0.3\textwidth}
					\includegraphics[scale=0.19]{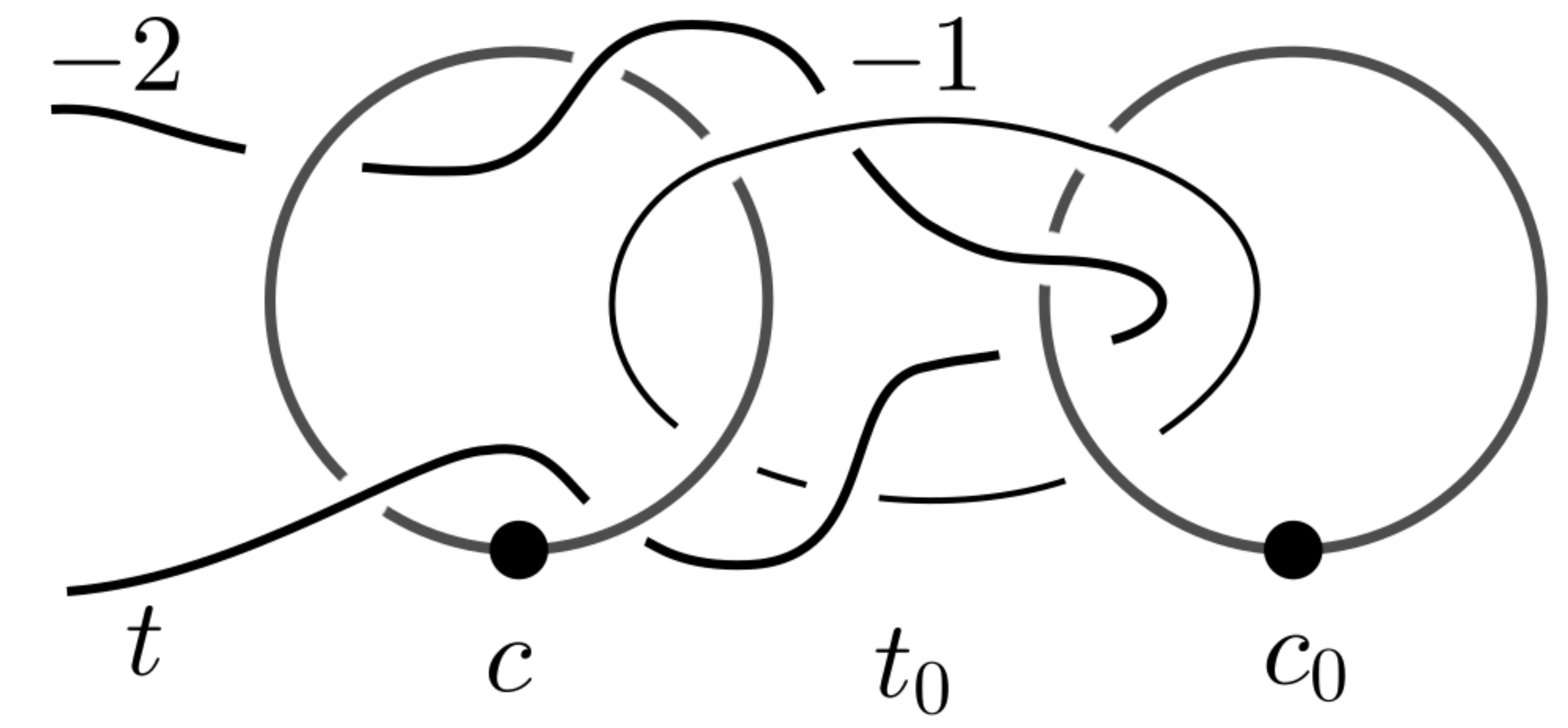}
				\end{subfigure}
				&
				\raisebox{1.2cm}{\scalebox{1.5}{$\rightsquigarrow$}}
				&
				\begin{subfigure}[t]{0.27\textwidth}
					\includegraphics[scale=0.19]{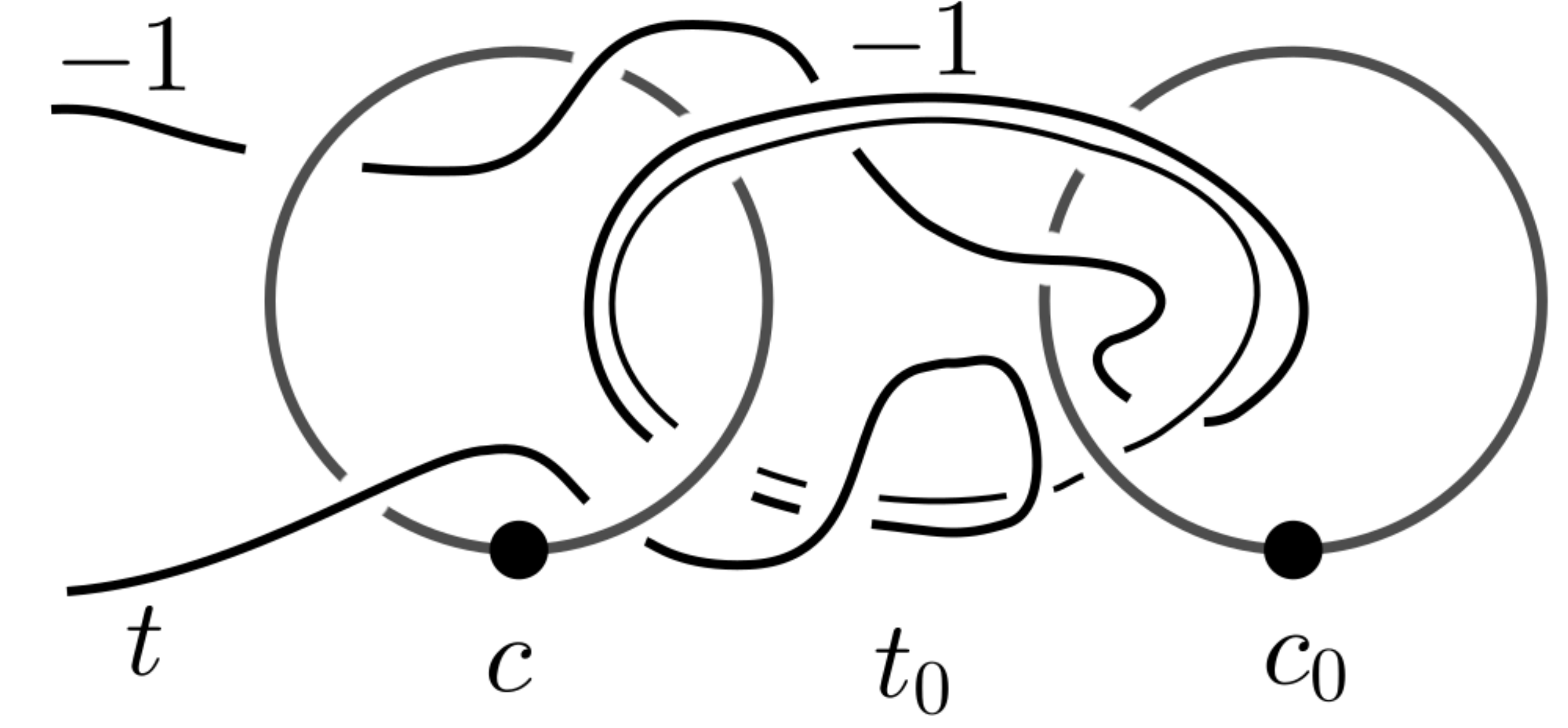}
				\end{subfigure}
				&
				\raisebox{1.2cm}{\scalebox{1.5}{$\sim$}}
				&
				\begin{subfigure}[t]{0.27\textwidth}
					\includegraphics[scale=0.19]{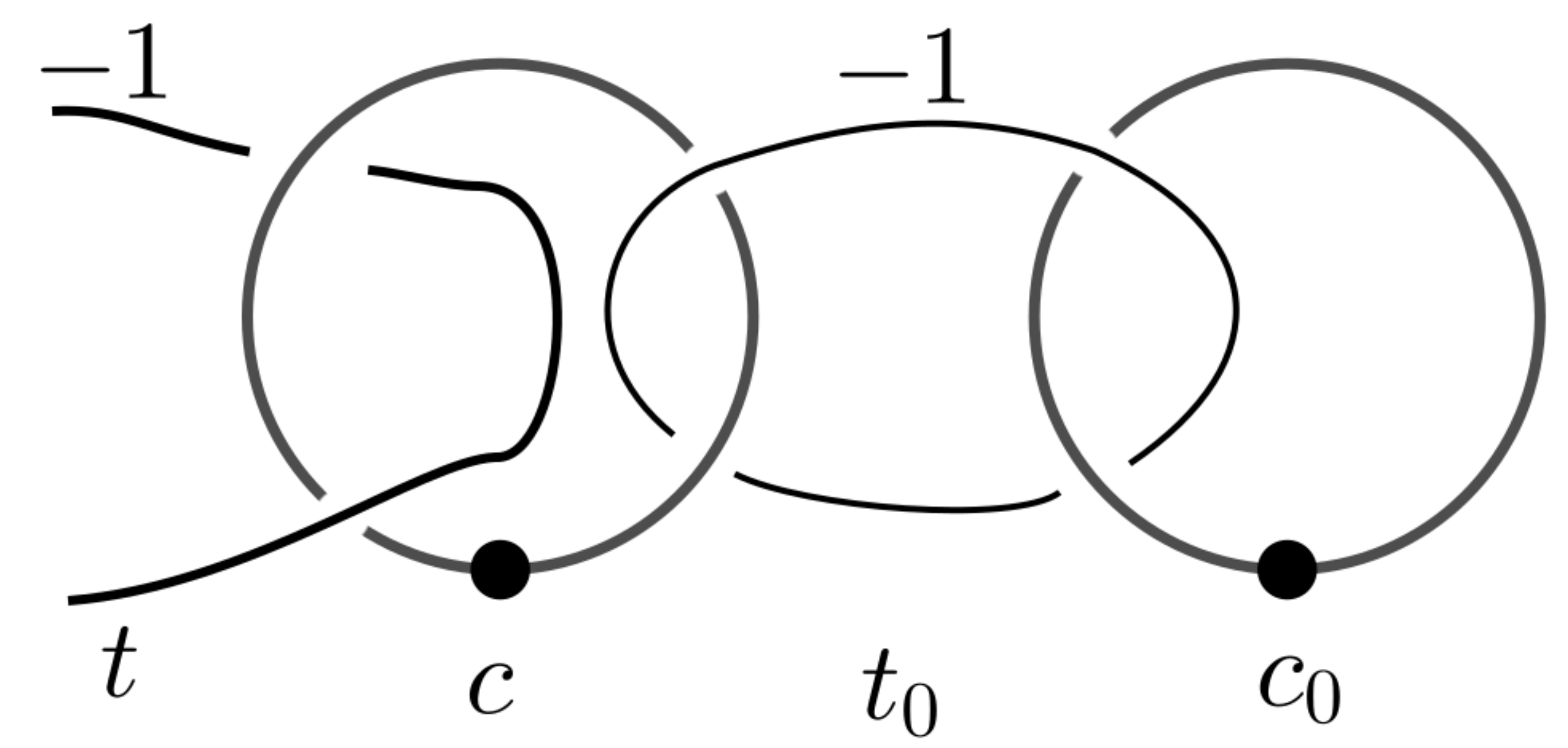}
				\end{subfigure}
			\end{tabular}
			\caption{Sliding $t$ over $t_0$}
			\label{fig:ind}
		\end{subfigure}
	\smallskip

	\begin{subfigure}{\textwidth}
		\begin{tabular}{ccccc}
			\begin{subfigure}[t]{0.27\textwidth}
				\includegraphics[scale=0.35]{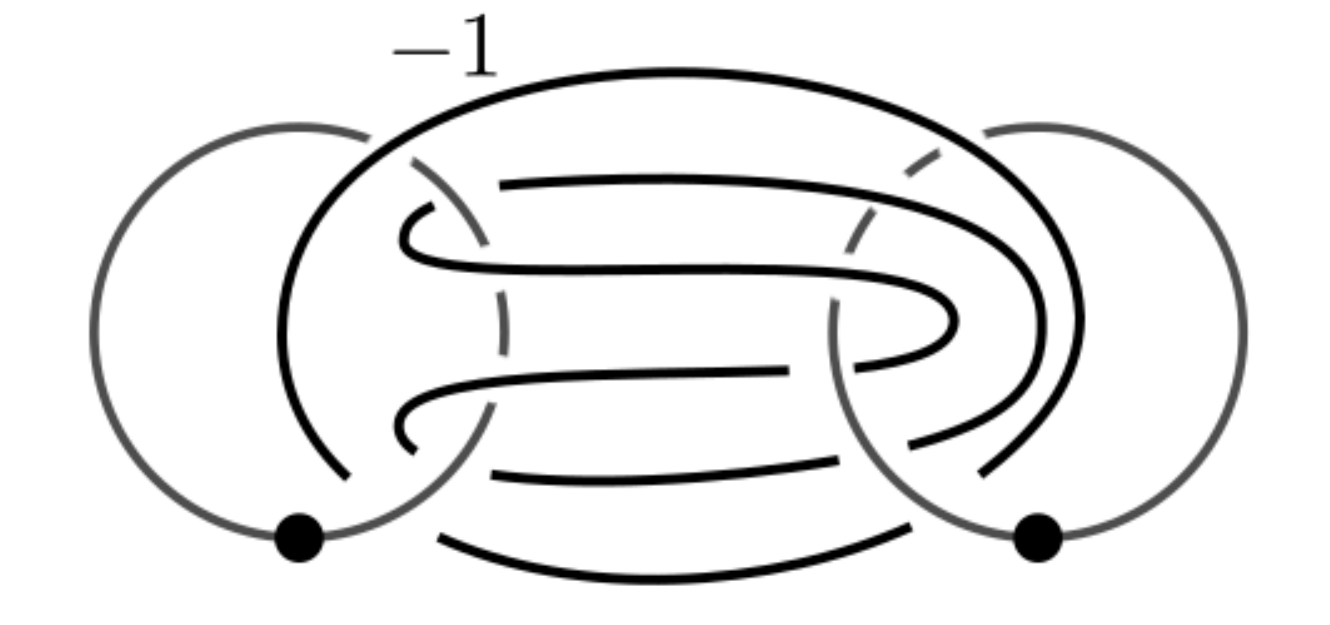}
			\end{subfigure}
			&
			\raisebox{1cm}{\scalebox{1.5}{$\rightsquigarrow$}}
			&
			\begin{subfigure}[t]{0.27\textwidth}
				\includegraphics[scale=0.35]{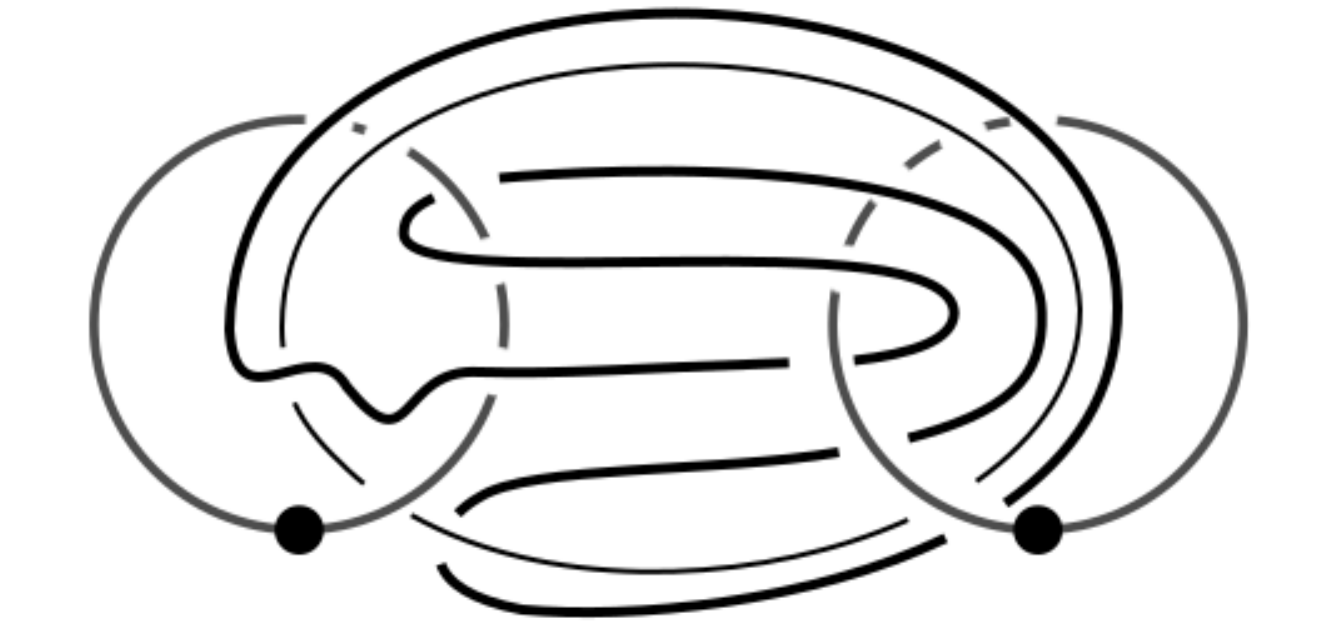}
			\end{subfigure}
			&
			\raisebox{1cm}{\scalebox{1.5}{$\sim$}}
			&
			\raisebox{-.2cm}{
				\begin{subfigure}[t]{0.27\textwidth}
					\includegraphics[scale=0.35]{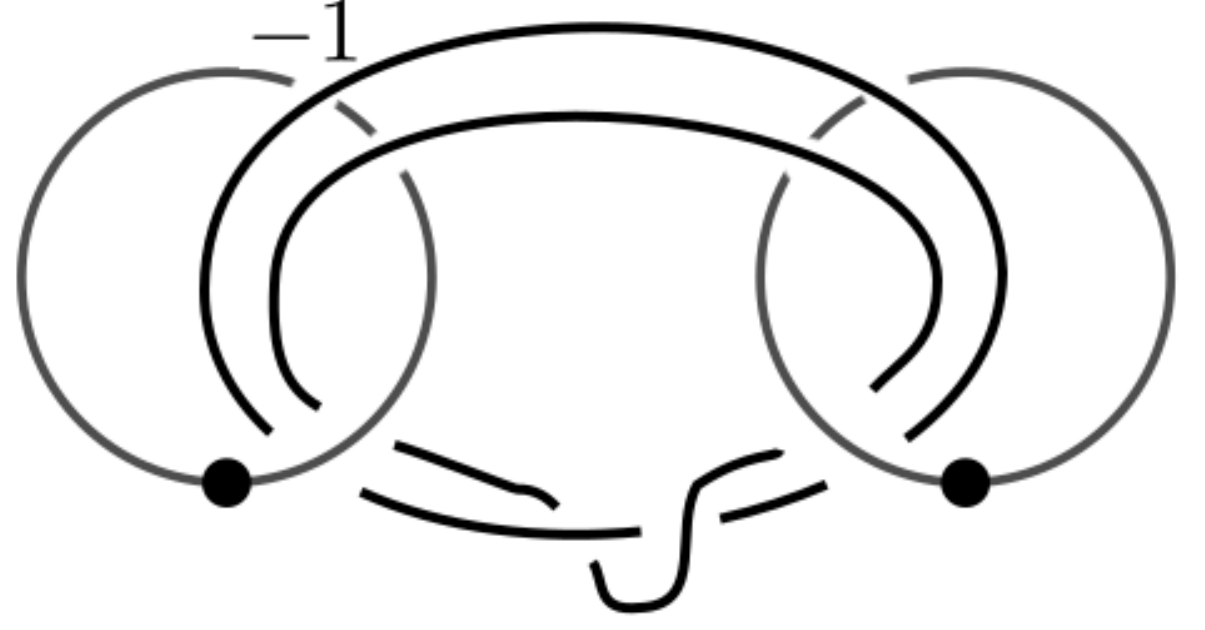}
				\end{subfigure}
			}
		\end{tabular}
		\caption{After sliding over $t$, $h$ becomes its parallel copy.}
		\label{fig:st_proof}
	\end{subfigure}
	\smallskip
	
	\begin{subfigure}{\textwidth}
		\begin{tabular}{ccccc}
			\begin{subfigure}[t]{0.26\textwidth}
				\includegraphics[scale=0.3]{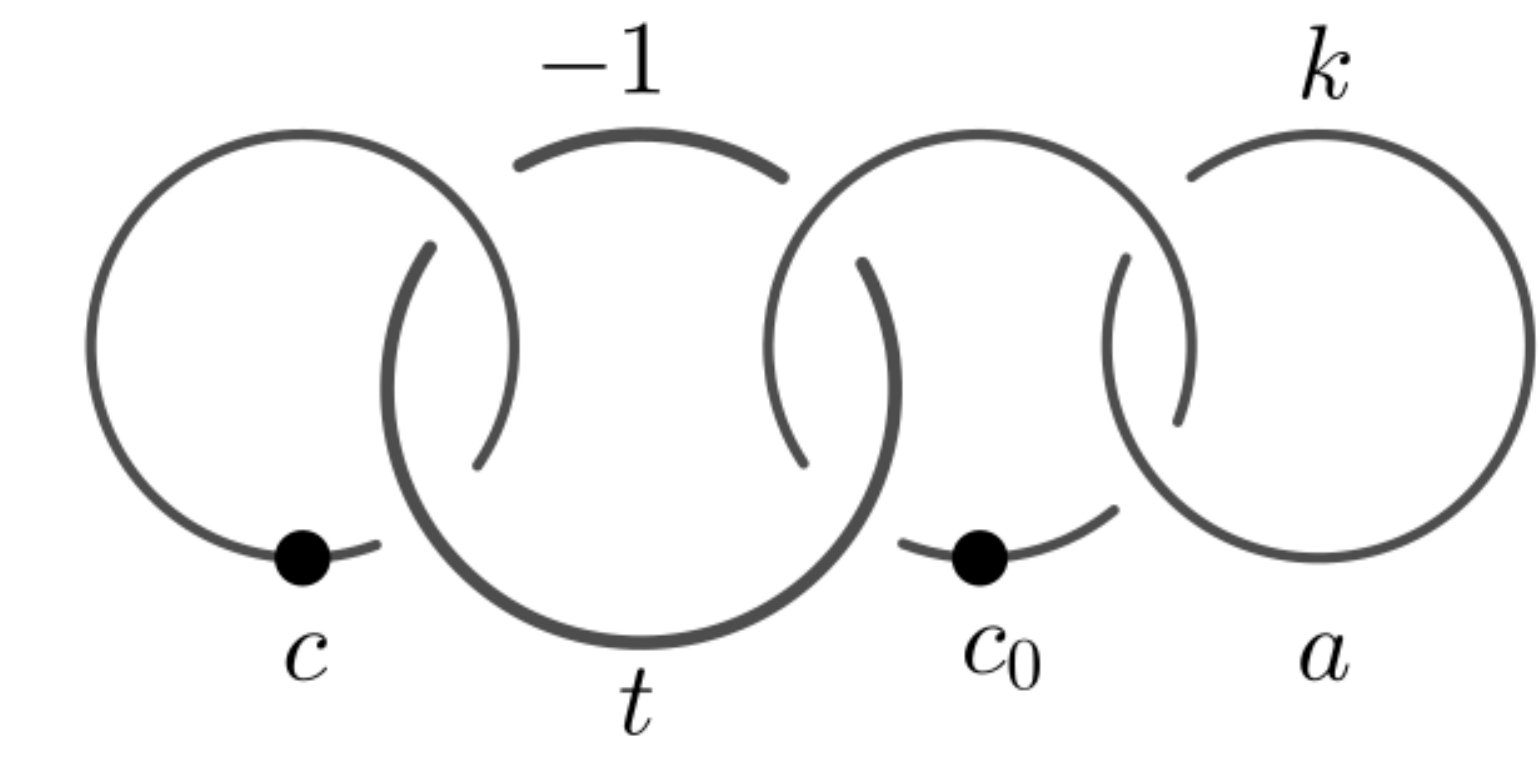}
			\end{subfigure}
			&
			\raisebox{1.2cm}{\scalebox{1.5}{$\rightsquigarrow$}}
			&
			\begin{subfigure}[t]{0.28\textwidth}
				\includegraphics[scale=0.3]{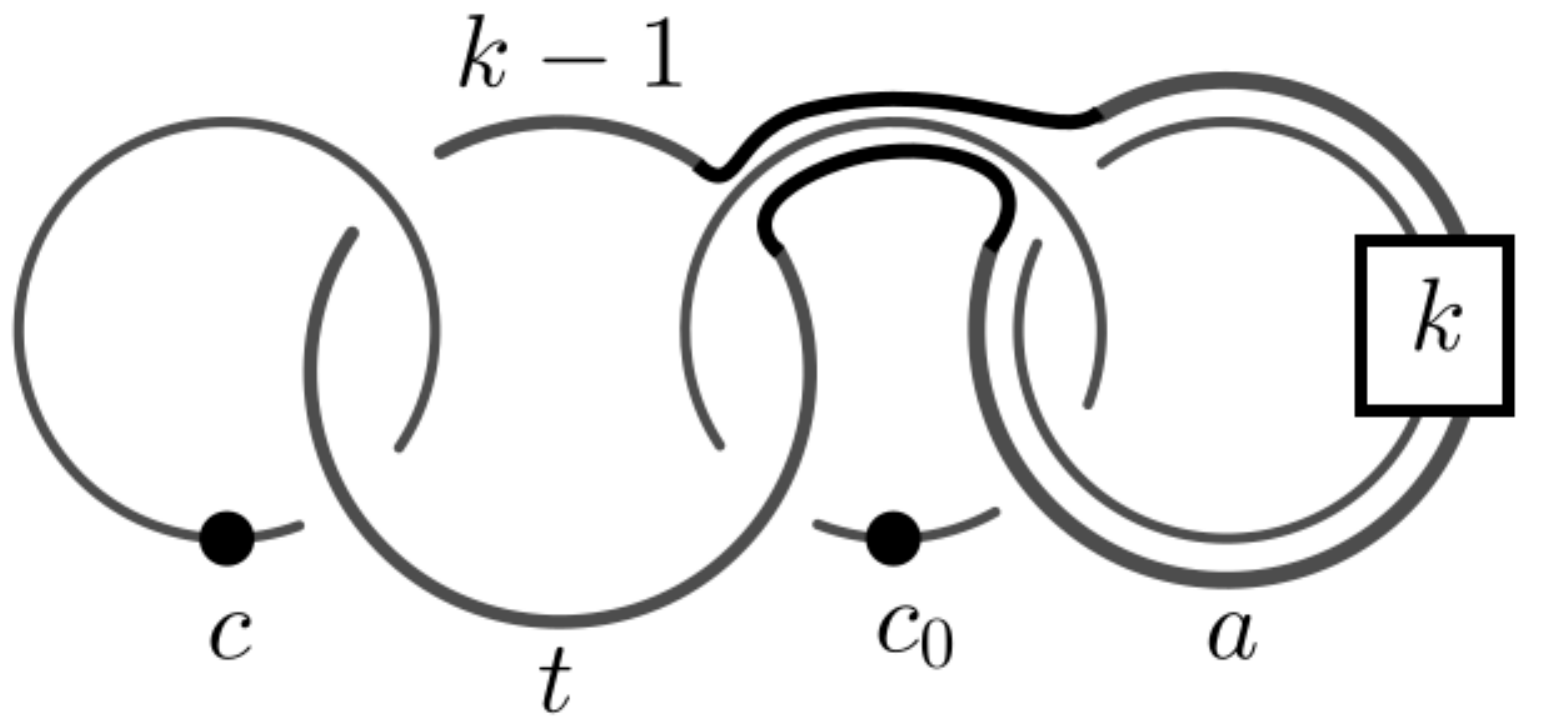}
			\end{subfigure}
			&
			\raisebox{1.2cm}{\scalebox{1.5}{$\rightsquigarrow$}}
			&
			\begin{subfigure}[t]{0.27\textwidth}
				\includegraphics[scale=0.3]{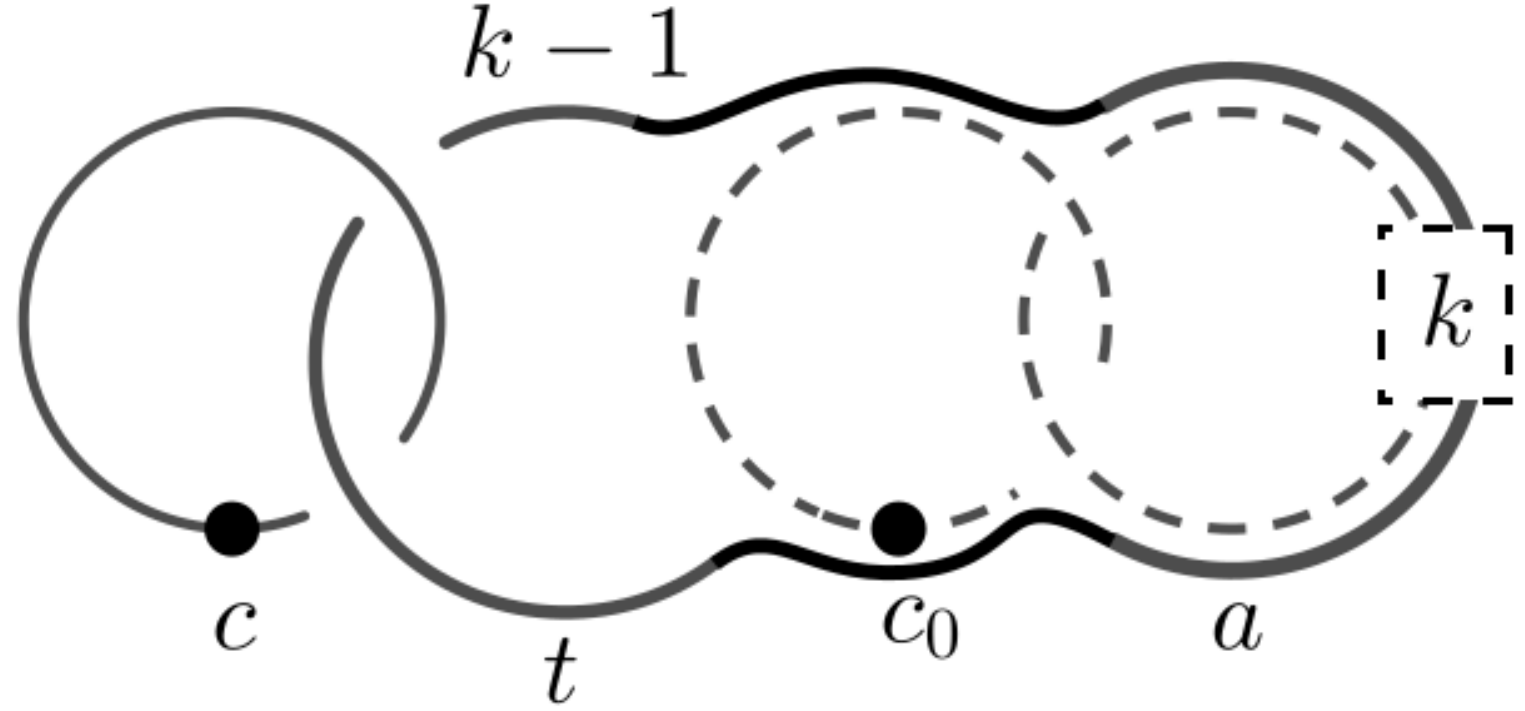}
			\end{subfigure}
		\end{tabular}
		\caption{Canceling the rightmost $1/2$-handle pair.}
		\label{fig:sliding_a}
	\end{subfigure}
		\caption{Proof of Lemma \ref{lem:slides}.}
		\label{fig:sliding_lemma}
\end{figure}

	\subref{fig:sliding_torus}	 As in \subref{fig:ind_0}, we draw $t$ rotated so that both $h$ and $t$ go under the bottom part of $c_2$, and do the band-sum there, see Figure \ref{fig:st_proof}. Then the bottom part of $h$ untangles from $c_1$, $c_2$, and becomes a parallel copy of $t$. Sliding it over $-t$ unlinks it from the rest of the  diagram.
	
	\subref{fig:a_0} After sliding $t$ over $a$, it becomes disjoint from $c_0$ and meets $c$ as before, see Figure \ref{fig:sliding_a}, where the box stands for $k$ twists. Moreover, $c_0$ now meets only $a$, once, so these handles cancel.	
\end{proof}

\begin{proof}[Proof of Theorem \ref{thm:diff}]
	Recall that by Proposition \ref{prop:handles}, $S$ is diffeomorphic to an interior of a $4$-manifold with Kirby diagram as in Figure \ref{fig:handles}, and by Corollary \ref{cor:not_homeo} such $S$'s are not homeomorphic for different $\{d_1,d_2\}$. Therefore, it remains to prove that Kirby diagrams in Figures \ref{fig:handles} and \ref{fig:knot} are equivalent.
	
\begin{figure}[ht]	
\begin{subfigure}{.8\textwidth}
	\begin{tabular}{ccc}
		\begin{subfigure}[t]{0.5\textwidth}
			\includegraphics[scale=0.33]{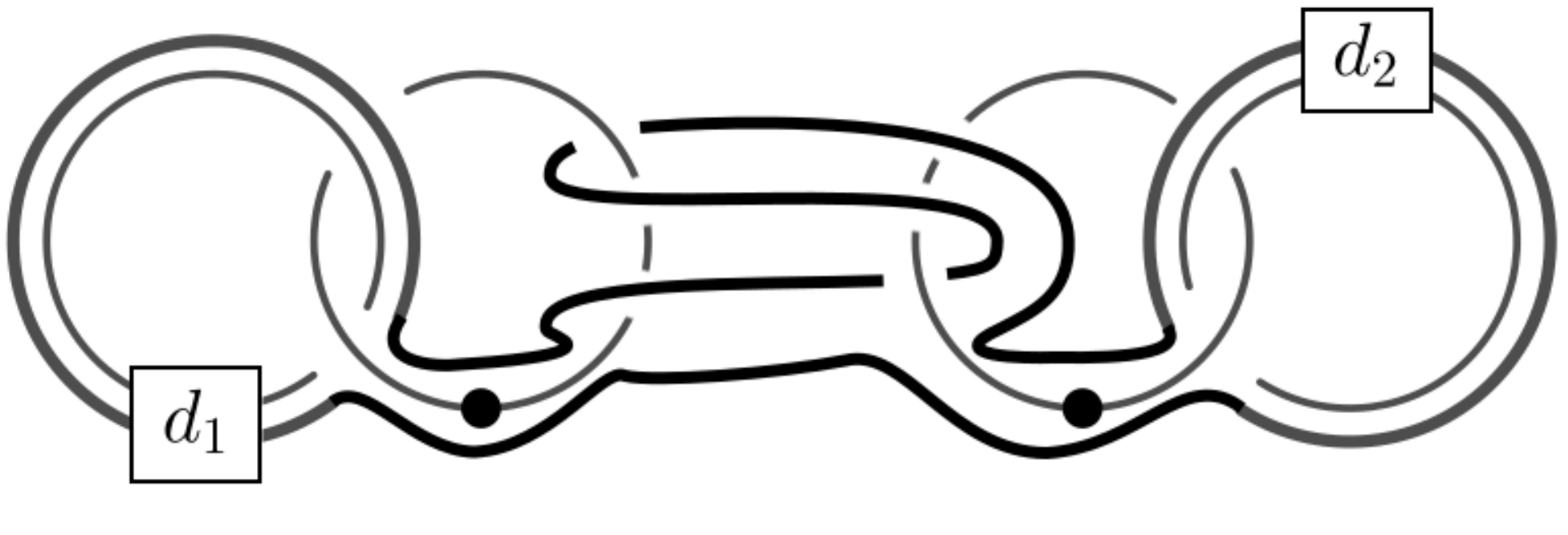}
		\end{subfigure}
		&
		\raisebox{1.2cm}{\scalebox{1.5}{$\sim$}}
		&
		\begin{subfigure}[t]{0.5\textwidth}
			\includegraphics[scale=0.33]{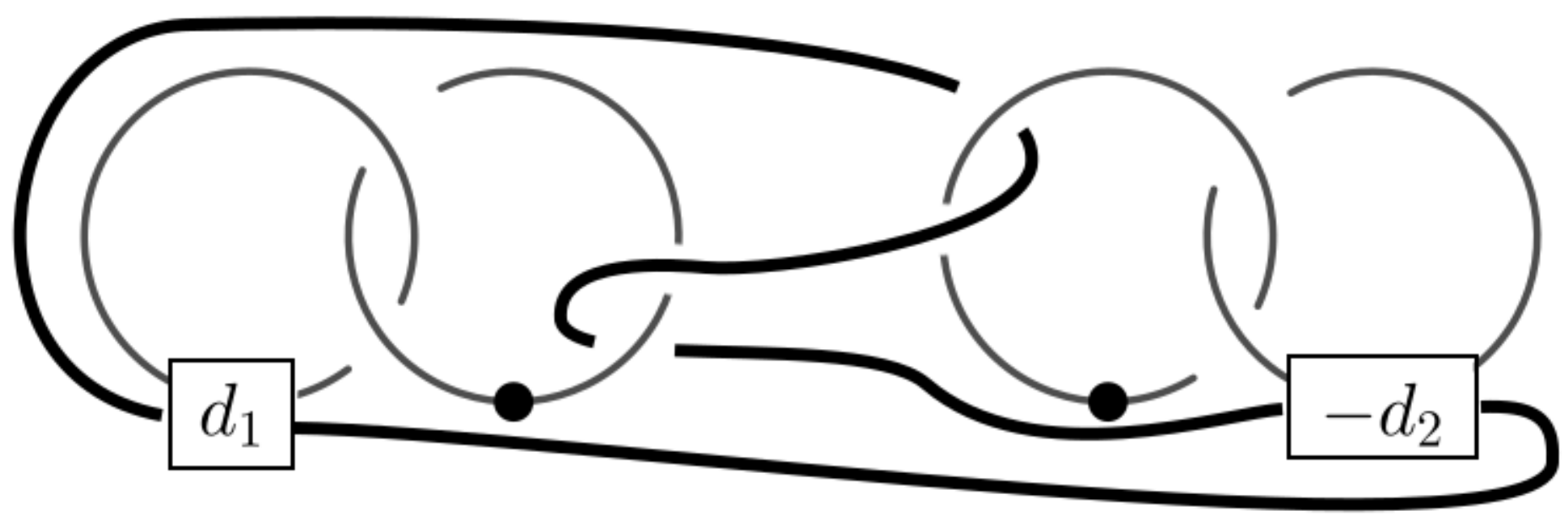}
		\end{subfigure}
	\end{tabular}
	\caption{Adding $a_2-a_1$ to $h$.}
	\label{fig:h}
\end{subfigure}
\begin{subfigure}{.8\textwidth}
	\begin{tabular}{ccc}
		\begin{subfigure}[t]{0.5\textwidth}
			\includegraphics[scale=0.33]{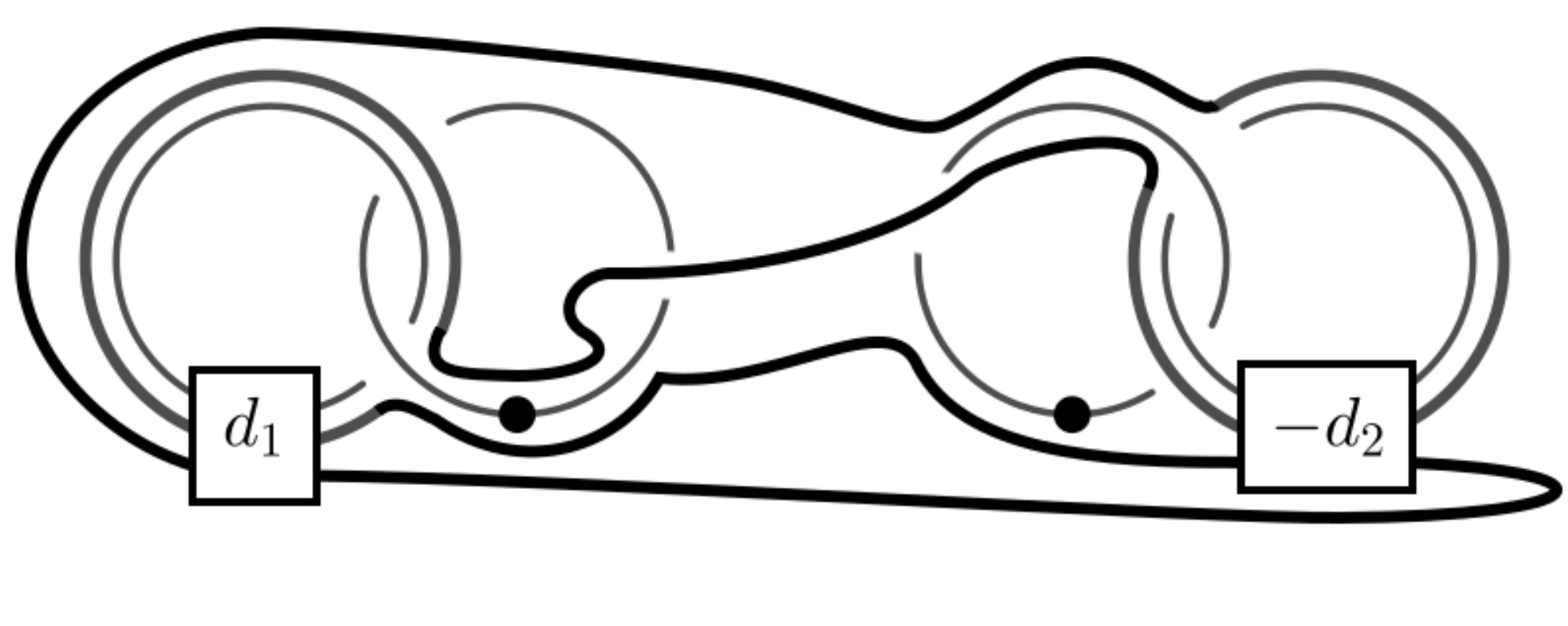}
		\end{subfigure}
		&
		\raisebox{1.2cm}{\scalebox{1.5}{$\sim$}}
		&
		\begin{subfigure}[t]{0.5\textwidth}
			\includegraphics[scale=0.33]{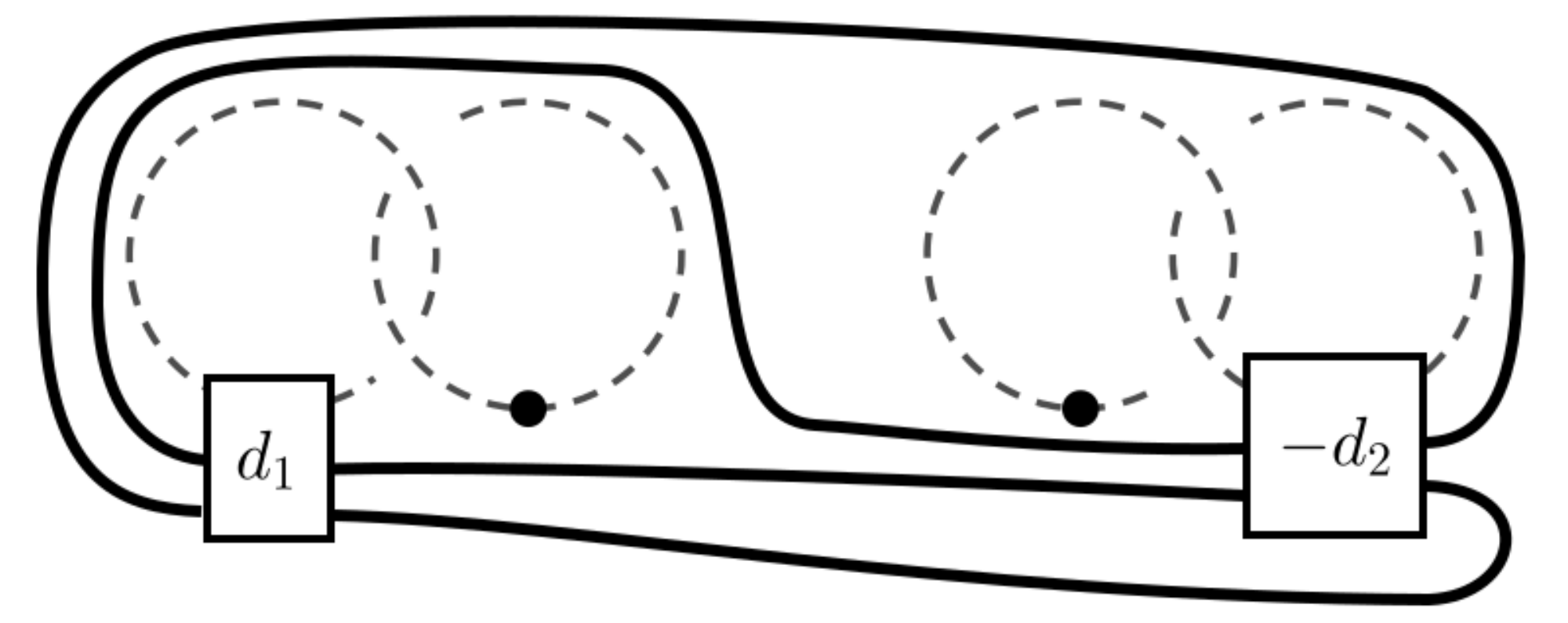}
		\end{subfigure}
	\end{tabular}
	\caption{Subtracting $a_2-a_1$ from $h$ and canceling pairs $(c_j,a_j)$.}
	\label{fig:hn}
	\end{subfigure}
\caption{Proof of Theorem \ref{thm:diff}}
\label{fig:proof_diff}
\end{figure}		

	Let $h$, $a_{j}$ be the $0$- and $(-d_{j})$-framed $2$-handles in Figure \ref{fig:handles}, and let $c_{j}$ be the $1$-handle meeting $a_{j}$. We will add and subtract $a_2-a_1$ from $h$ so that the latter untangles from $c_1$, $c_2$. Then we cancel pairs $(c_j,a_j)$. This is done in Figure \ref{fig:proof_diff}, which we now explain; see also \cite[5.1.8(b)]{GS_Kirby} for the case $d_1=d_2=1$.
	
	We start with Figure \ref{fig:handles}, but draw it with $a_2$ rotated along a horizontal axis. Next, we draw parallel copies of $-a_{1}$ and $a_{2}$, where a box stands for $d_{j}$ twists. Then we band sum $h$ with those copies near the points where $h$ and $a_{j}$ go under $c_{j}$, see Figure \ref{fig:h}. This move resolves nearby crossings of $h$ and $c_j$, leaving  $\lk(h,c_j)=-1$. Now we rotate $a_2$ again (this changes the sign in the box) and repeat the above move, but with reversed orientation of $a_j$, i.e.\ we add $a_1$ and subtract $a_2$. As before, the crossings of $h$ with $c_j$ get resolved, so these $1$-handles cancel with $a_j$, see Figure \ref{fig:hn}. After the first slide, $h$ twisted along $a_j$ in the boxes, so now it twists there over itself. A look at Figure \ref{fig:hn} shows that we do obtain a knot as in Figure \ref{fig:knot}. The framing of $h$ remains $0$ since we added and subtracted the same thing. 	
\end{proof}

\subsection{Computation of $\pi_{1}^{\infty}(S_{p_{1},p_{2}})$ and a relation with knot surgeries}\label{sec:pi_1}

Proposition \ref{prop:handles} allows us to compute the fundamental group of $V$ (which is trivial by Proposition \ref{prop:htp}), and, more interestingly, that of $M$, which equals $\pi_{1}^{\infty}(S)$, see \cite{Mumford-surface_singularities}.

\begin{prop}\label{prop:pi_1}
The fundamental group at infinity of $S$ has a presentation
\begin{equation}\label{eq:pi_1}
\pi_{1}^{\infty}(S)=
\langle \delta_{1},\delta_{2},\lambda\ | \ \delta_{1}=[\gamma_2,\lambda^{-1}],\ \delta_{2}=[\gamma_1,\lambda],\ [\gamma_1,\gamma_2]=1\rangle, \quad \mbox{ where } \gamma_{j}=\delta_{j}^{d_j}.
\end{equation}
Moreover, $\gamma_j$ for $j\in \{1,2\}$ is the attaching circle of the $(-d_j)$-framed $2$-handle in Figure \ref{fig:handles}, and $\lambda$ is the belt sphere of the $0$-framed one.
\end{prop}
\begin{proof}
	Let $\gamma_{1},\gamma_{2},\lambda\in \pi_{1}(\d (\bT^{2}\times \bD^{2}))=\pi_{1}(\bT^{3})$ be the standard generators. Proposition \ref{prop:handles} shows that $M$ is obtained from $\bT^{3}$ by a surgery along $\gamma_{1},\gamma_{2}$ with coefficients $d_{1}$, $d_{2}$. The meridians of deleted tori are $\mu_{1}\de[\gamma_{2},\lambda^{-1}], \mu_{2}\de [\gamma_{1},\lambda]$. Now by van Kampen theorem, each surgery replaces the relation $\mu_j=1$ in $\pi_{1}(\bT^{3})$ by $\mu_j=\delta_j$, where $\delta_j\in \pi_{1}(M)$ is the corresponding attaching circle, satisfying $\delta_{j}^{d_j}=\gamma_{j}$.
\end{proof}

\begin{rem}\label{rem:pi_1-Mumford}
	The presentation \eqref{eq:pi_1} can also be obtained from the plumbing construction in \cite{Mumford-surface_singularities}, see Section \ref{sec:graph_manifolds}. If $D$ is a tree, \cite{Mumford-surface_singularities} shows how to calculate $\pi_{1}(M)$ in terms of fibers of the plumbed $\bS^{1}$-bundles, using van Kampen theorem. For our $D$, one needs to use the van Kampen theorem for groupoids \cite{Brown_van-Kampen}, say, to glue $\d \mathrm{Tub}(L_{1,\infty}+L_{2,\infty})$ to the rest  of $\d \mathrm{Tub}(D)$ along two disjoint plumbing tori, each containing a base point. This gives the additional generator $\lambda$, which is the sum of paths joining the two base points in each of the glued pieces. In this interpretation, $\gamma_{j}$ is the fiber over $L_{j,0}$. Indeed, in the proof of Proposition \ref{prop:handles} it was drawn parallel to the boundary of a disk corresponding to $T_{j,d_j-1}$, which goes once around $L_{j,0}$. In turn, $\delta_{j}$ is a fiber over $T_{j,1}$, that is, a singular fiber of a lens space $L(1,d_{j})$ attached to $\d \mathrm{Tub}(L_{j,0})$.
\end{rem}

In the proof of Corollary \ref{cor:not_homeo}, we have essentially seen the JSJ decomposition of $M$. We will now show how to see it from the above description of $M$, as a $0$-surgery on $K_{[2d_1,2d_2]}$. In order to settle the notation, we recall the basic definitions following \cite{Hatcher_3M}.

Let $N$ be a $3$-manifold. A surface $\Sigma$ in $N$ is \emph{incompressible} if every loop in $\Sigma$ bounding a disk in $N$ bounds a disk in $\Sigma$, too. We write $N|\Sigma$ for a $3$-manifold (with boundary) obtained from $N$ by removing $\Sigma$ together with some tubular neighborhood. 
Assume that $N$ is irreducible (see Section \ref{sec:graph_manifolds}). Then $N$ has a collection $\Sigma$ of incompressible tori such that each component of $N|\Sigma$ is either a Seifert fiber space or has no incompressible tori. A minimal  such $\Sigma$  (in the sense of inclusion) is unique up to an isotopy \cite[Theorem 1.9]{Hatcher_3M}, and $N|\Sigma$ is called a \emph{JSJ-decomposition} of $N$. 

We now recall the notation for Seifert fibered spaces following \cite[\S 2.1]{Hatcher_3M}. Let $B$ be a compact surface of genus $g$ with $r$ boundary components, and let $M_0\to B_0$ be an $\bS^{1}$-bundle, where $B_0$ is $B$ with $k$ disks $D_{1},\dots, D_{k}$ removed. The preimage $T_{i}\subseteq M_{0}$ of $\d D_{i}$ is a torus: we attach to it a solid torus $\bD^{2}\times \bS^{1}$ in such a way that a meridian $0\times \bS^{1}$ is attached to a loop in $T_{i}$ whose preimage in its universal cover is a line of slope $\tfrac{p_{i}}{q_{i}}$. The resulting manifold is a \emph{Seifert fibered space} $M(g,r;\tfrac{p_{1}}{q_{1}},\dots, \tfrac{p_{k}}{q_k})$. It is a graph manifold. Its graph has one vertex corresponding to $M_0$ and $k$ maximal twigs, whose type is a continued fraction expansion of $\tfrac{p_i}{q_i}$. This graph is normal whenever a Seifert fibration is unique, see \cite[5.7]{Neumann-plumbing_graphs}. Few important exceptions, including lens spaces $\bS^3/\Z_{p}$, are listed in \cite[Theorem 2.3]{Hatcher_3M}. Their normal graphs are described in \cite[6.1]{Neumann-plumbing_graphs}.

Let again $M=\d \mathrm{Tub}(D)$ be as in Notation \ref{not:d}. The following result can be inferred from the computation of $\Gamma(M)$ in Figure \ref{fig:graphs} using the algorithm in \cite[\S 5]{Neumann-plumbing_graphs}. Nonetheless, we sketch a direct argument using the geometric description of incompressible tori in $M$ deduced in \cite[Lemma 2.2]{BW_surgery} from \cite{HT_inc}.

\begin{prop}[JSJ decomposition of $M$]\label{prop:JSJ}
	Let $\Sigma_{1},\Sigma_{2}\subseteq M$ be the plumbing tori at $L_{1,0}\cap L_{2,0}$ and $L_{1,0}\cap L_{2,\infty}$, see Section \ref{sec:graph_manifolds}. Then the JSJ decomposition of $M$ is
	\begin{enumerate}
		\item\label{item:d_1>1} $M|(\Sigma_{1}\sqcup \Sigma_2)=M(2,0;\tfrac{1}{d_1})\sqcup M(2,0;\tfrac{1}{d_2})$ if $d_{1},d_{2}\neq 1$,
		\item\label{item:d_1=1} $M|\Sigma_1=M(2,0;\tfrac{1}{d_2})$ if $d_{1}=1$ and $d_2>1$,
		\item\label{item:Moser} $M=M(0,0;\tfrac{1}{2},\tfrac{1}{3},\tfrac{1}{6})$ if $d_1=d_2=1$.
	\end{enumerate}
	Let $\gamma_{j},\delta_{j}$ be as in Proposition \ref{prop:pi_1}. Then $\delta_{j}$ is the singular fiber of $M(2,0;\tfrac{1}{d_j})$ and $\gamma_{j}$ is a general one.
\end{prop}
\begin{proof}
	By Theorem \ref{thm:diff}, $M$ is a $0$-surgery on $\bS^{3}$ along $K=K_{[2d_1,2d_2]}$. Because $K_{[2,2]}$ is just a trefoil, \ref{item:Moser} follows from \cite[Proposition 3.1]{Moser_knots} which completely describes surgeries on torus knots. Assume $d_2>1$. If $d_1>1$, too, then cutting along $\Sigma_1\sqcup \Sigma_2$ in case \ref{item:d_1>1} corresponds to removing the edges of the circular subgraph in Figure \ref{fig:g_2}, and decorating the leftmost vertices with $r=2$, i.e.\ puncturing the base spheres twice. If $d_1=1$ then cutting along $\Sigma_1$ corresponds to removing the loop in Figure \ref{fig:g_3} and adding boundary as before. This shows the equalities in \ref{item:d_1>1}--\ref{item:d_1=1}, and proves that $M$ and, if $d_1=1$, $M|\Sigma_i$ for $i\in \{1,2\}$ are not Seifert fibered: indeed, $\Gamma(M|\Sigma_i)$ is not one of the graphs in \cite[5.7, 6.1]{Neumann-plumbing_graphs}. It remains to show that $\Sigma_{1},\Sigma_{2}\subseteq M$ are incompressible.
	
	The $3$-manifold $M$ is irreducible by \cite[Theorem 2]{HT_inc} (it also follows from Lemma \ref{lem:graphs}\ref{item:irr}). Put $M_0=\bS^{3}\setminus \mathrm{Tub}(K)\subseteq M$. Put $M_0=\bS^{3}\setminus \mathrm{Tub}(K)\subseteq M$. By \cite[Proposition 1(1) and the proof of Proposition 2]{HT_inc} there are, up to isotopy, at most two  incompressible surfaces $\Sigma_0\subseteq M_0$ such that the surfaces $\Sigma\subseteq M$ obtained by capping $\Sigma_0$ are tori (see also proof of \cite[Lemma 2.2]{BW_surgery}). Such $\Sigma_0$ are denoted by $S_{1}(0)$ and $S_{1}(1)$ in \cite{HT_inc}, and they have exactly one boundary component each. By \cite[Proposition 1.5(a)]{Przytycki_inc} they are unknotted (see Definition 1.3 loc. cit), so by Theorem 1.4 loc.\ cit, $\Sigma$ remains incompressible. Moreover, by \cite[Theorem 1(e)]{HT_inc} $S_1(0)$ is isotopic to $S_1(1)$ if and only if $d_1=1$. Thus our claim will follow once we show that $\Sigma=\Sigma_1,\Sigma_2$ for $\Sigma_0=S_1(0),S_1(1)$. 
	
	To this end, we recall the description of $S_{1}(0)$, $S_{0}(1)$ from \cite[p.\ 227, Fig.\ 1]{HT_inc}. Consider Figure \ref{fig:knot} as a knot in $\R^{3}$, and let $\pp$ be a vertical plane between the \enquote{twists}. It meets the knot in four points. Connect them to get a quadrilateral  $\qq\subseteq \pp$: it is a $0$-handle for $S_{1}(0)$. Now attach to it two $1$-handles along the bands with $d_{1}$ and $d_{2}$ twists, respectively: this gives a punctured torus $S_{1}(0)$. The torus $S_{1}(1)$ is constructed the same way, with $\qq$ replaced by its complement in the compactification $\bar{\pp}=\bS^{2}\times \{\mathrm{pt}\}\subseteq (\bS^{2}\times \R)\sqcup \{\pm \infty\}=\bS^{3}$.
	
	The loops based at $\qq$ and going around the $1$-handles of $S_1(0)$ are isotopic to the non-dotted circles in Figure \ref{fig:handles}, which by Remark \ref{rem:pi_1-Mumford} are fibers $\gamma_j$ over $L_{j,0}$. Hence $S_{1}(0)$ caps to $\Sigma_1$. To see the $1$-handles of $S_{1}(1)$, we need to change the base point from $\qq$ to its complement. The gluing described in Remark \ref{rem:pi_1-Mumford} shows that this amounts to changing, say, $L_{2,0}$ to $L_{2,\infty}$. This way we get $\Sigma_2$, as claimed.
\end{proof}

\begin{rem}[A direct argument for minimality of $\Sigma$]In the above proof, we have used the classification of graph manifolds to show that the collections $\Sigma_1\sqcup\Sigma_2$ in \ref{item:d_1>1} and $\Sigma_1$  in \ref{item:d_1=1} of incompressible tori decomposing $M$ to Seifert fibered spaces is minimal. However, this can be also seen directly, as follows.

First, we note that since $\Gamma(M|\Sigma_{i})$ is a tree, one can use van Kampen theorem as in \cite{Mumford-surface_singularities} to compute
\begin{equation}\label{eq:pi_1-Sigma}
\pi_{1}(M|\Sigma_{1})=\langle \delta_{1},\delta_{2}|[\gamma_1,\gamma_2]=1 \rangle, \mbox{ where } \delta_j\mbox{ and } \gamma_j=\delta_{j}^{d_{j}}\mbox{ are as in Remark \ref{rem:pi_1-Mumford}.}
\end{equation}

Suppose that $d_1,d_2>1$, but $M|\Sigma_i$ is Seifert fibered for some $i\in \{1,2\}$. Then by \cite[Proposition 1.11]{Hatcher_3M}, the incompressible torus $\Sigma_{3-i}\subseteq M|\Sigma_{i}$ is either horizontal or vertical. In the first case, it covers the base orbifold, which has two boundary components, so comparing Euler characteristics as done in \cite[p.\ 22]{Hatcher_3M} shows that $M|\Sigma_i$ has no singular fibers, so $M|\Sigma_{i}\cong (\bD^{1}\times \bS^{1})\times \bS^{1}$, a contradiction with \eqref{eq:pi_1-Sigma}. In the second case, cutting $M|\Sigma_{i}$ further along $\Sigma_{3-i}$  preserves the Seifert fibration: therefore by the equality in \ref{item:d_1>1}, it restricts to the unique one on $M(2,0;\tfrac{1}{d_{j}})$. Hence after gluing $\Sigma_{3-i}$ back, the fiber of one fibration becomes a loop around the degenerate fiber of the other. This is a contradiction, because the former is central in $\pi_{1}$, and the latter is not.

Suppose now that $d_{2}>1$, but $M$ is Seifert fibered. Again, the incompressible torus $\Sigma_{1}$ is either horizontal or vertical. In the first case, since $M|\Sigma_{1}$ is connected, \cite[p.\ 22]{Hatcher_3M} implies that $M|\Sigma_1\cong \Sigma_1\times \bD^{1}$, a contradiction with \eqref{eq:pi_1-Sigma}. In the second case, we obtain that $M|\Sigma_{1}$ is a Seifert fiber space. We have shown that this is possible only if $d_{1}=1$, in which case the general fiber is represented by $\gamma_{1}\in \pi_{1}(M)$. In particular, $\langle \gamma_{1}\rangle \subseteq \pi_{1}(M)$ is a (nontrivial)  central subgroup. Hence by \eqref{eq:pi_1} $\delta_{2}=[\gamma_{1},\lambda]=1$ and $\delta_{1}=[\delta_{2}^{d_{2}},\lambda^{-1}]=1$, so $\gamma_{1}=1$; a contradiction.
\end{rem}

\section{$\C^{*}$-fibrations and complete algebraic vector fields on $S_{p_1,p_{2}}$}\label{sec:AAut}

In this section we prove Theorem \ref{thm:AAut}\ref{item:AAut_thm}. It follows from \cite[Theorem 1.6]{KKL_avf} that a flow of any complete algebraic vector field on $S_{p_{1},p_{2}}$ preserves some $\C^{*}$-fibration of $S_{p_{1},p_{2}}$ (see Proposition \ref{prop:fibrations}\ref{item:vf} for details). Therefore, to understand $\AAut(S_{p_{1},p_{2}})$ we need to classify all $\C^{*}$-fibrations of $S_{p_{1},p_{2}}$. 

Their degenerate fibers are affine curves with Euler characteristic at least $\etop(\C^{*})$, see \cite[III.1.8]{Miyan-OpenSurf}, i.e.\ unions of $\C^{*}$ and $\C^{1}$. Therefore, we begin by classifying all curves in $S_{p_{1},p_{2}}$ which are isomorphic to $\C^{1}$. This is done in Lemma \ref{lem:C1}. In Lemma \ref{lem:C*_lines} we show that any $\C^{*}$-fibration of $S_{p_{1},p_{2}}$ comes from a $\C^{*}$-fibration of a complement of suitably chosen pair of disjoint $\C^{1}$'s; which by Lemma \ref{lem:tori} turns out to be a torus $\C^{*}\times \C^{*}$. The union of such tori is exactly the complement of $Z$ from Theorem \ref{thm:AAut}\ref{item:AAut_thm}.

\begin{notation}\label{not:S}
	As in Notation \ref{not:d}, we put $d_j=\deg p_{j}+1$ and write $S\de S_{p_{1},p_{2}}$. We use notation introduced in Construction \ref{constr} and denote by $T_{j}$ the maximal $(-2)$-twig of $D$ meeting $A_{j}$. Additionally, we denote by $\cc=\{x_1+x_2=1\}\subseteq \P^{1}\times \P^{1}$ the curve of type $(1,1)$ passing through $(\infty,\infty)$ and both base points $(1,0)$, $(0,1)$ of $\phi^{-1}$, and by $C\de\phi^{-1}_{*}\cc$ its proper transform on $X$. 
	
	It would be sometimes convenient to work with explicit equations of $S$. To write them, denote by $\hat{p}_{j}=t^{d_{j}-1}p_{j}(t^{-1})$, $j\in \{1,2\}$ the polynomial obtained from $p_{j}$ by writing its coefficients in the opposite order. Note that $\hat{p}_{j}(0)=1$ because $p_{j}$ is monic. Now $S\subseteq \C^{4}=\Spec\C[x_1,x_2,y_1,y_2]$ is given by
	\begin{equation}\label{eq:S}
	y_1x_1^{d_1}=x_2-\hat{p}_1(x_1),\quad y_2x_2^{d_2}=x_1-\hat{p}_2(x_2).
	\end{equation}
	Note that $S$ is in fact a surface in $\C^{3}$ because, say, $x_1$ can be computed from the second equation.
\end{notation}

\begin{lem}[{$\C^{1}$'s in $S$, cf.\ \cite[Lemma 8.1]{Kal-Ku1_AL-theory}}]\label{lem:C1}
	A curve $\Gamma\subseteq X$ satisfies $\Gamma\cap S\cong \C^{1}$ if and only if one of the following holds:
	\begin{enumerate}
		\item\label{item:A} $\Gamma=A_{j}$ for some $j\in \{1,2\}$. Then $\Gamma^{2}=-1$ and $\Gamma\cdot D=\Gamma\cdot T_j=1$.
		\item\label{item:L} $\Gamma=L_{j,1}$ and $p_{j}=t^{d_{j}-1}$ for some $j\in \{1,2\}$. Then $\Gamma^{2}=-d_{j}$, $\Gamma\cdot D=\Gamma\cdot L_{3-j,\infty}=1$ and $\Gamma\cdot A_{3-j}=1$. 
		\item\label{item:C} $\Gamma=C$ and $p_{j}=1$ or $t^{d_{j}-2}(t-1)$ for both $j\in \{1,2\}$. Then $\Gamma^{2}=2-d_{1}-d_{2}$, $\Gamma\cdot D=2$, $\Gamma$ meets $D$ at $L_{1,\infty}\cap L_{2,\infty}$ and $\Gamma\cdot A_{j}=1$ for $j\in \{1,2\}$.
	\end{enumerate}
	Moreover, $A_{1}\cdot A_{2}=0$, $L_{1,1}\cdot L_{2,1}=1$ and $A_{j}\cdot L_{j,1}=C\cdot L_{j,1}=0$ for $j\in \{1,2\}$.
\end{lem}
\begin{proof}
	Clearly, $A_{j}\cap S\cong \C^{1}$ and $A_{j}^{2}=-1$ as in \ref{item:A}. We will now check when $L_{j,1}\cap S \cong \C^{1}$. Because $\ll_{j,1}$ meets $\phi_{*}D$ transversally on $\ll_{3-j,0}$ and $\ll_{3-j,\infty}$, we have  $L_{j,1}\cdot(D+A_{3-j})=2$ and $L_{j,1}$ meets $L_{3-j,\infty}$, so $L_{j,1}\cap S\cong \C^{1}$ if and only if $L_{j,1}$ meets $A_{3-j}$, that is, all $d_{j}$ blowups over $\ll_{j,1}\cap \ll_{3-j,0}$ touch the image of $L_{j,1}$ (hence $L_{j,1}^{2}=\ll_{j,1}^{2}-d_{j}=-d_{j}$). The proper transform of $\ll_{j,1}$ after the first blowup meets the exceptional curve at the point with coordinate $(x_{3-j}-1)/x_{j}=0$ (see Construction \ref{constr}), so the next blowup, if occurs, touches it if and only if the coefficient near $t^{d_{j}-2}$ in $p_{j}$ is zero. Because these two curves meet transversally, the infinitely near points on the next proper transforms will have coordinate zero, too. Following Construction \ref{constr} we conclude that $L_{j,1}\cap S\cong \C^{1}$ if and only if $p_{j}=t^{d_{j}-1}$, as in \ref{item:L}.

	To check when $C\cap S\cong \C^{1}$, we argue in a similar way. We have $C\cdot(D+A_{1}+A_{2})=4$ and $C$ meets $D$ at $L_{1,\infty}\cap L_{2,\infty}$, so $C\cap S\cong \C^{1}$ if and only if all the blowups in the decomposition of $\phi$ touch the image of $C$. After the first blowup over $\ll_{j,1}\cap \ll_{3-j,0}$ the proper transform of $\cc$ meets the exceptional curve at a point of coordinate $(x_{3-j}-1)/x_{j}=1$, so all the remaining blowups touch it if and only if $p_{j}=1$ (i.e.\ there are no more blowups) or $p_{j}=t^{d_{j}-1}-t^{d_{j}-2}$, since as before the next infinitely near points will have coordinate zero on the exceptional curves. In this case, $C^{2}=\cc^{2}-(d_{1}+d_{2})=2-d_{1}-d_{2}$, as in \ref{item:C}.
	\smallskip
	
	It remains to show that if $\Gamma\cap S\cong \C^{1}$ for some curve $\Gamma\subseteq X$ then $\Gamma$ equals $A_{j}$, $L_{1,j}$ or $C$. Consider first the case when $\Gamma$ is vertical for the $\P^{1}$-fibration induced by $|L_{j,\infty}|$, so $\Gamma\subseteq \phi^{*}\ll_{j,t}$ for some $t\in \C^{1}$ because $\Gamma\not\subseteq D$. If $t\not\in\{0,1\}$ then $\Gamma\cap S\cong \C^{*}$, which is false. If $t=0$ then $\Gamma=A_{j}$ because $\phi^{*}\ll_{j,0}=L_{j,0}+T_{j}+A_{j}$. If $t=1$ then $\Gamma=L_{1,j}$ or $A_{3-j}$ because $(\phi^{*}\ll_{j,1})\redd=L_{j,1}+T_{3-j}+A_{3-j}$.
	
	Consider now the case when $\Gamma$ is horizontal for both $\P^{1}$-fibrations induced by $|L_{j,\infty}|$, $j\in \{1,2\}$. Then $\Gamma$ meets both $L_{j,\infty}$, so the point $\Gamma\cap D$ equals $L_{1,\infty}\cap L_{2,\infty}$. Since $\Gamma$ is disjoint from $L_{1,0}+L_{2,0}=[1,1]$, the linear system of the latter induces a $\P^{1}$-fibration for which $\Gamma$ is vertical. The curves $L_{j,\infty}$ are $1$-sections for this $\P^{1}$-fibration, so $\Gamma\cdot L_{j,\infty}=1$. It follows that $\phi(\Gamma)\subseteq \P^{1}\times \P^{1}$ is of type $(1,1)$, passes through $(\infty,\infty)$, and because $\Gamma\cdot L_{j,0}=0$, $\phi(\Gamma)$ passes through the base point of $\phi^{-1}$ on $\ll_{j,0}$, that is, through $(0,1)$ for $j=1$ and $(1,0)$ for $j=2$. Thus $\phi(\Gamma)=\cc$, as claimed.
\end{proof}

\begin{lem} \label{lem:C*_lines}
	For any $\C^{*}$-fibration of $S$ one can find on $S$ two disjoint, vertical curves isomorphic to $\C^{1}$, whose closures $\Gamma_{1},\Gamma_{2}\subseteq X$ satisfy $\Gamma_{i}^{2}\geq -1$, with strict inequality if $\Gamma_{i}\cdot D\geq 2$.
\end{lem}
\begin{proof}
	Fix a $\C^{*}$-fibration $f$ of $S$. Let $\alpha_{1}\colon X''\to X$ be the minimal resolution of the base points of $f$ on $X$, and let $\alpha_{2}\colon X''\to X'$ be the contraction of all vertical $(-1)$-curves which are superfluous in the subsequent images of $(\alpha_{1}^{*}D)\redd$. Put $\alpha=\alpha_{2}\circ\alpha_{1}^{-1}$ and $D'=((\alpha^{-1})^{*}D)\redd$.
	
	Denote the circular subdivisor of $D$ by $R$. Then $\alpha_{1}^{*}R$ contains a unique circular subdivisor $R''$. Because all fibers are trees, $R''$ contains a horizontal component, say $H$. Suppose $R''-H$ is vertical. Then it is connected and meets $H$ twice, so $H$ is a $2$-section, in fact $H=D''\hor$; a contradiction with Lemma \ref{lem:untwisted}. Therefore, $R''$ contains two $1$-sections. In particular, $R''\hor=D''\hor$, so $\alpha_{1}$ is a composition of inner blowups with respect to the images of $R''$, hence it does not touch $\alpha_{1}^{*}(D-R)$. Because the latter contains no $(-1)$-curves and meets $R''$ in branching components of $D''$, $\alpha_{2}$ does not touch it, either. In other words, $\alpha$ is a composition of inner blowups on the images of $R''$ and does not touch the twigs of $D''$.
	
	In particular, $D'\hor$ consists of two $1$-sections, both contained in $R'=(\alpha_{2})_{*}R''$. 
	
	Consider the case when $D'\vert$ has a $(-1)$-curve, say $\Gamma$, and let $F$ be the fiber containing it. Because $D'$ has no superfluous vertical $(-1)$-curves, $\beta_{D'}(\Gamma)\geq 3$, so $\Gamma\subseteq R'$ meets a maximal twig $T$ of $D'$. Because $\alpha^{-1}$ does not touch $T$, we get $\alpha^{-1}_{*}(T+\Gamma)=T_{j}+L_{j,0}=[(2)_{d_{j}},1]$ for some $j\in \{1,2\}$. Since  $\beta_{F\redd}(\Gamma)\leq 2$, $\Gamma$ meets $R'\hor$. The latter consists of $1$-sections, so $\Gamma$ has multiplicity $1$ in $F$, and thus $\Gamma$ is a tip of $F$. Because $F$ contracts to a $0$-curve, it follows that $F=[1,(2)_{d_{j}},1]$. The $(-1)$-curve $\Gamma'=F-T-\Gamma$ satisfies $\Gamma'\cdot D'=1+(F-T-\Gamma)\cdot D'\hor=1$, so $A=\alpha^{-1}_{*}\Gamma'$ satisfies $A\cdot D=A\cdot T_{j}=1$, thus $A=A_{3-j}$ by Lemma \ref{lem:C1}. It follows that $F=\alpha^{*}(A_{j}+T_{j}+L_{3-j,0})$, so our $\P^{1}$-fibration is in fact induced by a pullback of $|L_{3-j,\infty}|$. Therefore, $\alpha=\id$ and the curves $A_{j}\cap S$ and $A_{3-j}\cap S$ are as required.
	
	Consider the case when  $D'\vert$ contains no $(-1)$-curves. Because $R'$ is circular, $R'\hor$ has $2-b_{0}(R'\vert)\in \{0,1,2\}$ nodes. Let $F_r$ be a fiber passing through a node $r\in R'\hor$. Then $(F_{r}\cdot D'\hor)_{r}=2$, so because each connected component of $D'\vert$ meets $D'\hor$, $F_{r}\cap D'=\{r\}$. Because $X'\setminus D'\cong S$ is affine, $F_{r}\cap (X'\setminus D')\cong \C^{1}$. The curve $F_{r}$ meets each $1$-section in $D'\hor$ once, so the above description of $\alpha$ shows that $\alpha_{2}$ touches $(\alpha_{2}^{-1})_{*}F_{r}$ at most once. Therefore,  $(\alpha^{-1}_{*}F_{r})^{2}\geq F_{r}^{2}-1=-1$, and the equality holds if and only if the proper transform of $F_{r}$ is touched by $\alpha_{2}$ and not by $\alpha_1$, in which case $\alpha^{-1}_{*}F_{r}\cdot D=1$. Hence $F_{r}\cap S$ is as in the statement of the lemma. 
	
	Similarly, if $A\subseteq X'$ is a vertical $(-1)$-curve such that $A\cdot D'=1$, then $A\not\subseteq D'$ by assumption, and  $\alpha_{2}$ does not touch $(\alpha_{2}^{-1})_{*}A$, hence $A\cap S$ is as in the statement, and clearly $A$ is disjoint from any $F_{r}$. Therefore, it remains to find $b_{0}(R'\vert)$ disjoint vertical $(-1)$-curves $A\subseteq X'$ such that $A\cdot D'=1$.
	
	Assume that every connected component $V$ of $R'\vert$ is contained in a fiber $F_{V}\neq V$. Let $A_V\subseteq F_V$ be a $(-1)$-curve. Then $A_V\not\subseteq D'$, $A_V\cdot D'\hor \leq (F_V-V)\cdot D'\hor=0$ and $A_V\cdot D'\vert=1$ because $F_V$ is a tree, so $A_V\cdot D'=1$. This ends the proof because  $A_V\cap A_{V'}=\emptyset$ for $V\neq V'$.
	
	Thus we can assume that $R'\vert$ contains a fiber. We have $3\leq \#R'\leq \#D'=\rho(X')$ by Lemma \ref{lem:S0}. Hence $X'$ contains a degenerate fiber $F$. Let $A$ be a $(-1)$-curve in $F$, so $A\not\subseteq D'$. 
	
	Suppose that $F\subseteq D'\vert+A$. Then $A$ is a unique $(-1)$-curve in $F$, so it has multiplicity $\mu\geq 2$ in $F$. Denote our $\C^{*}$-fibration of $S$ by $f$. Because $D'$ contains a fiber, $f(S)\subseteq\C^{1}$, and we can assume that $f(A\cap S)=\{0\}$. Because $\C[S]$ is a UFD, $A|_{S}$ is a divisor of a regular function, say $g$. Then $f/g^{\mu}\in \C[S]^{*}=\C^{*}$, so $f=\lambda g^{\mu}$ for some $\lambda\in \C^{*}$. A general fiber of such $f$ is reducible, a contradiction.
	
	Suppose $L\cdot D'\geq 2$ for every $(-1)$-curve $L\subseteq F$. Then $2\leq L\cdot D'\leq b_0(F\cap D')$ because $F$ is a tree, and $b_0(F\cap D')\leq F\cdot D'\hor=2$ because $D'$ is connected. It follows that $A$ is a unique $(-1)$-curve in $F$, and meets two connected components of $F\cap D'$. Because $F$ meets $D'\hor$ in components of multiplicity $1$, it follows that $F$ is a chain meeting $D'$ in tips contained in $D'\vert$. Hence $F\subseteq D'\vert+A$; a contradiction.
	
	Thus we can assume $A\cdot D'=1$. If $b_0(R'\vert)<2$ then the lemma follows, so we can assume $b_0(R'\vert)=2$. Lemma \ref{lem:S0}\ref{item:units} implies that $D'$ contains exactly one fiber, so some connected component $V$ of $R'\vert$ is contained in a degenerate fiber $F$ as above. Recall that $A\subseteq F$ is a $(-1)$-curve such that $A\not\subseteq D'$,  $A\cdot D'=1$ and $F\not\subseteq D'\vert+A$.
	
	For every component $L$ of $F-D'\vert$ we have $L\cdot D'\hor\leq (F'-V)\cdot D'\hor=0$, so since $F$ is a tree, $L\cdot D'=1$ and all such $L$'s are disjoint. If two of them are $(-1)$-curves then the lemma follows. Hence we can assume that $A$ is a unique $(-1)$-curve in $F$. Put $T=0$ if $A$ meets $V$, otherwise denote by $T$ the maximal twig of $D'$ meeting $A$. If $T\neq 0$ then $\alpha^{-1}$ does not touch $A+T$, hence by Lemma \ref{lem:C1}, $A+T=\alpha_{*}(A_{j}+T_{j})=[1,(2)_{d_{j}}]$ for some $j\in \{1,2\}$. Moreover, if $T$ meets a curve $L\not\subseteq D'$ with $L\cdot D'=1$ then Lemma \ref{lem:C1} gives $L=A$. Hence $T\cdot(F\redd-A-V)=0$. In any case, let $\tau$ be the contraction of $A+T$. Then $\tau$ does not touch $F\redd-V$, and the fiber $\tau_{*}F$ has a unique $(-1)$-curve, contained in $\tau_{*}V$. Perform further  inner contractions with respect to the images of the chain $\tau_{*}V$ until one of the components of the image of $F$ meeting a $1$-section in the image of $D'\hor$ becomes a $(-1)$-curve. Then the latter has multiplicity $1$ in the image of $F$, so  it is not its unique $(-1)$-curve. It follows that the image of $F$ is of type $[1,1]$ and both its components are in the image of $V$. Hence $F\subseteq D'\vert+A$; a contradiction. 	
\end{proof}

In the following, we fix coordinates on $S=S_{p_{1},p_{2}}$, treating it as a closed subset of $\C^{4}$ given by \eqref{eq:S}.

\begin{lem}[Open tori in $S$]\label{lem:tori}
	Fix an open subset $U\subseteq S$ and a ring homomorphism  $\sigma\colon \C[x_1,x_2,y_1,y_2]\to \C[v_{1}^{\pm 1},v_{2}^{\pm 1}]$ as in \ref{item:AA}--\ref{item:LC} below. Then $\sigma$ induces an isomorphism $U\cong \C^{*}\times \C^{*}$.
	\begin{enumerate}
		\item\label{item:AA} $U=S\setminus (A_1+A_2)$ and 
		$\sigma(x_{j})=v_{j}$, $\sigma(y_{j})= (v_{3-j}-\hat{p}_{j}(v_{j}))v_j^{-d_{j}}$ for $j\in \{1,2\}$;
		\item\label{item:AL} In case $p_{3-j}=1$ for some $j\in \{1,2\}$:
		 $U=S\setminus (A_{j}+L_{j,1})$ and
		 \begin{equation*}
		 \sigma(x_{j})=v_{j},\quad
		 \sigma(x_{3-j})=(v_{j}-1)v_{3-j}^{-1},\quad
		 \sigma(y_{j})=(v_{j}-1-\hat{p}_{j}(v_{j})v_{3-j})\cdot (v_{j}^{d_j}v_{3-j})^{-1},\quad
		 \sigma(y_{3-j})=v_{3-j};
		 \end{equation*}
		\item\label{item:LC} In case $p_{1}=p_{2}=1$:
		 $U=S\setminus (C+L_{j,1})$ for some $j\in \{1,2\}$ and
		 \begin{equation*}
		 \sigma(x_{j})=-(v_{2}+1)v_{1}^{-1},\quad
		 \sigma(x_{3-j})=-(v_{1}+v_{2}+1)\cdot (v_{1}v_{2})^{-1},\quad
		 \sigma(y_{j})=(v_{1}+1)v_{2}^{-1},\quad
		 \sigma(y_{3-j})=v_{2}.
		 \end{equation*}
	\end{enumerate}
The volume form $\tfrac{dv_{1}}{v_{1}}\wedge \tfrac{dv_{2}}{v_{2}}$ on $U$ extends, up to a sign, to a volume form $\tfrac{dx_{1}}{x_{1}}\wedge \tfrac{dx_{2}}{x_{2}}$.
\end{lem}
\begin{proof}
	We check that by a direct computation that $\sigma$ is zero on the ideal of $S$, so it gives a morphism $\C^{*}\times \C^{*}\to S$. It has a rational inverse, given by
	\begin{equation*}
	\mbox{\ref{item:AA} }v_{j}\mapsto x_{j},\ j\in \{1,2\};\quad 
	\mbox{\ref{item:AL} }v_{j}\mapsto x_{j}, v_{3-j}\mapsto y_{3-j};\quad
	\mbox{\ref{item:LC} }v_{1}\mapsto y_{1}y_{2}-1, v_{2}\mapsto y_{3-j}.
	\end{equation*}
	This inverse is regular on the complement of:
	\begin{equation*}
	\mbox{\ref{item:AA} }\{x_{1}x_{2}=0\}=(A_{1}+A_{2})|_{S};\quad 
	\mbox{\ref{item:AL} }\{x_{j}y_{3-j}=0\}=(A_{j}+L_{j,1})|_{S};\quad
	\mbox{\ref{item:LC} }\{y_{3-j}(1-y_{1}y_{2})=0\}=(L_{j,1}+C)|_{S};
	\end{equation*}	
	and sends $\tfrac{dv_{1}}{v_{1}}\wedge \tfrac{dv_{2}}{v_{2}}$ to $\pm \tfrac{dx_{1}}{x_{1}}\wedge \tfrac{dx_{2}}{x_{2}}$. This proves the lemma. 
\end{proof}
\begin{rem*}
	The formulas $\{y_{3-j}=0\}=L_{j,1}|_{S}$, $\{y_{1}y_{2}=1\}=C|_{S}$ used in the above proof do \emph{not} hold for general $p_{1},p_{2}$, but they do hold under the assumptions of \ref{item:AL} and \ref{item:LC}, respectively.
\end{rem*}

\begin{prop}[{Complete algebraic vector fields on $S_{p_{1},p_{2}}$, cf.\ \cite[Theorem 19]{Kal-Ku1_AL-theory}}]\label{prop:fibrations}\ 
	\begin{enumerate}
		\item\label{item:fibr} Every $\C^{*}$-fibration of $S$ restricts to a $\C^{*}$-fibration of one of the open tori $U\subseteq S$ from Lemma \ref{lem:tori}
		\item\label{item:extend} For any $U\subseteq S$ from Lemma \ref{lem:tori} and any $a_{j}\geq d_{j}$, $j\in \{1,2\}$, the complete vector field $\nu_{a_{1},a_{2}}=v_{1}^{a_{1}}v_{2}^{a_{2}}(a_{2}v_{1}\tfrac{\d}{\d v_{1}}-a_{1}v_{2}\tfrac{\d}{\d v_{2}})$ on $U$ extends to a  complete vector field on $S$.
		\item\label{item:vf} Every complete algebraic vector field on $S$ restricts to a complete vector field on some torus $U\subseteq S$ from Lemma \ref{lem:tori}. In particular, it has divergence zero with respect to the volume form $\tfrac{dx_1}{x_1}\wedge \tfrac{dx_2}{x_2}$.
		\item\label{item:AAut} For each open torus $U\subseteq S$ from Lemma \ref{lem:tori}, the group $\AAut(S)\cap \AAut(U)$ acts on $U$ $m$-transitively for every $m$.
	\end{enumerate}
\end{prop}
\begin{proof}
	\ref{item:fibr} Fix a $\C^{*}$-fibration $p$ of $S$ and let $\Gamma_{1},\Gamma_{2}\subseteq X$ be as in Lemma \ref{lem:C*_lines}. Then $p$ restricts to a $\C^{*}$-fibration of $U=S\setminus (\Gamma_{1}+\Gamma_{2})$. We claim that $U$ is as in Lemma \ref{lem:tori}. Because for $i\in \{1,2\}$ $\Gamma_{i}\cap S\cong \C^{1}$, $\Gamma_{i}$ is one of the curves from Lemma \ref{lem:C1}. If $\{\Gamma_{1},\Gamma_{2}\}=\{A_{1},A_{2}\}$ then $U$ is as in Lemma \ref{lem:tori}\ref{item:AA}. Assume that $\Gamma_{1}=C$. Because $C\cdot D=2$, by Lemma \ref{lem:C*_lines} $0\leq \Gamma_{1}^{2}=2-d_1-d_2$, so $d_1=d_2=1$, i.e.\ $p_1=p_2=1$. In this case, $C\cap S$ meets $A_j \cap S$ for both $j\in \{1,2\}$, so $\Gamma_{2}=L_{1,j}$ for some $j\in \{1,2\}$ and therefore $U$ is as in Lemma \ref{lem:tori}\ref{item:LC}. We are left with the case when $\Gamma_{1}=L_{1,j}$ for some $j\in \{1,2\}$ and $\Gamma_{2}\neq C$. As before, the condition $\Gamma_{1}^{2}\geq -1$ from Lemma \ref{lem:C*_lines} implies by Lemma \ref{lem:C1}\ref{item:L} that $d_{j}=1$, i.\ e.\ $p_{j}=1$.  In this case, $L_{1,j}\cap S$ meets $A_{3-j}\cap S$ and $L_{3-j,1}\cap S$, so $\Gamma_{2}=A_{j}$ or $C$. The second case is excluded by assumption, so the first case holds and $U$ is as in Lemma \ref{lem:tori}\ref{item:AL}.
	
	\ref{item:extend} follows from a direct computation, see \cite[Corollary 2]{Andersen-tori}.
	
	\ref{item:vf} Part \ref{item:extend} implies that complete algebraic vector fields on $S$ do not share a rational first integral. Hence by \cite[Theorem 1.6]{KKL_avf}, for any complete algebraic vector field $\xi$ on $S$ there is a $\C^{1}$- or $\C^{*}$-fibration of $S$ preserved by the flow of $\xi$. Because $\kappa(S)=0$, $S$ admits no $\C^{1}$-fibrations, hence $\xi$ preserves one of the $\C^{*}$-fibrations from \ref{item:fibr}. In each case, the complement of $U$ is a union of vertical curves isomorphic to $\C^{1}$, so the flow of $\xi$ cannot send them to non-degenerate fibers in $U$. It follows that $\xi$ restricts to a complete vector field on $U$. In particular, by \cite[Corollary 1]{Andersen-tori} it preserves the volume form $\tfrac{dv_{1}}{v_{1}}\wedge \tfrac{dv_{2}}{v_{2}}$, where $v_{1},v_{2}$ are coordinates on $U\cong \C^{*}\times \C^{*}$. Lemma \ref{lem:tori} implies that $\xi$ preserves $\tfrac{dx_{1}}{x_{1}}\wedge \tfrac{dx_{2}}{x_{2}}$, too.
	
	\ref{item:AAut} Choose integers $a_{ij}$, $i,j\in \{1,2\}$ such that $a_{ij}\geq d_{j}$ and $\det [a_{ij}]=1$, and put $\sigma_{i}=\nu_{a_{i1},a_{i2}}$, $f_{i}=v_{1}^{a_{i1}}v_{2}^{a_{i2}}\in\ker \sigma_{i}\setminus \ker \sigma_{3-i}$. Now the result follows from  \cite[Propostion 8.9]{Kal-Ku1_AL-theory}. Indeed, although that result is formulated in loc.\ cit.\ for $G=\AAut(U)$, the proof actually shows $m$-transitivity for the group generated by the elements of the flows of $q(f_{i})\cdot \sigma_{i}$ for $q\in \C[t]$.
\end{proof}

\begin{proof}[Proof of Theorem \ref{thm:AAut}]Part \ref{item:Aut_thm} was shown in Corollary \ref{cor:st}\ref{item:Aut}. For \ref{item:AAut_thm}, note that the open subset $S\setminus Z$ of $S$ is precisely the union of all $U\subseteq S$ from Lemma \ref{lem:tori}. By Proposition \ref{prop:fibrations}\ref{item:vf}, $S\setminus Z$ is fixed by $\AAut(S)$. On the other hand, by Proposition \ref{prop:fibrations}\ref{item:AAut} each $U\subseteq S$ from Lemma \ref{lem:tori} is contained in a single orbit of $\AAut(S)$. Because every two such $U$, being open and dense in $S$, have nonempty intersection, they are all contained in the same orbit of $\AAut(S)$, which therefore equals $S\setminus Z$. Proposition \ref{prop:fibrations}\ref{item:AAut} implies that $\AAut(S)$ acts on this orbit $m$-transitively for every $m$, as claimed.
\end{proof}

\begin{rem}[Density property for $S_{p_{1},p_{2}}$]
	In view of Theorem \ref{thm:AAut}\ref{item:AAut_thm}, it is natural to ask for an even stronger property than $m$-transitivity of $\AAut(S)$, namely for the \emph{algebraic density property} (ADP), see \cite{Kal-Ku1_AL-theory}. Recall that a smooth affine variety $X$ has ADP if the Lie algebra $\operatorname{VF_{alg}}(X)$ of algebraic vector fields on $X$ coincides with the Lie algebra $\operatorname{Lie_{alg}}(X)$ generated by complete ones. Since in our case all vector fields preserve $Z$, one should rather ask if the \emph{relative} ADP holds, see \cite{KLM_relative-density}, namely, if there exists $l\geq 0$ such that $I^{l}\operatorname{VF_{alg}}(S)\subseteq \operatorname{Lie_{alg}}(S)$, where $I$ is the ideal of $Z$. This, however, is not true since by Lemma \ref{lem:tori}\ref{item:vf} all elements of $\operatorname{Lie_{alg}}(S)$ preserve the volume form $\omega=\tfrac{dx_1}{x_1}\wedge \tfrac{dx_2}{x_2}$. Nonetheless, one could ask for the relative \emph{volume} density property, defined as follows (see \cite{KK_avdp}). Let $X$ be a normal variety with a volume form $\omega$, and let $Y\subseteq X$ be a closed subset containing $\Sing X$. Denote by $\operatorname{VF_{alg}^{\omega}}(X,Y)$ the Lie algebra of those algebraic vector fields on $X$ which vanish on $Y$ and preserve the volume form $\omega$; and by  $\operatorname{Lie_{alg}^{\omega}}(X,Y)$ the Lie algebra generated by complete vector fields in $\operatorname{VF_{alg}^{\omega}}(X,Y)$. We say that $X$ has an algebraic volume density property (AVDP) relative to $Y$ if there exists $l\geq 0$ such that $I^{l}\operatorname{VF_{alg}^{\omega}}(X,Y)\subseteq \operatorname{Lie_{alg}^{\omega}}(X,Y)$, where $I$ is the ideal of $Y$. In this setting, one could ask the following:
\begin{question}
	Let $Z\subseteq S_{p_{1},p_{2}}$ be as in Theorem \ref{thm:AAut}\ref{item:AAut_thm}. Does $S_{p_{1},p_{2}}$ have AVDP relative to $Z$?
\end{question}
For $p_{1}=p_{2}=1$ (so $Z=\emptyset$), the answer to this question is positive \cite[Theorem 6]{KK_avdp}.
\end{rem}

\bibliographystyle{amsalpha}
\bibliography{bibl2018}

\end{document}